\documentclass[11pt,reqno, a4paper]{amsart}

\usepackage{marginnote}

\usepackage[utf8]{inputenc}
\usepackage{amsmath,textcomp,  microtype, lmodern, mathrsfs,csquotes,fbb}
\usepackage{todonotes}
\usepackage{amsmath}
\usepackage{amsfonts}
\usepackage{mathrsfs} 
\usepackage{dsfont}
\usepackage{amssymb}
\usepackage{amsthm}
\usepackage{hyperref}
\usepackage{graphicx}
\usepackage[noadjust]{cite}
\usepackage{caption}
\usepackage{subfig}
\usepackage{esint}
\usepackage{dutchcal}
\usepackage{bbm}
\usepackage{tikz}
\usetikzlibrary{decorations.pathreplacing}
\usepackage[T1]{fontenc}
\usepackage{mathtools}
\usepackage{pdfpages}
\usepackage{hyphenat}

\usepackage[top=2.92cm, bottom=2.92cm, left=2.5cm, right=2.5cm, headsep=0.2in]{geometry}
\usepackage{fancyhdr}

\usepackage{fancyhdr}

\DeclareMathOperator{\Vol}{Vol}

\theoremstyle{plain}
\newtheorem{theorem}{Theorem}[section]

\newtheorem{Lemma}[theorem]{Lemma}

\newtheorem{remark}[theorem]{Remark}
\newtheorem{note}[theorem]{Notation}

\newtheorem{proposition}[theorem]{Proposition}

\newcommand{\RR} {\mathbb R}

\newcommand{\CC} {\mathbb C}
\newcommand{\DD} {\mathbb D}

\newcommand{\ZZ} {\mathbb Z}
\newcommand{\NN} {\mathbb N}
\newcommand{\HH} {\mathbb H}

\newcommand{\DDD}{\mathcal{D}}
\newcommand{\cD}{\mathcal{D}}

\newcommand{\pa} {\partial}

\newcommand{\beq} {\begin{equation}}
\newcommand{\eeq} {\end{equation}}

\newcommand{\tnorm}[1]{\left|\!\left|\!\left| #1 \right|\!\right|\!\right|}

\newcommand{\supp} {\operatorname{supp}}

\newcommand{\InjRad}{\operatorname{InjRad}}

\renewcommand{\Re} {\operatorname{Re}}
\newcommand{\Op} {\operatorname{Op}}

\numberwithin{equation}{section}

\newcommand{\SSa}[1]{{\color{orange} {\bf (SS: #1)}}}

\begin{document}
\title[Ergodicity on large surfaces]{
Quantum Ergodicity on large hyperbolic surfaces for local and pseudolocal operators
}

\author{Nalini Anantharaman}
\address{Coll\`{e}ge de France, 11 place Marcelin Berthelot, 75005 Paris / IRMA, 7 rue Ren\'{e} Descartes, 67084 Strasbourg, France}
\email{nalini.anantharaman@college-de-france.fr}
\author{Soumyajit Saha}
\address{Institut de Recherche Math\'{e}matique Avanc\'{e}e (IRMA), 7 rue Ren\'{e} Descartes, 67084 Strasbourg, France}
\email{soumyajit.saha@unistra.fr}

\begin{abstract}
We prove a quantum ergodicity theorem for sequences of closed hyperbolic surfaces converging to the Poincar\'e disc in the Benjamini–Schramm sense. Assuming a uniform lower bound on injectivity radius and spectral gap, we establish vanishing of quantum variance on fixed spectral windows for a class of observables that contains differential operators and finite-propagation smooth operators. This generalises a result of Le Masson and Sahlsten from scalar observables to both local and ``pseudolocal'' operator settings.
\end{abstract}

\maketitle

\tableofcontents
\allowdisplaybreaks

\section{Introduction}

 Let $(X, g)$ be any compact Riemannian surface.  Consider the eigenequation
\begin{equation*}\label{eqtn: Eigenfunction equation}
    -\Delta \psi = \nu \psi, %\hspace{2mm} i=1, 2, \cdots,
\end{equation*}
where $\Delta$ is the Laplace-Beltrami operator. %given by (using the Einstein summation convention) 
%$$
%\Delta f = \frac{1}{\sqrt{|g|}}\pa_i \left( \sqrt{|g|}g^{ij}\pa_j f\right),
%$$
%where $|g|$ is the determinant of the metric tensor %$g_{ij}$. 
Recall that  
the Laplacian $-\Delta$ has a discrete spectrum 
$$ 
0\leq\nu_1 \leq \dots \leq \nu_k \leq \dots \nearrow \infty,
$$
repeated with multiplicity with corresponding eigenfunctions $\psi_k$, which form a complete orthonormal basis of $L^2(X)$. We are interested in studying the ``ergodicity'' of eigenfunctions on the large-scale geometric limit for eigenvalues that stay bounded in a fixed interval, a point of view introduced by \cite{AnantharamanLeMasson2015} for large regular graphs and by \cite{MassonSahlsten17} for large hyperbolic surfaces. Before introducing the general notion, we begin by observing the following ``toy'' case. 

Consider the Dirichlet Laplacian on a fixed interval $[0,L]$. The corresponding eigenvalues and normalised eigenfunctions are :
\[
\nu_{k,L}= \left(\frac{k\pi}{L}\right)^2,\hspace{20pt}\psi_{k,L}(x)=\sqrt{\frac{2}{L}}\sin{ \left(\frac{k\pi x}{L}\right)}; \hspace{15pt} k\geq 1.
\]
It is easy to check that for any $a(x)\in C^1([0,L])$,
\[\int_0^L a(x)|\psi_{k,L}|^2dx=\frac{1}{L}\int_0^La(x)dx-\frac{1}{L}\int_0^La(x) \cos{ \left(\frac{2k\pi x}{L}\right)},\]
and
\[
\frac{1}{L}\int_0^La(x) \cos{ \left(\frac{2k\pi x}{L}\right)}\to 0 ~~~~~\text{ as }~~~~~ k\to \infty
\]
if $a$ is continuous. The limit $k\to \infty$ corresponds to the high-frequency (or small wavelength) limit. So, the sequence of probability measures $|\psi_{k,L}|^2dx$ converges weakly to the uniform measure $\frac{1}{L}dx$. In other words, the eigenfunctions equidistribute in high-energy limits in the position space.

\subsection{Classical Quantum Ergodicity theorem} Let us recall the Quantum Ergodicity theorem in its original form, where the frequency $\nu$ goes to infinity.
\begin{theorem}[Quantum Ergodicity Theorem, \cite{Snilerman74, CDV85, Zelditch87}] \label{t:QE}
    Let $X$ be a compact Riemannian manifold whose metric is normalised so that $\Vol(X)=1$, and let $(\nu_j, \psi_j)$ be its orthonormal eigendata corresponding to the Laplace-Beltrami operator $-\Delta_X$. Assume that the geodesic flow on $X$ is ergodic with respect to the Liouville measure. Then for a continuous function $a$ on $X$, we have 
    \[
    \frac{1}{N(\nu)}\sum_{j: \nu_j\leq \nu}\left| \langle \psi_j, a\psi_j \rangle_{L^2(X)}-\int a(x)d\Vol(x)\right|\to 0 \qquad \text{as }~~ \nu\to \infty,
    \]
    where $N(\nu)=\#\{j: \nu_j\leq \nu\}$. 
\end{theorem}

\begin{remark}
   The above averaging phenomenon implies that there is a density 1 subsequence\footnote{$S\subset \NN$ is of density 1 if $\displaystyle \lim_{j\to \infty}\frac{\#\{S\cap[0,j]\}}{j}=1$.} $S\subset \NN$ such that 
   \[
    \langle \psi_j, a\psi_j \rangle= \int a(x)|\psi_j(x)|^2d\Vol(x)\to \int a(x)d\Vol(x), \qquad \text{as }~~~ j\to \infty, j\in S.
   \]
   In other words, the sequence of measures $\{|\psi_j(x)|^2d\Vol(x)\}_{j\in S}$ converges weakly to the uniform measure $d\Vol(x)$.
\end{remark}
\begin{remark}The full statement of the above QE theorem is more general~: it says that the multiplication operator $a$ in the inner product can be replaced with a pseudo-differential operator $A$ of order zero and the $a$ in the integral with the principal symbol $\sigma_A$ and the normalised Liouville measure $d\tilde{L}$. More precisely, 
\[
\frac{1}{N(\nu)}\sum_{j: \nu_j\leq \nu}\left| \langle \psi_j, A\psi_j \rangle_{L^2(X)}-\int_{S^*X} \sigma_Ad\tilde{L}\right|\to 0 \qquad \text{as }~~ \nu\to \infty,
\]
where $S^*X$ denotes the unit cotangent bundle of $X$.

The toy case of the interval $[0, L]$ may be seen as a trivial case of a manifold with boundary, where everything can be computed explicitly. Note, however, that in the case of the circle $\RR/ L\ZZ$, the ergodicity assumption is not satisfied for the geodesic flow. It is easy to show that equidistribution holds for test functions $a(x)$ but not for general pseudodifferential operators. 
\end{remark}

\subsection{1-D model on large intervals} Now, we return to our example. Changing our perspective, instead of fixing the interval $[0, L]$ and letting $k\to \infty$, we allow the interval to grow. %A priori, we can think of two ways of doing it. 
%\begin{itemize}
%     \item ({\em Fixing $k$, letting $L\to \infty$}) Choosing any compactly supported $a$ with $\supp(a)\subset [0, R]$ and for $L\geq R$,
 %   \[\int_0^L a(x)|\psi_{k,L}|^2dx \leq \frac{2}{L}\|a\|_{L^1}\to 0 ~~~~~\text{ as }~~~~~ L\to \infty,\] 
  %   while 
   %% \frac{1}{L}\int_0^La(x)dx\to 0 ~~~~~\text{ as }~~~~~ L\to \infty.
     %\]
     %In this regime, the ``equidistribution'' becomes trivial since the convergence is due to the spreading of mass over a growing region. There is no
     To get a similar oscillatory averaging phenomenon that one sees during the regime of fixed interval and high-energy limit, one fixes a spectral window $\nu_k\in I$ and lets $L\to \infty$. Let $I=[a,b]$ and $N(L, I)$ denote the number of eigenvalues contained in $I$. Remark that $\nu_{k, L}\in I$ amounts
     to $(k\pi)^2\in [aL^2, bL^2]$, from which it can be shown that 
     \[
     N(L, I)\sim C_IL, ~~~~~\text{ where }~~~~ C_I \text{ depends on } I.
     \]
    Consider a sequence of uniformly bounded functions $a_L$ ($\|a_L\|_\infty\leq M$ for every $L$), and define the quantum variance 
    \[
    \frac{1}{N(L, I)}\sum_{\nu_{k, L}\in I} \left|  \langle \psi_{k,L}, a_L\,\psi_{k,L}\rangle -\frac{1}{L}\int_0^L a_Ldx  \right|^2= \frac{1}{N(L, I)}\frac{1}{L^2}\sum_{\nu_{k, L}\in I} \left| \int_0^L a_L (x) \cos{\left(\frac{2k\pi x}{L}\right)}dx \right|^2.
    \]
    Considering the Fourier coefficients $c_{k, L}=\int_0^La_L(x)e^{-i\frac{2k\pi x}{L}}dx$, Parseval's identity gives us
    \[\|a_L\|_2^2= \frac{1}{L}\sum_{m\in\ZZ}|c_{m, L}|^2.\]
    Then the variance term simplifies as follows:
    \begin{align*}
        \frac{1}{N(L, I)}\frac{1}{L^2}\sum_{\nu_{k, L}\in I} \left| \int_0^L a_L (x) \cos{\left(\frac{2k\pi x}{L}\right)}dx \right|^2&= \frac{1}{N(L, I)}\frac{1}{L^2}\sum_{\nu_{k, L}\in I} \left| \Re\left(\int_0^L a_L (x) e^{i\frac{2k\pi x}{L}}dx \right)\right|^2\\
        &\leq \frac{1}{N(L, I)}\frac{1}{L^2}\sum_{\nu_{k, L}\in I} \left|\int_0^L a_L (x) e^{i\frac{2k\pi x}{L}}dx \right|^2\\
        &= \frac{1}{N(L, I)}\frac{1}{L^2} \sum_{\nu_{k, L}\in I} |c_{-k, L}|^2 \\
        &\leq \frac{1}{N(L, I)}\frac{1}{L^2} \sum_{m\in \ZZ} |c_{m, L}|^2= \frac{1}{N(L,I)}\frac{1}{L^2} L\|a_L\|_2^2\\
        &\leq \frac{1}{N(L,I)}\frac{1}{L} M^2L=  \frac{M^2}{N(L,I)}\to 0 \text{ as } L\to \infty.
    \end{align*}
  
%\end{itemize}

Studying equidistribution in the large-volume limit with a fixed spectral window is the natural analogue of the high-energy limit on a fixed domain. In fact, for intervals (or more generally, higher-dimensional cubes or tori), the two limits may be seen to be equivalent by writing
$$\int_0^L a_L (x) \frac1L\cos{\left(\frac{2k\pi x}{L}\right)}dx=
\int_0^1 a_L (Ly) \cos{\left(2k\pi y\right)}dy
$$
where $\nu_{k, L} \in [a, b]$ implies that $k$ goes to infinity. The reason for a different treatment of the two limits is that the function $y\mapsto a_L (Ly)$ is now $L$-dependent, and we did not want to assume uniform regularity of this sequence of functions.

For more general manifolds, anyway, the two settings are no longer equivalent~: the study of eigenvalues in a fixed spectral window on a sequence of expanding manifolds is different from the high-energy limit at fixed geometry. 

\subsection{Main result} A typical case of application of Theorem \ref{t:QE} is the case
where $X$ is a compact boundaryless hyperbolic surface, where the geodesic flow is known to be ergodic. In this article, our main focus is also on hyperbolic surfaces, but we propose
to consider the large-scale geometric limit instead of the high-frequency limit. That is, we consider sequences of hyperbolic surfaces converging to the hyperbolic disc under the notion of Benjamini-Schramm convergence. Although originally introduced for graphs in \cite{BS}, it can be extended to the manifold setting \cite{Abert-et-al}. We use the following special case of the definition for our purposes. We say that a sequence of compact hyperbolic surfaces $X_p = \Gamma_p\textbackslash\DD$ \emph{Benjanimi-Schramm converges} to the hyperbolic plane $\DD$ if for any $R>0$, 
\begin{equation}
    \lim_{p\to \infty} \frac{\Vol(\{z\in X_p: \InjRad_{X_p}(z)<R\})}{\Vol(X_p)}=0.
\end{equation}
In other words, given fixed $R>0$, one picks a point $z$ at random on the surface $X_p$ using the normalised measure $\frac{1}{\mu_{g_p}(X_p)}\mu_{g_p}$. The above equation says that the probability for the ball of centre $z$ and radius $R$ to be isometric to a ball in the hyperbolic disc goes to 1 as $p\to \infty$. The \emph{injectivity radius} $\InjRad_{X_p}(z)$ gives the largest $R>0$ such that the ball $B_X(z, R)$ is isometric to a ball of radius $R$ in the hyperbolic disc. We denote $\InjRad(X):= \inf_{z\in X} \InjRad_X(z)$ the minimal injectivity radius. For instance, if $\InjRad(X_p)\to \infty$, then the above convergence is automatically satisfied.

Le Masson and Sahlsten in \cite{MassonSahlsten17} show a similar position-space equidistribution as the 1-$D$ example demonstrated above, but on large hyperbolic surfaces. For any fixed spectral interval and given sequence of uniformly bounded test functions $(a_p)$, most of the eigenfunctions of the Laplacian evaluated on $a_p$ approach the uniform measure. More precisely, on a fixed interval $I\subset (1/4, \infty)$ and given sequence of hyperbolic surfaces $\{X_p\}$ (satisfying certain geometric restrictions) that converge to $\DD$ in the sense of Benjamini and Schramm, 
\[
 \frac{1}{N(X_p, J)} \sum_{\nu_j^{(n)}\in J} \left|\langle\psi_j^{(p)}, a_p\psi_j^{(p)}\rangle-\fint a _pd\Vol_{X_p} \right|^2\to 0
\]
when $p\to 0$, for any uniformly bounded sequence of measurable functions $\{a_p\}$. Here $N(X_p, J)$ is the number of eigenvalues in the interval $J$ counted with multiplicity.

Our aim is to extend this framework from scalar to more general operators. More precisely, instead of testing eigenfunctions against multiplication operators, we allow differential operators and, more generally, operators with kernel supported near the diagonal (which we see as analogous to the pseudo-differential operators in the standard Quantum Ergodicity theorem). 
%For this generality, one must impose additional assumptions on the symbol. More precisely, we assume a vanishing condition in the angular $K$-variable together with a suitable control of its second angular derivative. With these modifications, 
Theorem \ref{thm: main_theorem} recovers the equidistribution statement in \cite{MassonSahlsten17} but also applies to a much broader class of operators (or ``observables'' in the language of quantum mechanics). These include local differential operators and more general ``pseudo-differential'' operators.
In forthcoming work, we expect to be able to use this more general theorem to investigate fine properties of eigenfunctions
on large hyperbolic surfaces.

\subsubsection*{\underline{Framework}:}  Consider a sequence of closed hyperbolic surfaces $\{X_p\}_{p\in \NN}$ with orthonormal eigendata $(\nu_{j,p}, \psi_{j,p})_{j=0}^{+\infty}$ corresponding to the Laplacian $-\Delta_{X_p}$. Consider the following assumptions on $\{X_p\}$:
    \begin{enumerate}
    \item $X_p:=\Gamma_p\backslash\DD$ converges in the sense of Benjamini-Schramm to $\DD$.
    \item The injectivity radii are uniformly bounded below by $l_{\min}$, i.e., $\InjRad(X_p)\geq l_{\min}$ for all $p\in \NN$ for some $l_{\min}>0$.
    \item There is a uniform lower bound of the spectral gap of the Laplacian on $X_p$, i.e. there exists some $\eta$ such that $\nu_{1,p}\geq \eta$ for all $p$. 
\end{enumerate}

\begin{theorem}\label{thm: main_theorem}
Consider a fixed interval $J\subset (1/4, \infty)$, and some fixed $S\in \RR$,  $k\in \NN\cup \{0\}$. Let $\{A_p\}$ be a sequence of differential operators, or operators on $L^2(X_p)$ whose distributional kernel $K_{A_p}\in \cD'(X_{p}\times X_{p}) $ satisfies the finite propagation condition: $\supp K_{A_p}\subset \{(z, w),d(z, w)\leq S\} $
with
 a uniform bound
\begin{equation}\label{ineq: local assumption}
     |A_p u (x)| \leq C || u||_{C^k(B(x, S))}
\end{equation}
for all $u$, all $x\in X_p$, and $S, C, k$ independent of $p$.

Then we have
\[
\frac{1}{N(X_p, J)} \sum_{\nu_{j,p}\in J}\left|\langle \psi_{j, p},  A_p\psi_{j,p} \rangle-\frac{1}{\Vol(X_p)}\int_{X_p\times X_p}
K_{A_p}(x, y)\varphi_{\lambda_{j,p}}(d(x, y))\, d\mu(x) d\mu(y)\right|^2 \to 0 \text{ as }~~ p\to \infty,
\]
where $\nu_{j,p}= \frac{1}{4}+ \lambda_{j,p}^2$, $\varphi_\lambda$ is the hyperbolic spherical function of parameter $\lambda$, and
and $N(X_p, J)$ is the number of eigenvalues in $J$ counted with multiplicity.
\end{theorem}

\begin{remark}
In the paper, we will use that $X_p$ can be written as a quotient of the hyperbolic disc $X_p=\Gamma_p\backslash\mathbb{D}$
and that $K_{A_p}$ comes from a $\Gamma_p$-bi-invariant kernel $\tilde K_{A_p}$ on $\DD$. In our proof we shall denote by $a_{p}(z, \lambda, b)\in C^\infty(\DD\times \RR^+\times B)$ the corresponding $\Gamma_p$-invariant symbol, as described in Subsection \ref{subsec: Fourier analysis on Poincare Disc}.
Then the quantity appearing in the conclusion of the theorem may be expressed in terms of the integral of the symbol $a_{p}(z, \lambda, b)$, as follows~:
$$\int_{\cD_p\times \DD}
\tilde K_{A_P}(z, y)\varphi_{\lambda_{j,p}}(d(z, y))\, d\mu(z) d\mu(y)=\int_{\cD_p \times B}a_{p}(z, b, \lambda_{j,p})e^{\langle z,b \rangle}d\mu(z)db  
$$
where $\cD_p$ is a fundamental domain for the action of $\Gamma_p$ on $\DD$. After normalisation by $\frac{1}{\Vol(X_p)}$ the latter quantity is the same
as $\frac{1}{\Vol(X_p)}\int_{X_p\times X_p}
K_{A_p}(x, y)\varphi_{\lambda_{j,p}}(d(x, y))\, d\mu(x) d\mu(y)$, up to some error which vanishes in the limit $p\longrightarrow +\infty$.

\end{remark}

%\begin{remark}
%In the proof we will, without loss of generality, make a simplifying assumption on the symbol $a^{(m)}$~:  
%\begin{itemize}
       % \item[(A1)] $\int_{D_m\times B}a^{(m)}(z, b,\lambda)e^{\langle z,b \rangle}db=0\,\,\,\,\, \text{for every } \lambda\in\RR^+ \text{and } z\in\HH.$
       % \end{itemize}
        %The reason why we may assume that is that the case where $a^{(m)}$ depends only on the variable $z$ is already treated in \cite{MassonSahlsten17}.
% In other words, with the notation of \S ??, we assume
% $$\int_0^{2\pi}a^{(m)}(g k_\theta, \lambda)d\theta=0$$
% for all $\lambda$ and $g$.
% Note that this implies that
%$$\int_{SX_m}a^{(m)}(z, b,\lambda)e^{\langle z,b \rangle}dzdb =0$$ for all $\lambda$,
% or equivalently $\int_0^{2\pi}a^{(m)}(g, \lambda)\, dg =0$.
 %\end{remark}

%In order to prove the theorem, we shall first deduce a quantitative estimate (i.e., the rate of convergence) of the quantum variance on a fixed surface $X_\Gamma$ in Section \ref{sec: quant estimate} and then combine it with a Weyl-type estimate under the Benjamini-Schramm convergence. 
We begin by recalling some preliminary tools and results that we shall use throughout the article.

\iffalse

\begin{theorem}
Consider a fixed interval $I\subset (1/4, \infty)$ and some fixed $S\in \RR$, let $\{A_m\}$ be a sequence of differential operators or pseudo-differential operators of order zero on $L^2(X_m)$ whose kernel $K_{A_m}: X_{m}\times X_{m}\to \RR$ satisfies the finite propagation condition: $K_{A_m}(z, w)=0$ whenever $d(z, w)>S$ for every $m$.

Denote by $a^{(m)}(z, \lambda, b)\in C^\infty(X_\Gamma\times B\times \RR^+)$ the symbol corresponding to the operator $A_m$ as described in Subsection \ref{subsec: Fourier analysis on Poincare Disc} and define $a^{(m)}_\lambda:= a^{(m)}(\cdot, \lambda, \cdot)$. Assume that, for every $m$:
\begin{itemize}
        \item[(A1)] $\int_0^{2\pi}a^{(m)}_\lambda(h' k_\theta)d\theta=0\,\,\,\,\, \text{for every } \lambda\in\RR^+.$
        \item[(A2)] $\int_{G} \frac{\pa^2}{\pa \theta^2}a^{(m)}(h' k_\theta)dg=0$ {\color{blue}(need to write this properly)}
\end{itemize}
where $(z, b)\in SX_m$ is identified with a point $h\in \Gamma\backslash \mathrm{PSU}(1,1)$. Additionally, assume that $\sup_{\lambda^\in  {I}}\|\pa^2_\theta a^{(m)}_\lambda\|_{L^2(SX_m)}$ is uniformly bounded in $m$,
Then we have
\[
\frac{1}{N(I, X_m)} \sum_{\nu_j^{(m)}\in I}\left|\langle \psi_j,  A_m\psi_j \rangle-\int_{SX_m}a^{(m)}(z,\lambda_j^{(m)}, b)e^{\langle z,b \rangle}dzdb\right|^2 \to 0 \qquad \text{as }~~ m\to \infty,
\]
where $\nu_j^{(m)}= \frac{1}{4}+ (\lambda_j^{(m)})^2$ and $N(I, X_m)$ is the number of eigenvalues in $I$ counted with multiplicity.
\end{theorem}
\fi

\section{Definitions and Preliminary results}\label{sec: Disc and HFT} 
We follow the description as in \cite{Zelditch86} (see also \cite{AnantharamanZelditch06, AnantharamanZelditch12}).
 In the Euclidean space, a pseudo-differential $A$ acts via 
 \[
 Au(x)=\int_\Omega e^{ix\xi}\sigma(x, \xi)\hat{u}(\xi)d\xi,
 \]
where $\hat{u}(\xi)$ is the Fourier transform and $\sigma(x, \xi)$ is the symbol of the operator. So, in order to modify the above expression of $A$ to the Poincaré disc, we need to first define the Fourier transform appropriately. 

Let $\DD=\{z\in \CC: |z|<1\}$ be the Poincaré disc with boundary circle $B$, endowed with the Riemannian metric
\[
ds^2=\frac{4|dz|^2}{(1-|z|^2)^2}.
\]
The group of orientation-preserving isometries can be identified with $G=\mathrm{PSU}(1,1)$ acting by M\"{o}bius transformations $g(z)=\frac{\alpha z+\beta}{\overline{\beta}z+\overline{\alpha}}$, $|\alpha|^2-|\beta|^2=1$. The stabilizer
of \(0\) is \(K \simeq \mathrm{SO}(2)\) and thus we will often identify
\(\mathbb{D}\) with \(\mathrm{SU}(1,1)/K\).

Given $(z,b) \in \DD \times B$, let $\xi(z,b)$ be the unique horocycle through $z$ tangent at $b$, and $\gamma(z,b)$ denote the unique geodesic through $z$ with forward endpoint $b$. The non-Euclidean distance from the origin $0$ to $\xi(z,b)$ is denoted by $\langle z,b \rangle$ (it is negative if $\xi(z,b)$ encloses $0$). Note that 
\[
e^{\langle z,b\rangle}= \frac{1-|z|^2}{|z-b|^2}=P_{\DD}(z, b),
\]
where $P_{\DD}(z,b)$ is the Poisson kernel of the unit disc.

\subsection{Coordinates on the sphere bundle} 
We work both with the hyperbolic disc $\DD$ and with the hyperbolic half-plane model $\HH$. The disc model is convenient for harmonic analysis because the Poisson kernel admits a nice expression. The half-plane model is more convenient for the study of the geodesic and horocyclic flow, because the matricial interpretation of these flows is nicer. Both models are conjugate under the map
$$ z\mapsto \frac{z-i}{z+i}$$
which maps $\HH$ isometrically to $\DD$, sending the base point $i$ to $0$ and the boundary points $0, \infty$ respectively to $-1, 1$.

The unit tangent bundle $S\DD$ of the hyperbolic disc $\DD$ is by definition the manifold of unit vectors in the tangent bundle $T\DD$ with respect to the hyperbolic metric. We now make the following identifications:
\begin{itemize}
    \item $S\DD\equiv \DD\times B$. Here, we identify $(z, b)\in \DD\times B$ where $(z, b)$ is identified with $(z, v)\in S\DD$ where $v\in S_z\DD$ is the unit vector tangent to the unique geodesic through $z$ ending at $b$. In geometric terms, if $(z,b)$ is identified with a unit tangent vector in $S\DD$, then $b$ represents the (forward) limit point of the geodesic generated by $(z,b)$.

    \item $S\DD\equiv \mathrm{PSU}(1,1)$. This comes from the fact that $\mathrm{PSU}(1,1)$ acts freely and transitively on $S\DD$. We identify a group element $g\in \mathrm{PSU}(1,1)$ with a unit tangent vector $(z, v)$ if $g\cdot(0, 1)=(z, v)$. A similar identification can be done with $\mathrm{PSL}(2, \RR)$ when working on the hyperbolic plane $\HH$. We denote $G=\mathrm{PSU}(1,1) \equiv \mathrm{PSL}(2, \RR)$.
\end{itemize}

We use the $KAN$ (or rather here $ANK$) decomposition of $G$, which is nicer to write in $\mathrm{PSL}(2, \RR)$ and admits a dynamical interpretation in terms of geodesic, horocyclic and cyclic flow.
Every element of $g\in \mathrm{PSL}(2, \RR)$ can be expressed in a unique way as a product of three subgroups $A$, $N$ and $K$,
\[
g=a_sn_u k_\theta
\]
where
\[
 \quad a_s=\begin{pmatrix}
e^{s/2} & 0\\[6pt]
0 & e^{-s/2}
\end{pmatrix}, ~~~~~~~~~~~~s\in \RR
\]
\[
n_u=\begin{pmatrix}
1 & u\\[6pt]
0 & 1
\end{pmatrix}, ~~~~~~~~~~~~u\in \RR,
\]
and 
\[ k_\theta=
\begin{pmatrix}
\cos\frac\theta{2} & -\sin\frac\theta{2} \\
\sin\frac\theta{2} & \cos\frac\theta{2}
\end{pmatrix}, ~~~~~~~~~~~~\theta\in \RR/2\pi\ZZ.
\]
In the identification $S\HH \equiv \mathrm{PSL}(2,\mathbb{R})$, the geodesic flow $(g^s)_{s\in \mathbb{R}}$ is given by the right action of the group $A$~: $g \mapsto ga_s$.
The action of the horocycle flow $(h^u)_{u\in\mathbb{R}}$ is defined by the right action of $N$, in other words by $g \mapsto gn_u$.
Finally the ``cyclic'' flow $(r_\theta)_{\theta\in \RR/2\pi\ZZ}$ which rotates a tangent vector of an angle $\theta$ corresponds to the right action of $K$, $g\mapsto g k_\theta$.

\begin{note}
    With appropriate identification, the above actions can also be defined on $\DD$ for the group $G=\mathrm{PSU}(1,1)$. We call the corresponding actions on $S\DD$ as $a_s, n_u$ and $k_\theta$ again. 
\end{note}

%Writing $ G \sim (G/K) \times K \sim\DD \times B $, where $ B $ is the boundary at infinity of $ \DD $, identified with the unit circle $ S^1 $ in the Poincaré disc model. The group $ K $ and the boundary at infinity $ S^1 $ are identified by the map

Fixing $v=(z, b_0)\in \DD$, we have $r_\theta (v)= (z, b(\theta))$ which gives an identification (depending on $z$) between the boundary $B$ and
$S^1=\RR/2\pi \ZZ$. Under this identification,
$$d\theta= e^{\langle z,b\rangle} db= \frac{1-|z|^2}{|z-b|^2}db$$
which turns out to coincide with the Poisson 1-form
\begin{equation}
P_\DD(z,b)db = e^{\langle z,b\rangle} db = \frac{1-|z|^2}{|z-b|^2}db.
\end{equation}
Here $z\mapsto\langle z,b\rangle$ is the opposite of the Busemann function centred at $b$.
It satisfies the identities
\begin{equation}
\langle g\cdot z, g\cdot b\rangle = \langle z,b\rangle + \langle g\cdot o, g\cdot b\rangle,
\end{equation}
and
\begin{equation}
\frac{d}{db} g\cdot b = e^{-\langle g\cdot o,g\cdot b\rangle},
\end{equation}
it follows that
\begin{equation}
P_\DD(g\cdot z,g\cdot b)d(g\cdot b)=P_\DD(z,b)db.
\end{equation}
Haar measure on $G$ is denoted $dg$. In terms of $z,b$ coordinates, it is given by
\begin{equation}
dg = P(z,b)\mathrm{Vol}(dz)db,
\end{equation}
where $\mathrm{Vol}(dz)$ is the hyperbolic area form. Under the identification $G \sim S\DD$, the Haar measure on $G$ is also the same as the Liouville measure on $S\DD$.

\subsection{Fourier analysis on Poincaré disc}\label{subsec: Fourier analysis on Poincare Disc}
Fix $(b, \lambda)$. It is known that the function $z\mapsto e^{(i\lambda + \frac{1}{2})\langle z,b\rangle}$ are generalised eigenfunctions in $\DD$,
\[
-\Delta_\DD e^{(i\lambda + \frac{1}{2})\langle z,b\rangle} = \left(\lambda^2 + \frac{1}{4}\right)e^{(i\lambda + \frac{1}{2})\langle z,b\rangle}.
\]
Helgason defines the non-Euclidean Fourier transform $\mathcal{F}: C_0^\infty(\DD) \to C^\infty(\mathbb{R}^+ \times B)$ by
\[
\mathcal{F}u(\lambda,b)=\widehat{u}(\lambda,b)=\int_{\DD}e^{(\frac{1}{2}-i\lambda)\langle z,b\rangle}u(z)\,d\mu(z).
\]
The hyperbolic Fourier inversion formula is given by
\[
u(z)=\int_{\RR^+}\int_B e^{(\frac{1}{2}+i\lambda)\langle z, b\rangle}\widehat{u}(\lambda, b)\lambda\tanh{(2\pi\lambda)}d\lambda db. 
\]
Given any operator $A:C^\infty(\DD)\to C^\infty(\DD)$ we define its \emph{complete symbol} $a(z,b, \lambda)$ by
\begin{equation}
Ae^{(\frac{1}{2}+i\lambda)\langle z,b\rangle}=a(z,b, \lambda)e^{(\frac{1}{2}+i\lambda)\langle z,b\rangle}.
\end{equation}
By the inversion formula, $Au$ can be represented, for $u\in C_0^\infty(\DD)$, by
\begin{align}\label{eq: Op(a) with inverse Helgason transform}
    Au(z)&=\frac{1}{2\pi}\iint_{\mathbb{R}^+\times B}e^{(\frac{1}{2}+i\lambda)\langle z,b\rangle}a(z,b, \lambda)\widehat{u}(\lambda,b) \,d\lambda\,db\\
    &=\frac{1}{2\pi}\iint_{\RR^+\times B}\int_{\DD} a(z, \lambda, b) e^{(\frac{1}{2}+i\lambda)\langle z,b\rangle}e^{(\frac{1}{2}-i\lambda)\langle w, b\rangle} u(w)\lambda\tanh({2\pi\lambda}) d\mu(w)\, d\lambda\, db.
\end{align}
This gives us the kernel of $A$,
\begin{equation}
    K_A(z, w)= \frac{1}{2\pi}\iint_{\RR^+\times B} a(z, \lambda, b) e^{(\frac{1}{2}+i\lambda)\langle z,b\rangle}e^{(\frac{1}{2}-i\lambda)\langle w, b\rangle}\lambda\tanh({2\pi\lambda})\, db d\lambda
\end{equation}
We will use the notations $K_a$ and $K_A$ interchangeably throughout the paper. With the terminology of \cite{Zelditch86}, \cite{Zelditch87}
we also write $A=\Op(a)$.
Note that in \cite{Zelditch86}, \cite{Zelditch87} those objects are considered in the high-frequency limit, where there is a well-developed theory of pseudodifferential operators. Here, we use the same objects in the large-scale limit. Although we might be tempted to still call them pseudodifferential operators, one should be cautious and note that a nice theory of these operators in the large-scale limit does not exist. We mostly use Hilbert-Schmidt estimates~:

As a consequence of the Plancherel formula (see Proposition 4.1 of \cite{AnantharamanZelditch12}), we have that 
\begin{equation*}
    \iint_{\DD\times\DD}|K_a(z, w)|^2d\mu(z)d\mu(w)=\left\|\Op(a)\right\|_{\mathrm{HS}(\DD)}^2=\iiint_{\DD\times B \times \RR^+}  |a(z, \lambda, b)| ^2 e^{\langle z, b \rangle} \lambda \tanh{(2\pi\lambda)} d\mu(z)\,d\lambda\,db.
\end{equation*}

Let now $\Gamma$ be a discrete subgroup of $G$ acting properly discontinuously on $\DD$ with compact quotient. In other words, $X_\Gamma=\Gamma\backslash G$ is a smooth compact hyperbolic surface.
 If $a$ is $\Gamma$-invariant (that is $a(\gamma z, \lambda, \gamma b)=a(z, \lambda, b)$ for all $\gamma\in\Gamma$), then it commutes with the action of $\Gamma$, in other word we get an operator $\Op_\Gamma(a)$ acting on functions on $X_\Gamma$.
 Its kernel is the $\Gamma$-bi-invariant expression  \begin{equation}\label{e:periodize}
K_a^{\Gamma}(z, w)= \sum_{\gamma\in \Gamma} K_a(z, \gamma w).
\end{equation}
Note, however, that there is a summability issue, which implies that we will only consider such expressions when $K_a$ satisfies a finite propagation condition.

From Section 7 of \cite{AnantharamanZelditch12}, one can express the Hilbert-Schmidt norm of $\Op_\Gamma(a)$ on $X_\Gamma$~:  
\begin{align*}
    \|\Op_\Gamma(a)\|_{\mathrm{HS}(X_\Gamma)}^2&=\operatorname{Tr}(\Op_\Gamma(a)\Op_\Gamma(a)^{\dagger})=\int_{X_\Gamma}\int_{X_\Gamma} |K_a^\Gamma(z, w)|^2\,d\mu(z)\,d\mu(w)\\
    &= \int_\DDD\int_\DDD |\sum_{\gamma\in \Gamma} K_a(z, \gamma\cdot w)|^2\,d\mu(z)\,d\mu(w)
\end{align*}
where $\DDD$ denotes a fundamental domain for the action of $\Gamma$ on $\DD$.

%Although, prima facie, our proof also relies on the averages of the propagated operator $P_tAP_t$, as in the general strategy of \cite{MassonSahlsten17}, the novelty of our work does not lie in the introduction of such operators themselves. %Rather, it lies in the fact that we work with a substantially broader class of observables, including differential operators.

\subsection{Finite propagation and Hilbert-Schmidt norm}
From Lemma 5.1 in \cite{MassonSahlsten17}, we know that if $K_a$ has finite propagation, i.e., $K_a(z, w)=0$ whenever $d(z, w)\geq R$, then we have
\begin{multline}\label{ineq: LMS- PDO HS norm}
\|\Op_\Gamma(a)\|_{\mathrm{HS}(X_\Gamma)}^2 \leq \int_{\DDD}\int_{\mathbb{D}}|K_a(z,w)|^2\, d\mu(z)\,d\mu(w)\\
+ \frac{e^{2R}}{l_{X_\Gamma}}\operatorname{Vol}\{z\in X:\operatorname{InjRad}_X(z)<R\}\,\,\sup_{(z,w)\in \DDD\times \mathbb{\DD}}|K_a(z,w)|^2,
\end{multline}
where $l_{X_\Gamma}$ denotes the length of the shortest closed geodesic on $\Gamma\backslash \mathbb{D}$. 

However, we will also have to consider symbols $a$ for which $\Op(a)$ does not have the finite propagation property.
In this case, we modify the definition of $\Op_\Gamma(a)$ and $K_a^\Gamma$ by performing a truncation as described below.

\begin{Lemma}\label{lem: truncated kernel HS-norm estimate}
Let $X_\Gamma=\Gamma\backslash\DD$, and let $\DDD\subset\DD$ be a
fundamental domain. Let $K:\DD\times\DD\to\RR$ be $\Gamma$-invariant under the
diagonal action. For $r>0$, define the truncated periodised kernel
\[
K^{\Gamma,r}(z,w):=\sum_{\gamma\in\Gamma}K(z,\gamma w)\,\chi\left(\frac{d(z,\gamma w)}{r}\right),
\]
where $\chi\in C^\infty([0,\infty))$ satisfies $0\leq \chi\leq 1$ and
\[
\chi\left(\frac{d(z,\gamma w)}{r}\right)=0
\quad\text{whenever } d(z,\gamma w)>r.
\]
Let $A_r$ be the integral operator on $X_\Gamma$ with kernel $K^{\Gamma,r}$.
Then
\[
\|A_r\|_{\mathrm{HS}}^2=\int_\DDD\int_\DDD |K^{\Gamma,r}(z,w)|^2\,d\mu(z)\,d\mu(w),
\]
and we have the estimate
\[
\|A_r\|_{\mathrm{HS}}^2\leq\int_\DDD\int_{\DD}|K(z,w)|^2 d\mu(z)d\mu(w)+ \frac{e^{2r}}{l_{X_\Gamma}} \mathrm{Vol}\{x\in X:\mathrm{InjRad}_X(x)<r\} \sup_{(z,w)\in \DDD\times\DD}|K(z,w)|^2,
\]
where $l_{X_\Gamma}$ denotes the length of the shortest closed geodesic on $X_\Gamma$.
\end{Lemma}

The argument is similar to Lemma 5.1 in \cite{MassonSahlsten17} with minor modifications. For completeness, we provide a proof below.

\begin{proof}
By definition,
\[
\|A_r\|_{\mathrm{HS}}^2=\int_\DDD\int_\DDD\left| \sum_{\gamma\in\Gamma} K(z,\gamma w)\,\chi\left(\frac{d(z,\gamma w)}{r}\right)\right|^2 d\mu(z)d\mu(w).
\]

Set
\[
\DDD(r):=\{z\in \DDD:\mathrm{InjRad}_{X_\Gamma}(z)\ge r\}, \qquad \DDD(r)^c:=\DDD\setminus D(r).
\]
We split the integral accordingly. If $z\in \DDD(r)$ and $w\in \DDD$, then because
$\chi(d(z,\gamma w)/r)$ vanishes unless $d(z,\gamma w)\le r$,
there can be at most one $\gamma\in\Gamma$ such that
$d(z,\gamma w)\le r$. Hence, the $\Gamma$-sum contains at most one nonzero term and,
since $|\chi|\le 1$,
\[
\left|
\sum_{\gamma}K(z,\gamma w)\,\chi\left(\frac{d(z,\gamma w)}{r}\right)
\right|^2= \sum_\gamma \left|K(z,\gamma w)\,\chi\left(\frac{d(z,\gamma w)}{r}\right)\right|^2\leq \sum_{\gamma:d(z,\gamma w)\leq  r} |K(z,\gamma w)|^2.
\]
Integrating over $w\in \DDD$ and unfolding the $\Gamma$-sum gives
\[
\int_\DDD \sum_{\gamma: d(z,\gamma w)\le r}|K(z,\gamma w)|^2d\mu(w)=
\int_{\DD}|K(z,w')|^2d\mu(w').
\]
Therefore
\[
\int_{\DDD(r)}\int_\DDD|K^{\Gamma,r}(z,w)|^2d\mu(z)d\mu(w)\leq
\int_\DDD\int_{\DD}|K(z,w)|^2d\mu(z)d\mu(w).
\]
If $z\in \DDD(r)^c$, then for each $w\in \DDD$ the set
\[
\{\gamma\in\Gamma: d(z,\gamma w)\leq r\}
\]
contains at most $Ce^{r}/l_{X_\Gamma}$ elements, by a standard lattice-point counting argument in hyperbolic space (see Lemma 5.1 of \cite{MassonSahlsten17} for details).
Applying Cauchy-Schwarz to the $\Gamma$-sum, we obtain
\[
|K^{\Gamma,r}(z,w)|^2\leq\frac{C e^{r}}{l_{X_\Gamma}}\sum_{\gamma: d(z,\gamma w)\leq r}|K(z,\gamma w)|^2.
\]
Integrating over $w\in \DDD$ gives
\[
\int_\DDD|K^{\Gamma,r}(z,w)|^2d\mu(w)\leq\frac{C e^{r}}{l_{X_\Gamma}}\int_{\DD}|K(z,w)|^2d\mu(w).
\]
Since
\[
\int_{\DD}|K(z,w)|^2d\mu(w)\leq\sup_{(z,w)\in \DDD\times\DD}|K(z,w)|^2\Vol(B_{\DD}(z,r)),
\]
and $\Vol(B_{\DD}(z,r))\leq e^{r}$, we obtain
\[
\int_{\DDD(r)^c}\int_\DDD|K^{\Gamma,r}(z,w)|^2d\mu(z)d\mu(w)\leq\frac{e^{2r}}{l_{X_\Gamma}}\Vol\{x\in X:\mathrm{InjRad}_X(x)<r\}\sup_{\DDD\times\DD}|K(z,w)|^2.
\]
\end{proof}

Observe that the validity of the above Lemma \ref{lem: truncated kernel HS-norm estimate} relies on the square-integrability of the kernel, as well as its uniform boundedness over $\DDD \times \DD$. This presents a basic obstruction to applying this to differential operators, whose kernels are singular distributions (involving the Dirac delta and its derivatives) supported entirely on the diagonal. So, the $\mathrm{HS}$-boundedness condition trivially fails, and we will not use this bound directly to estimate $\|A\|^2_{\mathrm{HS}(X_\Gamma)}$. However, in the proof of quantum ergodicity, one introduces a wave propagation operator $P_t$ (defined in Subsection \ref{subsec: radial prop operator P_t}) which is smoothing, and Lemma \ref{lem: truncated kernel HS-norm estimate} is actually applied to operators of the form $P_t AP_t$. %In particular, $P_t$ is chosen in a way such that there exists a constant $C_I$ for which 
%\[
%\sum_{\nu_j\in I}|\langle \psi_j, A\psi_j\rangle|^2\leq \|A\|^2_{\mathrm{HS}(X_\Gamma)}\leq C_I \sum_{\nu_j\in I}\left|\left\langle \psi_j, %\left(\frac{1}{T}\int_0^T P_tAP_t\,dt\right)\psi_j\right\rangle\right|^2.
%\]

\subsection{Quantisation of symbol cut-offs and its spectral multiplier approximation}\label{subsec: symbol cut-offs}
Since we are interested in a fixed spectral interval $J$, instead of working with $\lambda\in \RR^+$ (as used in the definitions in Subsection \ref{subsec: Fourier analysis on Poincare Disc}), we could focus only on $\lambda\in I$. Here $I$ is the interval for the parameter $\lambda$ corresponding to the spectral interval $\nu\in J$ with $\frac{1}{4}+\lambda^2=\nu$. %which automatically makes $C$ a bounded quantity. 
But in order to do so, we need to consider the symbol $a(z, \lambda, b)$ cut-off beyond $I$.

Recall the definition of the spherical function~:
$$\varphi_\lambda(t)=\frac{1}{2\pi}\int_{B} e^{(\frac{1}{2}+i\lambda)\langle u,b\rangle}
  \,db $$
  where $u\in\DD$ is any point at distance $t$ from the origin $0$.
%\SSa{I am a bit confused, Bray uses $e^{(-1/2+i\lambda)}$. Is it because we assumed our Busemann function to be of opposite sign to our distance?}
The following proposition follows from the definition.
\begin{proposition}\label{p:spherical_rho}
Let $\rho\in C_c^\infty(\mathbb \RR)$ be a smooth even function. 
%Define an operator $\Op(\rho)$ on $L^2(\mathbb D)$ by the kernel
%\begin{equation*}
%K_\rho(u,w)=\frac{1}{2\pi}\int_{\mathbb \RR^+\times B}\rho(\lambda)e^{(\frac{1}{2}+i\lambda)\langle u,b\rangle}
%e^{(\frac{1}{2}-i\lambda)\langle w,b\rangle} \lambda\tanh(2\pi\lambda)\,db\,d\lambda.
%\end{equation*}
Then:

\begin{enumerate}
\item $\Op(\rho)$ is a radial operator (in the sense that the kernel depends only on the distance), with kernel
$K_\rho(u,w)=k_\rho(d(z, w))$ where
\[k_\rho(t)=\int_{\mathbb \RR^+} \rho(\lambda)\varphi_\lambda(t) \lambda \tanh{(2\pi \lambda)} d\lambda
\]
\item The operator $\Op(\rho)$ coincides with the spectral multiplier
\[
\Op(\rho)=\rho\left(\sqrt{-\frac{1}{4}-\Delta_{\DD}}\right)\quad\text{on }L^2(\mathbb D),
\]
and satisfies 
\begin{equation}\label{eq: op(rho) property}
    \Op(\rho)\psi= \rho(\lambda)\psi ~~~~~~~~~\text{ on } \DD,
\end{equation}
whenever $\psi$ is a generalised eigenfunction of $-\Delta_\DD$ corresponding to the eigenvalue $\nu=\frac{1}{4}+\lambda^2$.

\item In particular, applying this to the spherical eigenfunction, $\rho(\lambda)$ is related to $k_\rho$ by
$$\rho(\lambda)=\int_{\mathbb \RR^+} k_\rho(t)\varphi_\lambda(t) \sinh(t) dt$$

\item 

\begin{equation}\label{eq: Op(a rho) identity}
    \mathrm{Op}(a\rho) = \Op(a)\Op(\rho),\quad\text{on }L^2(\DD)
\end{equation}

\end{enumerate}
\end{proposition}

%\begin{proof}
%From the definition of integral operators,
%\[
%(\Op(\rho) f)(z)=\int_{\DD}K_\rho(z,w)f(w)d\mu(w).
%\]
%ubstituting $K_\rho$ above,
%\[
%(\Op(\rho) f)(z)=\frac{1}{2\pi}\int_{\RR^+\times B}\rho(\lambda)
%e^{(\frac{1}{2}+i\lambda)\langle z,b\rangle}\left(\int_{\DD} f(w)e^{(\frac{1}{2}-i\lambda)\langle w,b\rangle}d\mu(w)
%\right)\lambda\tanh(2\pi\lambda)db d\lambda .
%\]
%The inner integral is precisely the Helgason Fourier transform
%$\widehat f(\lambda,b)$. Hence
%\[
%(\Op(\rho) f)(z)=\frac{1}{2\pi}\int_{\RR^+\times B}\rho(\lambda)\widehat f(\lambda,b)e^{(\frac12+i\lambda)\langle z,b\rangle}
%\lambda\tanh(2\pi\lambda)db d\lambda.
%\]
%By the Helgason inversion formula,
%\[
%\widehat{\Op(\rho) f}(\lambda,b)=\rho(\lambda)\widehat f(\lambda,b).
%\]
%Let $\psi$ be an eigenfunction of $\Delta_\DD$ corresponding to eigenvalue $\nu_\psi= 1/4+\lambda_\psi^2$. Then 
%\[
%\widehat{\Op(\rho) \psi}(\lambda,b)=\rho(\lambda_\psi)\widehat \psi(\lambda,b), 
%\]
%which implies 
%\[
%\Op(\rho)\psi=\rho(\lambda_\psi)\psi= \rho(\sqrt{1/4-\Delta_\DD})\psi
%\]
%Moreover, 
%\begin{align*}
%    \Op(a)\Op(\rho)f(z)&=\frac{1}{2\pi}\int_{\RR^+\times B} a(z, \lambda, b)e^{(\frac{1}{2}+i\lambda)}\widehat{\Op(\rho) f}(\lambda, b) \lambda\tanh(2\pi\lambda)db d\lambda \\
 %   &= \frac{1}{2\pi}\int_{\RR^+\times B} a(z, \lambda, b)e^{(\frac{1}{2}+i\lambda)}\rho(\lambda)\widehat{f}(\lambda, b) \lambda\tanh(2\pi\lambda)db d\lambda =\Op(a\rho).
%\end{align*}
%\end{proof}
The obtention of $\rho$ from $k_\rho$ is called the Selberg transform (see \cite{Iwaniec02}, Theorem 1.14). Observe that the above formulations are true only on the Poincar\'{e} disc and need to be adjusted while working with hyperbolic surfaces, due to the fact that $\Op(\rho)$ might not descend well to the quotient due to convergence issues in \eqref{e:periodize}.  
In other words, even if $\Op_\Gamma(a)$ is well-defined, the identity
\[
\Op_\Gamma(a\rho)= \Op_\Gamma(a)\Op_\Gamma(\rho) ~~~~~\text{ on } X_\Gamma
\]
might be ill-defined, for lack of convergence in the series defining $\Op_\Gamma(\rho)$ and $\Op_\Gamma(a\rho)$.

To remedy this, using Lemma \ref{lem: truncated kernel HS-norm estimate}, we define a new operator on $X_\Gamma$, $\Op_{\Gamma, r}(a\rho)$ which has finite propagation and whose kernel is given by 
\begin{equation}\label{eq:trunc_period}
K^{\Gamma,r}_{a\rho}(z,w):=\sum_{\gamma\in\Gamma}K_{a\rho}(z,\gamma w)
\chi\left(\frac{d(z,\gamma w)}{r}\right),
\end{equation}
where $\chi$ is a smooth cutoff function such that $\chi \leq 1$ and $\chi = 0$ whenever $d(z, w)> r$. 

The forthcoming proposition analyses the error terms that arise when
replacing $\Op_{\Gamma, r}(a\rho)\psi_j$ by $\Op_\Gamma(a)\rho(\lambda_j)\psi_j$ for $\psi_j$ an eigenfunction of the Laplacian on $X_\Gamma$.

\begin{proposition}\label{thm: truncated periodised Op approximation}
    For a $\Gamma$-invariant symbol $a(z,b, \lambda)$ on $\mathbb D\times\RR^+\times B$ and fixed $r\in \RR^+$, there exists an operator $R_r(a, \rho)$ on $L^2(X_\Gamma)$ and a scalar $E_{r, \rho}(\lambda_j)$ such that
\begin{equation}\label{eq: truncated periodised Op approximation}
    \Op_{\Gamma, r}(a\rho)\psi_j= \Op_\Gamma(a)\rho(\lambda_j)\psi_j+ \Op_\Gamma(a)E_{r, \rho}(\lambda_j)\psi_j+ R_r(a, \rho)\psi_j.
\end{equation}
\end{proposition}
 
The proof follows by combining Lemma \ref{lem:approx_trunc} and Lemma \ref{lem:trunc-rho multiplier approx}, which also show that the Hilbert-Schmidt norm of $R_r(a, \rho)$ and the multiplier $E_{r, \rho}(\lambda_j)$ are small when $r$ is large.

\iffalse
With the quantization \eqref{eq:kernel_quant}, we have
\begin{equation}\label{eq:kernel_comp}
\Op(a\rho)=\Op(a)\Op(\rho),\qquad
K_{a\rho}(z,w)=\int_{\mathbb D}K_a(z,u)\,K_\rho(u,w)\,d\mu(u),
\end{equation}
with $\Op(\rho)$ satisfying 
\[
\Op(\rho)\psi= \rho(\lambda)\psi ~~~~~~\text{ on } \DD.
\]
As explained earlier, the above identities do not descend on $X_\Gamma$. 
\fi

\begin{Lemma}\label{lem:approx_trunc} 
Let $\rho\in C_c^\infty(\RR^+)$ and let $\Op_{\Gamma,r}(a)$ and $\Op_{\Gamma,r}(a\rho)$ be the corresponding integral
operators on $L^2(X_\Gamma)$ as defined before. Assume the finite propagation condition for $K_a$ on $\mathbb D$,
\begin{equation}\label{eq:finite_prop}
K_a(z,w)=0 \qquad\text{whenever } d(z,w)>D
\end{equation}
for some fixed $D>0$. Then there exists an operator $R_r(a, \rho)$ on $L^2(X_\Gamma)$ such that
\begin{equation}\label{eq:approx_mult}
\Op_{\Gamma,r}(a\rho)=\Op_{\Gamma}(a)\Op_{\Gamma,r}(\rho)+ R_r(a, \rho),
\end{equation}
and its Hilbert-Schmidt norm satisfies
\begin{multline}\label{eq:Rr_HS}
\|R_r(a, \rho)\|_{\mathrm{HS}}^2\lesssim\left(\frac{D}{r}\right)^2\left(\sup_{(z,w)\in\DDD\times\DD}|K_a(z,w)|^2\right)\left(\int_0^\infty|\rho(\lambda)|^2\lambda\tanh(2\pi\lambda)d\lambda\right)\\\times {e^{2D}}\left(\Vol(X_\Gamma)+\frac{e^{(r+D)}}{l_{X_\Gamma}}\Vol\{z\in X_\Gamma:\InjRad_{X_\Gamma}(z)<r+D\}\right),
\end{multline}
where $l_{X_\Gamma}$ denotes the length of the shortest closed geodesic on $X_\Gamma$. %and $C$ is an absolute constant (depending only on the chosen normalisations in \eqref{eq:kernel-quant} and on the hyperbolic metric).
\end{Lemma}

\begin{proof}
Recall that, 

\begin{equation}\label{eq:kernel_comp}
%\Op(a\rho)=\Op(a)\Op(\rho),\qquad
K_{a\rho}(z,w)=\int_{\mathbb D}K_a(z,u)\,K_\rho(u,w)\,d\mu(u).
\end{equation}
Substituting  \eqref{eq:kernel_comp} in \eqref{eq:trunc_period}, we have
\begin{align*}
K^{\Gamma,r}_{a\rho}(z,w)
&=\sum_{\gamma\in\Gamma}\int_{\mathbb D}K_a(z,y)K_\rho(y,\gamma w) \chi\left(\frac{d(z,\gamma w)}{r}\right)d\mu(y)\\
%&=\sum_{\gamma\in\Gamma}\int_{\mathbb D}K_a(z,y)K_\rho(y,\gamma w)\left[ \chi\left(\frac{d(y,\gamma w)}{r}\right)+ \left(\chi\left(\frac{d(z,\gamma w)}{r}\right)-\chi\left(\frac{d(y,\gamma w)}{r}\right)\right) \right]d\mu(y)\\
&=\sum_{\gamma\in \Gamma}\int K_a(z,y)K_\rho(y,\gamma w)
\chi\left(\frac{d(y,\gamma w)}{r}\right)d\mu(y)\\
&\qquad\qquad+\sum_{\gamma\in \Gamma}\int K_a(z,y)K_\rho(y,\gamma w)
\left[\chi\left(\frac{d(z,\gamma w)}{r}\right)-\chi\left(\frac{d(y,\gamma w)}{r}\right)\right]d\mu(y).
\end{align*}
The first term is exactly the kernel of $\Op_{\Gamma}(a)\Op_{\Gamma, r}(\rho)$. Now we look at the remainder term, which we call $R_r(a, \rho)$. Since $\chi$ is a smooth bump function, assuming $x_1=d(z, \gamma w)$ and $x_2=d(u, \gamma w)$, using the mean value theorem, we have

\[
\left|\chi\left(\frac{x_1}{r}\right)- \chi\left(\frac{x_2}{r}\right)\right|\leq   \frac{\|\chi'\|_\infty}{r}\,|x_1-x_2|
\]
Now using the bound for the distance function 
\[
|d(z, \gamma w)-d(y, \gamma w)|\leq d(z, y),
\]
we have 
\[
\left|\chi\left(\frac{d(z,\gamma w)}{r}\right)-\chi\left(\frac{d(y,\gamma w)}{r}\right)\right| \leq \frac{\|\chi'\|_\infty}{r}\,d(z,y).
\]
On the support of $K_a(z, y)$ we have $d(z,y)\leq D$ by \eqref{eq:finite_prop}, which implies
\begin{equation}\label{eq:chi_diff}
\left|\chi\left(\frac{d(z,\gamma w)}{r}\right)-\chi\left(\frac{d(y,\gamma w)}{r}\right)\right| \leq \frac{D}{r}\,\|\chi'\|_\infty.
\end{equation}
Moreover, if $K_a(z,y)\neq 0$ and $\chi(d(y,\gamma w)/r)\neq 0$, then
$d(y,\gamma w)\leq r$ which gives us $d(z,\gamma w)\leq r+D$. Therefore $K_{R_r(a, \rho)}(z,w)$ is supported in
$\{(z,w):\exists\gamma~~ \text{such that } d(z,\gamma w)\leq r+D\}$. Combining this with \eqref{eq:chi_diff}, we have
\begin{equation}\label{eq:KR_pointwise}
|K_{R_r(a, \rho)}(z,w)|\leq \frac{D}{r}\|\chi'\|_\infty
\sum_{\gamma\in\Gamma}\mathds{1}_{\{d(z,\gamma w)\le r+D\}}\int_{\DD}|K_a(z,y)||K_\rho(y,\gamma w)|
d\mu(y).
\end{equation}
Now we estimate the $\mathrm{HS}$-norm of the remainder operator. By definition, for $z, w\in \DDD$,
\begin{multline*}
    \|R_r(a, \rho)\|_{\mathrm{HS}(X_\Gamma)}^2=\int_\DDD\int_\DDD |K_{R_r(a, \rho)}(z,w)|^2\,d\mu(z)\,d\mu(w)\\
    \leq \left(\frac{D}{r}\right)^2\|\chi'_\infty\|^2 \int_\DDD\int_\DDD \left|\sum_\gamma \mathds{1}_{\{d(z,\gamma w)\le r+D\}} I_\gamma(z, w)\right|^2d\mu(z)d\mu(w),
\end{multline*}
where 
\[
I_\gamma(z, w):= \int_{\DD}|K_a(z,y)||K_\rho(y,\gamma w)|d\mu(y).
\]
Then, Cauchy-Schwarz on the $\gamma$-sum gives us
\begin{equation}\label{e:CSgamma}
\|R_r(a, \rho)\|_{\mathrm{HS}}^2\leq \left(\frac{D}{r}\right)^2\|\chi'\|_\infty^2 \int_\DDD\int_\DDD \#\{\gamma: d(z, \gamma w)\leq r+D\}\sum_\gamma|I_\gamma(z, w)|^2d\mu(z)d\mu(w) 
\end{equation}
Now, applying Cauchy-Schwarz in the integral defining $I_\gamma(z, w)$ gives
\[
|I_\gamma(z, w)|^2\leq \int_{y\in B(z, D)} |K_a(z, y)|^2dy \int_{y'\in B(z, D)} |K_\rho(y', \gamma w)|^2 dy'.
\]
%Since $K_a(z, y)=0$ whenever $d(z, y)>D$, for $z\in \DDD$
%\[
%\int_{\mathbb D}|K_a(z,y)|^2d\mu(y) \leq \Vol(B_{\mathbb D}(z,D))\sup_{(z, y)\in \DDD\times\DD}|K_a(z, y)| ^2 \lesssim e^{D}\sup_{\DDD\times %\DD}|K_a|^2.
%\]
For $z\in \DDD(r):=\{z\in \DDD:\mathrm{InjRad}_{X_\Gamma}(z)\ge r\}$ the $\gamma$-sum has at most one
contributing term, while on $\DDD(r)^c$ it has at most $O(e^{r+D}/l_{X_\Gamma})$ terms. So, globally
\[
\#\{\gamma: d(z, \gamma w)\leq r+D\}\leq 1+ \frac{e^{r+D}}{l_{X_\Gamma}}\mathds{1}_{\InjRad_{X_\Gamma}(z)\leq r+D}
\]
Injecting into \eqref{e:CSgamma}, we get 
\begin{align*}
&\|R_r(a, \rho)\|_{\mathrm{HS}}^2
\\ &\leq 
%\left(\frac{D}{r}\right)^2\|\chi'_\infty\|^2 \int_\DDD\int_\DDD \#\{\gamma: d(z, \gamma w)\leq r+D\}\sum_\gamma|I_\gamma(z, w)|^2d\mu(z)d\mu(w) \\
\left(\frac{D}{r}\right)^2\|\chi'_\infty\|^2 \int_\DDD\int_\DDD \left( 1+ \frac{e^{r+D}}{l_{X_\Gamma}}\mathds{1}_{\InjRad_{X_\Gamma}(z)\leq r+D} \right) \\
&\qquad\qquad\qquad\qquad\qquad\qquad 
 \times
 \int_{y\in B(z, D)} |K_a(z, y)|^2dy \sum_\gamma\int_{y'\in B(z, D)} |K_\rho(y', \gamma w)|^2 d\mu(y') d\mu(z)  d\mu(w)\\
%&\leq \left(\frac{D}{r}\right)^2\|\chi'_\infty\|^2  e^D\sup_{\DDD\times \DD} |K_a|^2 \int_{z\in\DDD}\int_{w\in\DDD} \left( 1+ \frac{e^{r+D}}{l_{X_\Gamma}}\mathds{1}_{\InjRad_{X_\Gamma}(z)\leq r+D} \right)\\
%&\qquad\qquad\qquad\qquad\qquad\qquad\qquad\qquad\qquad \times\left(\sum_\gamma\int_\DD |K_\rho(y, \gamma w)|^2dy\right)d\mu(w) d\mu(z)\\
&\leq \left(\frac{D}{r}\right)^2\|\chi'_\infty\|^2  e^D\sup_{\DDD\times \DD} |K_a|^2 \int_{z\in\DDD}\left( 1+ \frac{e^{r+D}}{l_{X_\Gamma}}\mathds{1}_{\InjRad_{X_\Gamma}(z)\leq r+D}\right)  \\
&\qquad\qquad\qquad\qquad\qquad\qquad\qquad\qquad\qquad 
 \times\int_{w\in\DDD}\left(\sum_\gamma\int_{y'\in B(z, D)} |K_\rho(y', \gamma w)|^2d\mu(y')\right)d\mu(w)\\
&\leq \left(\frac{D}{r}\right)^2\|\chi'_\infty\|^2  e^D\sup_{\DDD\times \DD} |K_a|^2 \left(\Vol(X_\Gamma)+\frac{e^{(r+D)}}{l_{X_\Gamma}}\Vol\{x\in X_\Gamma:\InjRad(x)<r+D\}\right)\\
&\qquad\qquad\qquad\qquad\qquad\qquad\qquad\qquad\qquad 
 \times \Vol(B(0, D))\int_{w\in\DD} |K_\rho(0, w)|^2d\mu(w)\\
&\leq \left(\frac{D}{r}\right)^2\|\chi'_\infty\|^2  e^{2D} \left(\int_0^\infty|\rho(\lambda)|^2\lambda\tanh(2\pi\lambda)d\lambda\right) \sup_{\DDD\times \DD} |K_a|^2
\\
&\qquad\qquad\qquad\qquad\qquad\qquad\qquad\qquad\times\left(\Vol(X_\Gamma)+\frac{e^{(r+D)}}{l_{X_\Gamma}}\Vol\{x\in X_\Gamma:\InjRad(x)<r+D\}\right).
\end{align*}
For the third inequality, we used the fact that 
\begin{align*}\sum_\gamma \int_{w\in\DDD}\int_{y'\in B(z, D)} |K_\rho(y', \gamma w)|^2d\mu(y')d\mu(w)
= \int_{y'\in B(z, D)}\int_{w\in\DD} |K_\rho(y',  w)|^2d\mu(y')d\mu(w)
\end{align*}
and the fact that $K_\rho(y',  w)$ depends only on the distance between $y'$ and $w$, so that $y'\mapsto \int_{w\in\DD} |K_\rho(y',  w)|^2d\mu(w)$ is constant.

\end{proof}

Now $\Op_{\Gamma, r}(\rho)$ in \eqref{eq:approx_mult} can be approximated with the spectral multiplier $\rho_{X_\Gamma}$, which introduces the other error term.

\begin{Lemma}\label{lem:trunc-rho multiplier approx}
Let $\psi_j$ be an $L^2$-eigenfunction of $\Delta_{X_\Gamma}$ corresponding to the eigenvalue $\nu_j=\frac{1}{4}+\lambda_j^2$. Then
\[
\Op_{\Gamma, r}(\rho)\psi_j= \rho(\lambda_j)\psi_j + E_{r, \rho}(\lambda_j)\psi_j,\quad\text{ where } \quad E_{r, \rho}(\lambda_j)=\mathcal{O}_{\rho}\left(\frac{1}{(1+r)^2}\right).
\]
\end{Lemma}
\begin{proof}

 %{\color{blue}Then lifting $\psi_j$ to $\mathbb D$ gives a generalised eigenfunction of $\Delta_{\mathbb D}$ with the same eigenvalue.}  
Let $\tilde{\psi_j}(z)=\psi_j(\pi(z))$ be the lift of $\psi_j$ on $\DD$, where $\pi:\DD\to X_\Gamma$ is the covering map. Then $\tilde{\psi_j}$ is $\Gamma$-invariant and satisfies
\[
\tilde{\psi_j}(\gamma w)=\tilde{\psi_j}(w)\qquad (\gamma\in\Gamma),
\]
%By definition,
%\[ 
%K_\rho(z,w)=\frac{1}{2\pi}\int_{\mathbb \RR^+\times B}\rho(\lambda)e^{(\frac{1}{2}+i\lambda)\langle z,b\rangle}
%e^{(\frac{1}{2}-i\lambda)\langle w,b\rangle} \lambda\tanh(2\pi\lambda)\,db\,d\lambda,
%%\]
Recall from Proposition \ref{p:spherical_rho} that the kernel $K_\rho$ of $\Op(\rho))$ is radial, i.e. it is of the form $K_\rho(z, w)= k_\rho(d(z,w))$. Recall that
\[
K^{\Gamma,r}_{\rho}(x,y):=\sum_{\gamma\in\Gamma}K_{\rho}(z,\gamma w)
\chi\left(\frac{d(z,\gamma w)}{r}\right),
\]
where $z, w\in \DDD$ are the lifts of $x, y\in X_\Gamma$ ($x=\pi(z), y=\pi(w)$). Then
\[
\Op_{\Gamma, r}(\rho)\psi_j(x):= \int_{X_\Gamma} K_{\rho}^{\Gamma, r}(x, y)\psi_j(y)d\mu(y)=\int_\DDD \sum_{\gamma\in\Gamma}K_{\rho}(z,\gamma w)
\chi\left(\frac{d(z,\gamma w)}{r}\right)\tilde{\psi}_j(w)d\mu(w).
\]
Since only finitely many $\gamma$ contribute to the sum when $d(z, \gamma w)\leq r$, the $\Gamma$-sum is finite and can be interchanged with the integral.
\[
\Op_{\Gamma, r}(\rho)\psi_j(x)=\sum_{\gamma\in\Gamma}\int_\DDD K_{\rho}(z,\gamma w) \chi\left(\frac{d(z,\gamma w)}{r}\right)\tilde{\psi}_j(w)d\mu(w).
\]
Writing $w'=\gamma w$ with $w\in \DDD$ and using the fact that $d\mu(w)=d\mu(w')$ along with $\tilde{\psi}_j(\gamma w)=\tilde{\psi}(w)$, we have
\begin{align*}
    \Op_{\Gamma, r}(\rho)\psi_j(x)=\int_\DD K_{\rho}(z,w) \chi\left(\frac{d(z,w)}{r}\right)\tilde{\psi}_j(w)d\mu(w)
\end{align*}
Defining $k_{\rho, r}(t)=k_\rho(t)\chi(t/r)$, we have $K_{\rho}(z,w) \chi\left(\frac{d(z,w)}{r}\right)=k_{\rho, r}(d(z,w))$ and 
\[
(\Op_{\Gamma, r}(\rho)\psi_j)\circ \pi (z)=\int_\DD k_{\rho, r}(d(z, w))\tilde{\psi}_j(w)d\mu(w)=h_{\rho, r}(\lambda_j) \tilde{\psi}_j(z),
\]
where $h_{\rho, r}(\lambda)=\mathcal{S}(k_{\rho, r})(\lambda)$ is given by the spherical transform, as in Proposition \ref{prop: radial integral op}. Then projecting back down to $X_\Gamma$, we have
\[
\Op_{\Gamma, r}(\rho)\psi_j(x)=h_{\rho, r}(\lambda_j) {\psi}_j(x).
\]
Moreover, by the definition of $\Op(\rho)$ on $\DD$ and the fact that $\Op(\rho)\psi=\rho(\lambda)\psi$ for generalised eigenfunctions on $\DD$, we have 
\[
\mathcal{S}(k_\rho)(\lambda)=\rho(\lambda).
\]
Now, $h_{\rho, r}(\lambda)$ can be written in terms of the spherical function $\varphi_{\lambda}$ as 
\begin{align*}
    h_{\rho, r}(\lambda)&= \mathcal{S}(k_{\rho, r})(\lambda)=\int_0^\infty k_{\rho, r}(t)\varphi_{\lambda}(t)\sinh{t}\, dt= \int_0^\infty k_{\rho}(t)\chi\left(\frac{t}{r}\right)\varphi_{\lambda}(t)\sinh{t}\, dt\\
    &= \int_0^\infty \left(k_\rho(t)-  k_{\rho}(t)\left(1-\chi\left(\frac{t}{r}\right)\right)\right)\varphi_{\lambda}(t)\sinh{t}\, dt\\
    &= \rho(\lambda)-\int_0^\infty k_\rho(t) \left(1-\chi\left(\frac{t}{r}\right)\right) \varphi_\lambda(t) \sinh{t}\, dt=: \rho(\lambda)+ E_{r, \rho}(\lambda_j)
\end{align*}
Then,
\begin{align*}
 |E_{r, \rho}(\lambda_j)| = |h_{\rho, r}(\lambda)-\rho(\lambda)|&= \left|\int_0^\infty  -k_\rho(t) \left(1-\chi\left(\frac{t}{r}\right)\right)\varphi_{\lambda}(t)\sinh{t}\, dt \right|\leq \int_r^\infty \left| k_\rho(t) \varphi_{\lambda}(t)\sinh{t} \right|\, dt. \\
\end{align*}
It is known that $|\varphi_\lambda|\leq Ce^{-t/2}(1+t)$ (see Equation (3.1.5) from Chapter 5 of \cite{bray1994fourier}) and $\sinh{t}\leq e^{t}/2$. Moreover, the required decay of $k_\rho(t)$ is proved in Appendix \ref{app: radial kernel decay}, where we show that for every $N\geq 0$,
\[
|k_\rho(t)|\lesssim e^{-t/2}(1+t)^{-N},
\]
for any $t$. So, for large $t\geq 1$, %So, assuming $t\geq 1$, we have $1+t\le 2t$, hence
\begin{align*}
|h_{\rho, r}(\lambda_j)-\rho(\lambda_j)|&\leq \int_r^\infty \left| k_\rho(t) \varphi_{\lambda_j}\sinh{t} \right|\, dt \lesssim  \int_r^\infty C_{N, \rho}\frac{e^{-t/2}}{t^N} \, e^{-t/2}(1+t) \, e^t dt\\
&\leq \int_r^\infty C_{N, \rho} \frac{1}{(1+t)^N}(1+t)dt=\mathcal{O}_{N, \rho}\left(\frac{1}{(1+r)^{N-2}}\right)
\end{align*}
for every $N$. Choosing $N=4$ suffices for our purpose.
\end{proof}

\section{Estimate on the quantum variance of hyperbolic surfaces}\label{sec: quant estimate}

Let $X_\Gamma:=\Gamma \backslash \DD$ with fundamental domain $\DDD\subset \DD$. Given any operator $A$ on $L^2(X_\Gamma)$ that satisfies the conditions in Theorem \ref{thm: main_theorem}, we define its ``quantum variance'' (on some fixed interval $J$) as 
\[
\mathrm{Var}_J(A):=\frac{1}{N(X_\Gamma, J)}\sum_{j:\nu_j\in J} \left |\langle \psi_j,  A\psi_j \rangle-\frac{1}{\Vol(SX_\Gamma)}\int_{SX_\Gamma}a(z,\lambda_j, b)e^{\langle z,b \rangle}d\mu(z)db\right|^2,
\]
where $N(X_\Gamma, J)$ denotes the number of eigenvalues contained in $J$. 

The aim of this section is to prove a quantitative estimate for $\mathrm{Var}_J(A)$. We first establish this bound under the additional assumption that
\[
\int_{SX_\Gamma}a(z,b, \lambda)e^{\langle z,b \rangle}d\mu(z)db=0
\]
 for any $\lambda\in I$. We will indicate in Section \ref{sec: proof_main_thm} how to deal with the general case where this assumption is not satisfied, and also prove the quantum ergodicity result, i.e, Theorem \ref{thm: main_theorem}.

\begin{note}
    We shall use $I, J$ to denote the support interval of the parameters $\lambda, \nu$ (respectively), where the parameters are related as $\nu=\frac{1}{4}+\lambda^2$.
\end{note}

The estimate is obtained by following these steps.
\begin{itemize}
    \item[(i)] (Introducing a radial propagator and an averaging operator)
    We introduce the smooth radial integral operator $P_{t}$ and define the averaging operator 
    \[
    \overline{A}_T:=\frac{1}{T}\int_0^T P_tAP_t dt.
    \]
    We use the explicit formula for its symbol $\overline{a}_T$ and kernel $K_{\overline{A}_T}$ obtained in Subsection \ref{subsec: radial prop operator P_t} to provide an upper bound for  $\|K_{\overline{A}_T}\|^2_{L^2(\DDD\times \DD)}$. %$\mathrm{HS}$-norm of $\overline{A}_T$. 
    Although $A$ is originally defined on $X_\Gamma$, we can identify it with its lift on $\DD$, whose kernel is $\Gamma$-invariant. Accordingly, most of the computations will be carried out on $\DD$, and we will indicate when we pass back to the quotient $X_\Gamma$.
     
    \item[(ii)] (Square integrability of $K_{\overline{A}_T}$ and Nevo's ergodic theorem) %Under the assumptions on $A$ from Theorem \ref{thm: main_theorem}, bounding $\|\overline{A}_T\|^2_{\mathrm{HS}(X_\Gamma)}$ is equivalent to estimating $\|K_{\overline{A}_T}\|^2_{L^2(\DDD\times \DD)}$. 
    This key estimate comes from rewriting the integral in suitable group coordinates and using the zero-average assumption on the symbol in its angular variable (assumption \ref{condition A1}). This allows us to view the resulting integral as a family of averaging (or ``convolution'') operators on $\mathrm{PSU}(1,1)$ applied to a function $f$ involving only the $\theta$-derivative of the symbol. Nevo's ergodic theorem used in Subsection \ref{subsec: Hs-norm of avg op.} precisely provides the necessary $L^2$-decay of these convolution operators, from which we deduce that
    \[
    \iint_{\DDD\times \DD}|K_{\overline{A}_T}(z, w)|^2\leq \frac{C}{T}\frac{\Vol(SX_\Gamma)}{\alpha(\lambda_1)},
    \]
    where $\alpha(\lambda_1)$ depends on the spectral gap of $X_\Gamma$, and $C$ depends on a uniform control of the relevant angular derivatives of the symbol. Establishing such a uniform control over $\lambda\in \RR^+$ would require strong assumptions on our symbol -- preventing us, for instance, from applying the result to differential operators. Since we are only interested in a fixed compact spectral interval $\nu\in J$ (equivalently, $\lambda\in I$), such stronger assumptions can be avoided by introducing a spectral cut-off for our symbol $a(z, \lambda, b)$. The important point for us is that the overall argument and use of Nevo's theorem produce decay in $T$, which is the key input for the variance bound.

    \item[(iii)] (Square integrability of $K_{\overline{A'}_T}$) The next step is to localise the symbol $a(z, \lambda, b)$ spectrally to our interval of interest $\lambda\in I$. For this, we choose a smooth cut-off $\rho$ which is supported in a slightly larger interval $I'$ and is equal to $1$ on $I$, and define
    \[
    \overline{A'}_T= \Op(\overline{a}_T\rho).
    \]
    Even though $K_{\overline{A}_T}$ has finite propagation, $K_{\overline{A'}_T}$ does not. So we invoke the Hilbert-Schmidt estimate from Lemma \ref{lem: truncated kernel HS-norm estimate} for the geometrically truncated operator $\Op_{\Gamma, r}(\overline{a}_T\rho)$, whose kernel is now supported within a distance $r$ from the diagonal. This leads to the following estimate in Subsection \ref{subsec: square-integ of trunc avg op}
     \[
    \iint_{\DDD\times \DD}|K_{\overline{A'}_T}(z, w)|^2\leq \frac{C_I}{T}\frac{\Vol(SX_\Gamma)}{\alpha(\lambda_1)},
    \]
    where $C_I$ is a bounded quantity depending only on $I$.
    
    \item[(iv)] (Quantum variance estimate) Finally, we return to the definition of the quantum variance and replace $A$ by the averaged and spectrally localised operator constructed above. More precisely, in Subsection \ref{subsec: estimating QV}, we use Proposition \ref{prop: MS Prop 2} to see that 
    \[
    \sum_{j:\nu_j\in J}  |\langle \psi_j,  A\psi_j \rangle|^2 \leq \frac{1}{C_I^2} \sum_{j=0}^\infty \left|\left\langle \psi_j, \left(\frac{1}{T}\int_0^TP_tAP_t\, dt\right)\rho_{X_\Gamma}\psi_j \right\rangle\right|.
    \]
    Using Proposition \ref{thm: truncated periodised Op approximation}, we use the approximation of $\Op_{\Gamma, r}(a\rho)$ with $\Op_\Gamma(a)\rho_{X_\Gamma}$,
    \[
    \frac{1}{C_I^2} \sum_{j=0}^\infty \left|\left\langle \psi_j, \overline{a}_T\rho_{X_\Gamma}\psi_j \right\rangle\right| \lesssim_I \|\Op_{\Gamma, r}(\overline{a}_T\rho)\|^2_{\mathrm{HS}}+ 2\left(\| \Op_\Gamma(\overline{a}_T)\|^2_{\mathrm{HS}}|E_{r, \rho}(\lambda_j)|+\|R_r(\overline{a}_T, \rho)\|^2_{\mathrm{HS}}\right).
    \]
    The main term in the above inequality is then controlled by the Hilbert–Schmidt bound obtained in the previous step. The remaining error terms arising from the spectral cut-off and the truncation procedure are dealt with using Lemma \ref{lem:approx_trunc} and Lemma \ref{lem:trunc-rho multiplier approx}. 
\end{itemize}

%\subsection{Symbol of $P_tAP_t$ and boundedness of its kernel} %Unlike \cite{MassonSahlsten17}, we will use the symbol of the operator instead of using the kernel directly and make appropriate changes. %We now find the symbol of the operator $P_tAP_t$. 

\subsection{Radial propagation operator and the averaging kernel}\label{subsec: radial prop operator P_t}
For fixed $\sigma>0$, let $\chi_{t, \sigma}(r):[0, \infty)\to \RR$ be a smooth decreasing function such that $\chi_{t, \sigma}(r)=1$ for $r\leq t-\sigma$, $\chi_{t, \sigma}(r)=0$ for $r> t$ and $0\leq \chi_{t, \sigma} \leq 1$. %In order to demonstrate the $\epsilon$-dependence of $\chi_\epsilon$ explicitly, 
%In particular, for $\epsilon<\epsilon_0$ (choice of $\epsilon_0$ to be determined later)
\iffalse
In particular,
\[
\chi_{t, \sigma}(r):=\eta\left(\frac{r-t}{\sigma}\right),
\]
where $\eta\in C^\infty(\RR)$ is some decreasing function such that $\eta=1$ on $(-\infty, -1]$ and $\eta=0$ on $[0, \infty)$.
\fi
Define $k_{t, \sigma}(r):= \frac{1}{\sqrt{\cosh{t}}}\chi_{t, \sigma}(r)$ and the corresponding smooth radial operator on $\DD$
\[
P_{t, \sigma}u(z):=\int_\DD K_{t, \sigma}(z, w)u(w)d\mu(w) = \int_\DD k_{t,  \sigma}(d_\DD(z, w)) u(w)d\mu(w).
\]
Moreover, consider the sharp radial operator (used in \cite{MassonSahlsten17}) 
\[
P_t^\sharp(u(z)):= \frac{1}{\sqrt{\cosh{t}}} \int_{B(z,t)} u(w)d\mu(w), ~~~~~~z\in \DD,
\]
with kernel 
\[
K_t^\sharp(z, w):= k_t^\sharp(d_\DD(z, w)), \quad \text{ where } \quad k_t^\sharp(r):= \frac{1}{\sqrt{\cosh{t}}} \mathds{1}_{\{r\leq t\}}.
\]

Since the observable in \cite{MassonSahlsten17} is multiplication by a bounded function $a$ whose kernel is a Dirac mass, choosing the step radial operator $P_t^\sharp$, the propagation $P_t^\sharp aP_t^\sharp$ immediately produces an operator with a square-integrable kernel. In contrast, while working with differential operators, the kernel involves derivatives of the Dirac distribution, and the interaction of these singularities with the sharp propagation operator introduces additional technical difficulties that do not arise in the \cite{MassonSahlsten17} setting. So, instead of using a sharp propagator as in \cite{MassonSahlsten17}, we use the smooth $P_{t,\sigma}$ which preserves similar spectral information while making the computations a bit nicer.
 
 $P_{t, \sigma}, P_t^\sharp$ being radial integral operators, they behave as spectral multipliers. In other words, we have the following converse of Proposition \ref{p:spherical_rho} (also see \cite{Iwaniec02}, Theorem 1.14):

\begin{proposition}\label{prop: radial integral op}
Let $X=\Gamma\backslash\mathbb{D}$ be a hyperbolic surface. Let $k:[0,+\infty)\to\mathbb{C}$ be a smooth function with compact support. If $\psi_\nu$ is an eigenfunction of the Laplacian on $X$ of eigenvalue $\nu$, then it is an eigenfunction of the radial integral operator $A$ associated to $k$. That is,
\begin{equation}\label{eq: radial_op_Selberg_relation}
    A\psi_\nu(z)=\int k(d(z,w))\psi_\nu(w)d\mu(w)=h(\lambda)\psi_\nu(z),
\end{equation}
where the eigenvalue $h(\lambda)$ is given by the Selberg transform of the kernel $k$:
\[
h(\lambda)=\mathcal{S}(k)(\lambda),
\]
and $\lambda\in\mathbb{C}$ is defined by the equation $\nu=\frac{1}{4}+\lambda^2$.
\end{proposition}

Let $h_t^\sharp, h_{t,\sigma}$ denote the corresponding Selberg transform of $k_t^\sharp, k_{t, \sigma}$ respectively. From Proposition 4.2 of \cite{MassonSahlsten17} we know that for any fixed compact interval $J\subset (1/4, \infty)$, there exists constants $C_I, T_I>0$ such that for  $\lambda\in I$ with $\nu=1/4+\lambda^2\in J$ and for all $T\geq T_I$ we have
\begin{equation}\label{ineq: LeMasson-Sahlsten Prop 4.2}
    \frac{1}{T}\int_0^T h_t^\sharp(\lambda)^2 dt \geq C^*_I.
\end{equation}

We can show a similar result for $P_{t, \sigma}$ and $h_{t, \sigma}$ as follows.

\begin{proposition}\label{prop: MS Prop 2}
    For any fixed compact interval $J\subset (1/4, \infty)$, there exists constants $\sigma_0, C(I,\sigma), T(I, \sigma)>0$ such that for $\lambda\in I$ with $\nu=1/4+\lambda^2\in J$ and for all $T\geq T(I, \sigma)$ we have
    $$\frac{1}{T}\int_0^T h_{t, \sigma}(\lambda)^2 dt \geq C(I, \sigma), \quad \text{ for every } \sigma<\sigma_0.$$
\end{proposition}
See Appendix \ref{app: proof of prop 3.3} for a detailed proof of the above proposition.
\begin{note}
    For notational convenience, moving forward, we shall fix some $\sigma<\sigma_0$ and drop it from our notations, defining the smooth operator $P_t:=P_{t,\sigma}$ supported on the ball $B(z, t)$ with its kernel $K_{t, \sigma}, k_{t, \sigma}$ denoted as $K_t, k_t$ (respectively) and the Selberg transform $h_{t, \sigma}$ as $h_t$.
\end{note}

\iffalse
\begin{note}
    For the remainder of this paper, with slight abuse of notation, we shall use $I$ to denote the support interval of both the parameters $\lambda$ as well as $\nu$, making the context clear wherever necessary.
\end{note}
\fi

We now use the Helgason Fourier transform introduced in Subsection \ref{subsec: Fourier analysis on Poincare Disc} to provide an explicit formula for the symbol and kernel of $P_tAP_t$ based on the symbol of $A$. Since $P_t$ is radial, using Proposition \ref{prop: radial integral op}, 
\[
AP_te^{(\frac{1}{2}+i\lambda)\langle z, b \rangle}= h_t(\lambda)A e^{(\frac{1}{2}+i\lambda)\langle z, b \rangle}= h_t(\lambda)a(z, \lambda, b)e^{(\frac{1}{2}+i\lambda)\langle z, b \rangle}.
\]
Then 
\[
P_tAP_te^{(\frac{1}{2}+i\lambda)\langle z, b \rangle}=\frac{h_t(\lambda)}{\sqrt{\cosh t}}\int_{B(z,t)}\chi_{\sigma}(d(z,w))a(w,\lambda, b)e^{(\frac{1}{2}+i\lambda)\langle w, b \rangle}d\mu(w),
\]
\iffalse
\begin{align*}
    P_tAP_tu(z)&=\int_{\DD}k_t(d(z, w))A(P_tu(w))d\mu(w)\\
    &= \int_{\DD}k_t(d(z, w))\left(\frac{1}{2\pi}\iint_{\RR^+\times B}e^{(\frac{1}{2}+i\lambda)\langle w,b\rangle}a(w, \lambda, b)\widehat{P_tu}(\lambda, b)\tanh({2\pi\lambda})\,d\lambda\,db \right)d\mu(w)\\
    &=\frac{1}{2\pi}\frac{1}{\sqrt{\cosh{t}}}\int_{B(z, t)}\left(\iint_{\RR^+\times B}e^{(\frac{1}{2}+i\lambda)\langle w,b\rangle}a(w, \lambda, b)h_t({\lambda})\widehat{u}(\lambda, b)\tanh({2\pi\lambda})\,d\lambda\,db \right)d\mu(w)\\
    &=\frac{1}{2\pi}\frac{1}{\sqrt{\cosh{t}}}\iint_{\RR^+\times B}\left(\int_{B(z, t)}e^{(\frac{1}{2}+i\lambda)\langle w,b\rangle}a(w, \lambda, b)d\mu(w)\right) h_t({\lambda})\widehat{u}(\lambda, b)\tanh({2\pi\lambda})\,d\lambda\,db \\
    &=\frac{1}{2\pi}\iint_{\RR^+\times B}\left(\frac{h_t({\lambda})}{\sqrt{\cosh{t}}}\int_{B(z, t)}e^{(\frac{1}{2}+i\lambda)(\langle w,b\rangle-\langle z,b\rangle)}a(w, \lambda, b)d\mu(w)\right) \widehat{u}(\lambda, b)e^{(\frac{1}{2}+i\lambda)\langle z,b\rangle}\tanh({2\pi\lambda})\,d\lambda\,db.
\end{align*}
\fi
and the symbol of $P_tAP_t$ is given by
\begin{equation}
    a_{P_tAP_t}(z, \lambda, b)= \frac{h_t({\lambda})}{\sqrt{\cosh{t}}}\int_{B(z, t)}\chi_{\sigma}(d(z,w))e^{(\frac{1}{2}+i\lambda)(\langle w,b\rangle-\langle z,b\rangle)}a(w, \lambda, b)d\mu(w)
\end{equation}
while the kernel is given by
\begin{align*}
      &K_{P_tAP_t}(z, z')= \frac{1}{2\pi}\iint_{\RR^+\times B} a_{P_tAP_t}(z, \lambda, b) e^{(\frac{1}{2}+i\lambda)\langle z,b\rangle}e^{(\frac{1}{2}-i\lambda)\langle z', b\rangle}\lambda\tanh({2\pi\lambda})\, db d\lambda \\
     % &= \frac{1}{2\pi}\iint_{\RR^+\times B} \frac{h_t({\lambda})}{\sqrt{\cosh{t}}}\left(\int_{B(z,t)}\chi_{t, \sigma}(r)e^{(\frac{1}{2}+i\lambda)(\langle w,b\rangle-\langle z,b\rangle)}a(w, \lambda, b)d\mu(w)\right)e^{(\frac{1}{2}+i\lambda)\langle z,b\rangle}e^{(\frac{1}{2}-i\lambda)\langle z', b\rangle}\lambda\tanh({2\pi\lambda})\, db d\lambda\\
      &= \frac{1}{2\pi}\iint_{\RR^+\times B} \frac{h_t({\lambda})}{\sqrt{\cosh{t}}}\left(\int_{B(z,t)}\chi_{t, \sigma}(d(z,w))e^{(\frac{1}{2}+i\lambda)\langle w,b\rangle}a(w, \lambda, b)d\mu(w)\right)e^{(\frac{1}{2}-i\lambda)\langle z', b\rangle}\lambda\tanh({2\pi\lambda})\, db d\lambda.
\end{align*}
The symbol of
\[
\overline{A}_T:=\frac{1}{T}\int_0^TP_tAP_tdt
\]
is given by
\begin{equation*}
   \overline{a}_T(z, \lambda, b) = \frac{1}{T} \int_0^Ta_{P_tAP_t}(z, \lambda, b)dt = \frac{1}{T}\int_0^T \left[\frac{h_t({\lambda})}{\sqrt{\cosh{t}}}\int_{B(z, t)}\chi_{t, \sigma}(r)e^{(\frac{1}{2}+i\lambda)(\langle w,b\rangle-\langle z,b\rangle)}a(w, \lambda, b)d\mu(w)\right] dt,
\end{equation*}
where $r=d(z,w)$.

In what follows, we repeatedly use estimates such as \eqref{ineq: LMS- PDO HS norm}, which require a pointwise bound on the kernel of the operator. Note, however, that our assumptions allow $A$ to be a differential operator, in which case the kernel of $A$ itself is distributional rather than a bounded function. 

However, $P_t$ has a smoothing effect as soon as $t>0$, and the following lemma yields the required bound for the kernel of $P_tA P_t$.

%\[
%A_m u(x)=\sum_{|\alpha|\leq k} C_{m, \alpha}(x)\pa^\alpha u(x).
%\]
%Hence
%\[
%|A_mu(x)|\leq \sum_{|\alpha|\leq k}|C_{m, \alpha}(x)||\pa^\alpha u(x)|\leq C\|u\|_{C^k(B(x,S))},
%\]
%provided the coefficients $C_{m, \alpha}$ are globally bounded and independent of $m$.  

\begin{Lemma}\label{lem: sup norm bound of avg op}
    Let $A$ be an operator on $\DD$ whose distributional kernel $K_A\in \DDD'(\DD\times \DD)$ is $\Gamma$-invariant and satisfies the conditions in Theorem \ref{thm: main_theorem}. Then for each $t>0$, $P_tAP_t$ has a bounded integral kernel on $\DD\times \DD$. In other words, 
    \[
    \sup_{(z, w)\in \DD\times \DD} |K_{P_tAP_t}(z, w)|\leq C(t, \sigma, k).
    \]
\end{Lemma}

\begin{proof}
    The kernel $K_t(z,w)$ associated with the radial operator $P_t$ is 
\[
K_t(z, w):= k_t(d(z, w)), \quad \text{ where } \quad k_t(r):= \frac{1}{\sqrt{\cosh{t}}} \chi_{t, \sigma}(r),
\]
where $\chi_{t, \sigma}(r):[0, \infty)\to \RR$ is a decreasing smooth function such that $\chi_{t, \sigma}(r)=1$ for $r\leq t-\sigma$ and $\chi_{t, \sigma}(r)=0$ for $r> t$. In order to demonstrate the $\sigma$-dependence on the bounds explicitly,  we choose
%In particular, for $\epsilon<\epsilon_0$ (choice of $\epsilon_0$ to be determined later)
\[
\chi_{t, \sigma}(r):=\eta\left(\frac{r-t}{\sigma}\right),
\]
where $\eta\in C^\infty(\RR)$ is some decreasing function such that $\eta=1$ on $(-\infty, -1]$ and $\eta=0$ on $[0, \infty)$.

\begin{align*}
    K_{AP_{t}}(u, w)&=\int_\DD K_A(u, v)K_{P_{t}}(v, w)d\mu(v)=\int_\DD K_A(u, v) \frac{1}{\sqrt{\cosh{t}}} \chi_{t, \sigma}(d_\DD(v, w)) d\mu(v)
\end{align*}
Let $f_w(v)= \frac{1}{\sqrt{\cosh{t}}} \chi_{t, \sigma}(d(v, w))$. Then
\[
 K_{AP_t}(u, w) = Af_w(u).
\]
From assumption \eqref{ineq: local assumption}, 
\[
|K_{AP_t}(u, w)|\leq C \|f_w\|_{C^k(B(u, S))}.
\]
Now, $f_w$ is smooth in $v$. %Indeed, for $d_\DD(u, w)\leq t-\sigma$, even though the distance function is not smooth on the diagonal, $f_w$ is constant, hence smooth. Away from the diagonal, the distance function is smooth.
Moreover, for any $v\in B(u, S)$, if $f_w(v)\neq 0$ then $d_\DD(v, w)\leq t$, which implies
\[
d_\DD(u, w)\leq d_\DD(u, v)+ d_\DD(v, w)\leq S+t.
\]
So, assuming $r(v)=d_\DD (v, w)$, for any $v\in B(u, S)$,
\[
 |\pa_v^\alpha f_w(v)|=\left|\frac{1}{\sqrt{\cosh{t}}}\pa^\alpha_v(\chi_{t, \sigma}\circ r)(v)\right|= \frac{1}{\sqrt{\cosh{t}}}\left|\sum_{j=1}^{|\alpha|}\chi_{t, \sigma}^{(j)}(r)\cdot P_j(\pa^1_v r, \pa^2_v r, \cdots)\right|\leq \frac{1}{\sqrt{\cosh{t}}}\sum_{j=1}^{|\alpha|}|\chi_{t, \sigma}^{(j)}(r)|\cdot |P_j|
\]
where $P_j$ is a polynomial in the derivatives of $r$ and independent of $\sigma$ (they depend only on the geometry of $\DD$). Also,
\[
|\chi_{t, \sigma}^{(j)}(r)|\leq \frac{\|\eta^{(j)}\|_\infty}{\sigma^j}<\infty
\]
for any fixed $\sigma$. This gives,
\[
\|f_w\|_{C^k(B(u, S))}=\sup_{v:d_\DD(v, w)\leq t+S}\sum_{|\alpha|\leq k}|\pa_v^\alpha f_w(v)| \leq \frac{1}{\sqrt{\cosh{t}}}\sum_{|\alpha|\leq k} \left(C_{|\alpha|}\sum_{j=1}^{|\alpha|}\frac{\|\eta^{(j)}\|_\infty}{\sigma^j}\right),
\]
where $C_{|\alpha|}= \max_{j=1}^{|\alpha|}\sup_{B(w, t+S)}|P_j|$.

Now, consider the full kernel
\begin{align*}
    |K_{P_{t}AP_{t}}(z,w)|&=\left|\int_\DD K_{t}(z, u)K_{AP_{t}}(u, w)d\mu(u)\right|=\left|\int_{u\in B(z, t)} K_{t}(z, u)K_{AP_{t}}(u, w)d\mu(u)\right|\\
    &\leq \sup_{u\in B(z, t)\cap B(w, t+S)}\left|K_{AP_t}(u, w)\right|\int_{u\in B(z, t)}|K_{t}(z, u)|d\mu(u).
\end{align*}
Clearly, $K_{P_{t}AP_{t}}(z, w)=0$ when $d(z, w)>2t+S$ which reinforces the finite propagation condition on $K_{P_{t}AP_{t}}$. Note that 
\[
\|K_{t}\|_{L^1(\DD)}=\int_{u\in B(z, t)}|k_{t}(d_\DD(z, u))|d\mu(u)\leq \frac{1}{\sqrt{\cosh{t}}}\Vol(B(z, t))=M_t ~(\text{say}).
\]
Combining these bounds, we have
\begin{equation}\label{ineq: pointwise averaged kernel estimate}
    \sup_{(z, w)\in \DD\times \DD}|K_{P_{t}AP_{t}}(z,w)|\leq C\sum_{|\alpha|\leq k} \frac{C'(t, |\alpha|)}{\cosh{t}}  \left(\sum_{j=1}^{|\alpha|}\frac{1}{\sigma^j}\right).
\end{equation}
\end{proof}

\begin{remark}
 For a differential operator, \eqref{ineq: local assumption} is the natural operator-theoretic substitute for saying ``order $\leq k$ with bounded coefficients''. Under the setting of Theorem \ref{thm: main_theorem}, we consider a sequence of operators $A_p$ on $X_p$. Let us denote again by $A_p$ their lifts to $\DD$. Since the proof of \eqref{ineq: pointwise averaged kernel estimate} uses only the local estimate \eqref{ineq: local assumption}, with constants $C$, $S$, and $k$ independent of $p$, it follows that the bound of $\|K_{P_tA_pP_t}\|_{L^\infty(\DD\times \DD)}$ is uniform in $p$. This observation will be used in Section \ref{sec: proof_main_thm} while proving Theorem \ref{thm: main_theorem}.
\end{remark}

\subsection{Square integrability of $K_{\overline{A}_T}$ and Nevo's ergodic theorem}\label{subsec: Hs-norm of avg op.}

Fixing $C(\lambda):=\lambda \tanh({2\pi\lambda})$, we have
\begin{align*}\label{eq: HS norm with (z,b)}
    &\iint_{\DDD\times\DD}|K_{\overline{A}_T}(z,z')|^2\, d\mu(z)\,d\mu(z') =\iiint_{\DDD\times B\times \RR^+} \left|\frac{1}{T} \int_0^Ta_{P_tAP_t}(z, \lambda, b)dt\right|^2 e^{\langle z, b \rangle} C(\lambda)\,d\mu(z)\,db\,d\lambda\\
    &= \iiint_{\DDD\times B\times \RR^+} \left|\frac{1}{T} \int_0^T\frac{h_t(\lambda)}{\sqrt{\cosh{t}}}\int_{w\in B(z, t)}\chi_{t, \sigma}(r)e^{(\frac{1}{2}+i\lambda)(\langle w,b\rangle-\langle z,b\rangle)}a(w, \lambda, b)d\mu(w)\, dt\right|^2 e^{\langle z,b \rangle}C(\lambda)  d\mu(z)\,db\,d\lambda,
\end{align*}
where $r=d(z, w)$. Applying Cauchy-Schwarz in $t$ on the inner integral, we have 
\begin{multline*}
     \left|\frac{1}{T} \int_0^T\frac{h_t(s_\lambda)}{\sqrt{\cosh{t}}}\int_{w\in B(z, t)}\chi_{t, \sigma}(d_\DD(z, w))e^{(\frac{1}{2}+i\lambda)(\langle w,b\rangle-\langle z,b\rangle)}a(w, \lambda, b)d\mu(w)\, dt\right|^2
     \\ \leq \left(\frac{1}{T}\int_0^T |h_t(s_\lambda)|^2dt \right)\left( \frac{1}{T}\int_0^T |I_\lambda(t, z, b)|^2dt\right),
\end{multline*}
where 
\[
I_\lambda(t, z, b)=\frac{1}{\sqrt{\cosh{t}}}\int_{w\in B(z, t)}\chi_{t, \sigma}(d_\DD(z, w))e^{(\frac{1}{2}+i\lambda)(\langle w,b\rangle-\langle z,b\rangle)}a(w, \lambda, b)d\mu(w).
\]
\iffalse
\begin{align*}
    \int_{w\in B(z, t)}e^{(\frac{1}{2}+i\lambda)(\langle w,b\rangle-\langle z,b\rangle)}a(w, \lambda, b)dw&= e^{-\langle z,b\rangle}\int_{w\in B(z, t)}e^{i\lambda(\langle w,b\rangle-\langle z,b\rangle)}a(w, \lambda, b)e^{\langle w,b\rangle}dw \\
    &=e^{-\langle z,b\rangle}\int_{w\in B(z, t)}e^{i\lambda(\langle w,b\rangle-\langle z,b\rangle)}e^{\langle w,b\rangle-\langle w,b'\rangle}a(w, \lambda, b')\delta(b'-b) e^{\langle w,b'\rangle}dwdb' 
\end{align*}
\fi

First, we deal with the integral $I_\lambda(t, z, b)$. For any fixed $z\in \DDD$ and $b\in B$, one can choose a unique $h\in \mathrm{PSU}(1,1)$ such that 
\[z=h\cdot 0;\,\,\,\,\,\,\, b=h\cdot 1,\]
where $0$ denotes the origin of the disk and $1\in B$ is some fixed boundary point, and where $h$ acts on $0, 1$ by a M\"{o}bius transformation.
In what follows we write $I_\lambda(t, h)$ for $I_\lambda(t, z, b)$.

So, having identified $(z, b)\in \DDD\times B$ with a $h\in \mathrm{PSU}(1,1)$, any point $(w,b')$ can be written uniquely as $(w,b'):=h g$,
meaning that $(w, b')=(h  g\cdot 0, h  g\cdot 1)$. We use the ${ANK}$ decomposition of $g$ to write $(w, b')=:=h a_sn_uk_\theta$ where $a_s\in A, n_u\in N, k_\theta\in K$.
The case $b'=b$ corresponds to $\theta=0$~: $(w,b):=  h a_sn_u$.
Accordingly $a(w, \lambda, b')$ can be rewritten as $a_\lambda(h a_sn_uk_\theta)$ and $a(w, \lambda, b) $ as $a_\lambda(h a_sn_u )$.

Now, considering the following identity,
\[
\langle hx, h\xi\rangle - \langle hy, h\xi\rangle = \langle x, \xi\rangle - \langle y, \xi\rangle, 
\]
and choosing $x=g\cdot 0, \, y=0,\, \xi=1$ we get
\[
\langle w, b\rangle - \langle z, b\rangle= \langle hg\cdot 0, h\cdot 1\rangle - \langle h\cdot 0, h\cdot 1\rangle= \langle g\cdot 0, 1\rangle- \langle 0, 1\rangle = \langle g\cdot 0, 1\rangle.
\]

 Recall that every symbol $a(z, \lambda, b)$ can be decomposed into a rotationally invariant part and a part with zero angular mean. More precisely, for fixed $\lambda$,
\[
a(z, \lambda, b)= \left(\frac{1}{2\pi} \int_0^{2\pi} a_\lambda(z, b')db' \right) + \left( a_\lambda(z, b)- \frac{1}{2\pi} \int_0^{2\pi} a_\lambda(z, b')db'\right),
\]
where the first term is independent of $b$, for every $z$. Observe that if the symbol is rotationally invariant, i.e. $a_\lambda(h'\cdot k_\theta)= a_\lambda(h')$ for every $h'$ and $\theta$, we fall into the case treated in \cite{MassonSahlsten17}. Hence, from now on, without loss of generality, we work under the following additional assumption on our symbol~:
\begin{equation}\label{condition A1}
    \int_0^{2\pi}a_\lambda(h'\cdot k_\theta)d\theta=\int_0^{2\pi} a(h'k_\theta, \lambda)=0\,\,\,\,\, \text{for every } \lambda\in\RR^+.
    \tag{A1}
\end{equation}

Let us do a coordinate change on the integral $I_\lambda(t)$ in terms of ${ANK}$ coordinates as follows. 
We call $G_t$ the subset $G_t\subset G$ defined by the condition that for any element $h'\in G$ and any $g\in G_t$, the point $h'gK$ is at distance at most $t$ from $h'K$ in the hyperbolic disc $\DD\sim G/K$.
The set $G_t$ is invariant under the right action of $K$ so we may identify it with a subset $F_t$ of $AN$.

\begin{align}
    I_\lambda(t, z, b)= I_\lambda(t, h)&=\frac{1}{\sqrt{\cosh{t}}}\int_{w\in B(z, t)} \chi_{t, \sigma}(d_\DD(z, w)) e^{(\frac{1}{2}+i\lambda)(\langle w,b\rangle-\langle z,b\rangle)}a(w, \lambda, b)d\mu(w)\\
   \nonumber &= \frac{1}{\sqrt{\cosh{t}}}\int_{F_t}\chi_{t, \sigma}(d_\DD(h\cdot0, h a_sn_u\cdot 0))e^{(\frac{1}{2}+i\lambda)\langle a_sn_u\cdot 0, 1\rangle}a_\lambda(ha_sn_u)dsdu\\
  \nonumber  &= \frac{1}{\sqrt{\cosh{t}}}\int_{G_t} \chi_{t, \sigma}(d_\DD(0, a_sn_u\cdot 0)) e^{(\frac{1}{2}+i\lambda)\langle a_sn_u\cdot 0, 1\rangle}a_\lambda(ha_sn_uk_\theta)\delta(\theta)ds du d\theta\\
    &= \label{e:ANKI}\frac{1}{\sqrt{\cosh{t}}}\int_{G_t} \chi_{t, \sigma}(d_\DD(0, a_sn_u\cdot 0)) e^{(\frac{1}{2}+i\lambda)s}a_\lambda(ha_sn_uk_\theta)\delta(\theta)ds du d\theta.
\end{align}
In the third equality, we have introduced the Dirac mass $\delta$ at $\theta=0$. In the final equality, we have used the fact that $\langle a_sn_u\cdot 0, 1 \rangle= \langle a_s\cdot 0, 1 \rangle$ since $n_u$ moves 0 along a horocycle and $\langle \cdot, b \rangle$ is constant on every horocycle. Then one can explicitly compute that $\langle a_s\cdot0, 1 \rangle=s$.

We now do an integration by parts to suppress the singularity at $\theta=0$. Let us momentarily fix $s$ and $u$ and let $A(\theta)=\chi_{t, \sigma}(d_\DD(0, a_sn_u\cdot 0)) e^{(\frac{1}{2}+i\lambda)s}a_\lambda(ha_sn_uk_\theta)$ and $S(\theta)=\sum_{k\neq 0}\frac{e^{ik\theta}}{k^2}$ so that $S''(\theta)= - \sum_{k\neq 0} e^{ik\theta}= -2\pi \delta(\theta)+1$ (follows from the Fourier expansion of the Dirac-delta). Then,
\begin{align*}
    \int_B A(\theta)\delta(\theta)d\theta = \frac{1}{2\pi}\int_B A(\theta)d\theta -\frac{1}{2\pi}\int_B A(\theta)S''(\theta)d\theta  = \frac{1}{2\pi}\int_B A(\theta)d\theta -\frac{1}{2\pi}\int_B A''(\theta)S(\theta)d\theta
\end{align*}
From assumption \eqref{condition A1}, we have $\int_B a_\lambda(h'\cdot k_\theta)=0$, which implies $\int_B A(\theta) d\theta=0$. So,
\begin{align*}
    I_\lambda(t, z, b)&=\frac{1}{\sqrt{\cosh{t}}} \int_{w\in B(z,t)} \chi_{t, \sigma}(d_\DD(z, w)) e^{(\frac{1}{2}+i\lambda)(\langle w,b\rangle-\langle z,b\rangle)} a(w, \lambda, b)d\mu(w) \\
    &= -\frac{1}{2\pi} \frac{1}{\sqrt{\cosh{t}}}\int_{G_t}\left(\sum_{k\neq 0}\frac{e^{ik\theta}}{k^2}\right)\frac{\pa^2}{\pa \theta^2}\left(\chi_{t, \sigma}(d_\DD(0, a_sn_u\cdot 0))e^{(\frac{1}{2}+i\lambda)s}a_\lambda(ha_sn_uk_\theta)\right)dsdud\theta\\
    &=  -\frac{1}{2\pi} \frac{1}{\sqrt{\cosh{t}}}\int_{G_t}\frac{\pa^2a_\lambda}{\pa \theta^2}(ha_sn_uk_\theta)\, \left(\sum_{k\neq 0}\frac{e^{ik\theta}}{k^2}\right)\chi_{t, \sigma}(d_\DD(0, a_sn_uk_\theta\cdot 0))e^{(\frac{1}{2}+i\lambda)s}\, dsdud\theta.
\end{align*}

\iffalse
For any function $f\in L^2(SX_\Gamma)$, using the identification above, we can express $f$ as a function in $L^2(G)$. 
Let $f=\frac{\pa^2a_\lambda}{\pa \theta^2}(ha_sn_uk_\theta)$ and consider the averaging operator $(\pi(\mu)f)(z, b)$ defined as
\begin{equation}
    \pi(\mu)f(z, b)=  \int_{G}f(g^{-1}(z, b))d\mu(g)= \int_G f(g^{-1}h)d\mu(g),
\end{equation}
where $G=\mathrm{PSU}(1,1)$, $h\sim (z, b)$, and $$d\mu_t(g)=\frac{1}{2\pi}\frac{1}{\sqrt{\cosh{t}}} \chi_{G_t}e^{(\frac{1}{2}+i\lambda)s}(\sum_{k\neq 0}\frac{e^{ik\theta}}{k^2})ds\,du\,d\theta.$$ 
\fi

Define $f(h)=\frac{\pa^2a_\lambda}{\pa \phi^2}(hk_\phi)|_{\phi=0}$  so that 
\[
f(ha_s n_u k_\theta)
=\left.\frac{\pa^2a_\lambda}{\pa \phi^2}
\left(h a_s n_u k_\theta k_\phi\right)\right|_{\phi=0}=
\left.\frac{\pa^2a_\lambda}{\pa \phi^2}\left(h a_s n_u k_{\phi+\theta}\right)\right|_{\phi=0}=
\frac{\pa^2a_\lambda}{\pa \theta^2}\left(h a_s n_u k_\theta\right),
\]
and 
%Let {\color{blue}$f(g)=\frac{\pa^2a_\lambda}{\pa \theta^2}(ha_sn_uk_\theta)|_{\theta=0}$} and 
\begin{align*}
 d\mu_t(g)&=\frac{1}{2\pi} \chi_{\sigma}(a_s n_u k_\theta) e^{(\frac{1}{2}+i\lambda)s}\left(\sum_{k\neq 0}\frac{e^{ik\theta}}{k^2}\right)ds\,du\,d\theta=\beta_t(g) dg
\end{align*}
where $dg= ds\,du\,d\theta$ is the Haar measure on $G$ and $\beta_t(a_s n_u k_\theta)=\frac{1}{2\pi} \chi_{\sigma}(a_s n_u k_\theta)e^{(\frac{1}{2}+i\lambda)s}(\sum_{k\neq 0}\frac{e^{ik\theta}}{k^2})$, with $\chi_{t, \sigma}(g)=0$ whenever $g\notin G_t$.

Then $I_\lambda(t,h)$ can be interpreted as an averaging (or convolution) operator $(\pi(\mu_t)f)(z, b)$ defined as
\begin{align*}
I_\lambda(t,h)&= \frac{1}{\sqrt{\cosh{t}}} (\pi(\mu_t)f)(h)= \frac{1}{\sqrt{\cosh{t}}}\int_G f(hg)d\mu_t(g)\\
&=\frac{1}{\sqrt{\cosh{t}}}\int_G f(hg)\beta_t(g) dg.
\end{align*}

%{\color{red}{We need $\chi_{t, \sigma}$}. In other words replace $\chi_{G_t}$ by $\chi_{t, \sigma}(a_s n_u k_\theta)$ everywhere.}

The papers \cite{Nevo98} and \cite{GorodnikNevo15} provide estimates on averaging operators (also see Section 6 in \cite{MassonSahlsten17})~:
\begin{itemize}
    \item For any $f\in L^2_0(SX_\Gamma):=\{f\in L^2(SX_\Gamma): \int_{SX_\Gamma} f\,dg =0\}$
    \[
    \|\pi(\mu)f\|_2 \leq \|\lambda_G(\mu)\|^{1/n}\|f\|_2,
    \]
    where $\lambda_G(\mu)$ is the convolution operator on $G$ and $n\geq 2$ depends on the principal series present in the decomposition of $L^2_0(SX_\Gamma)$ into irreducible unitary representations of $G$. In our context, $n$ depends only on the spectral gap of the Laplacian.

    \item For $G=\mathrm{PSU}(1,1)\equiv \mathrm{PSL}(2, \RR)$, the convolution operator $\lambda_G(\mu)$ satisfies the following \emph{Kunze-Stein inequality} (see Definition 4 in \cite{Nevo98}): for $F\in L^p(G),~~ 1\leq p< 2$ and $f\in L^2(G)$, 
    \[
    \|F\ast f\|_2\leq C_p \|F\|_p\|f\|_2.
    \]
    This implies that 
\[
\|\lambda_G(\mu)\|^{1/n}\leq C_p^{1/n} \|\beta_t\|^{1/n}_{p}.
\]
\end{itemize}
Then for any $f\in L^2(SX_\Gamma)$, we have the following estimate on the averaging operator, 
\begin{equation}\label{eq: conv rate}
    \left\| \pi(\mu)f -\int_{SX_\Gamma} f \right\|_{2}\leq C_p^{1/n} \|\beta_t\|_p^{1/n}\|f\|_2
\end{equation}

\begin{note}
We point out a distinction in the notation used to represent our $L^2$-norms. Moving forward, for $f\in L^2(SX_\Gamma)$, the notation $\|f\|_2$ is being used to denote the usual $L^2$-norm while $\tnorm{f}_2$ will be used to denoted a volume normalised $L^2$-norm and is defined as $\tnorm{f}^2_2=\frac{1}{\Vol(SX_\Gamma)}\|f\|^2_2$.

%   {\color{red}{Clarify what is $X$. This notation applies to $X=X_\Gamma$ ? Also $X=SX_\Gamma$ ? other spaces ? }}\SSa{In this section, I used it on $SX_\Gamma$. In section 3.4, I see the full norm being used as $\|K_\rho\|_2^2$ while writing the estimates.  Should we change the notation $\|{m_p}_s-m_p\|_{L^2(\RR^+, \lambda\tanh{(2\pi \lambda)}d\lambda)}$ to just $\|{m_p}_s-m_p\|_2$ if we want to follow the same norm-notation throughout or just restrict the notation to $SX_\Gamma$ and leave the rest as it is?}. 
\end{note}

For each fixed $h$, $a_\lambda(h\cdot k_\phi)$ is $2\pi$-periodic and using the assumption \eqref{condition A1}, we have the Fourier expansion $a_\lambda(hk_\phi)=\sum_{m\neq 0} a_m(h)e^{im\phi}$, where $a_m(h)=\int_0^{2\pi}a_\lambda(hk_\phi)e^{-im\phi}d\theta$ up to the normalisation constant. Then $f(h)= \sum_{m\neq0} (-m^2)a_m(h)e^{im\phi}$, and by default, the $m=0$ mode gives $\int_{SX_\Gamma} f=0$. 
%Now, define $\tilde f:= \frac{\pa^2 a}{\pa \theta^2}(h\, n_u\,a_s\,k_\theta)$ and $f\in L^2(SX_\Gamma)$ where $f(z, b):= \tilde f(k_\theta^{-1}a_s^{-1}n_u^{-1}h)\equiv \tilde f (g^{-1}h)$, with $h\sim (z, \theta)$. 

\begin{align*}
   \iint_{\DDD\times B} \left( \frac{1}{T}\int_0^T |I_\lambda(t, z, b)|^2dt\right)  e^{\langle z, b \rangle}d\mu(z)db &= \int_{\Gamma\backslash G} \left(\frac{1}{T}\int_0^T |I_\lambda(t, h)|^2 dt\right)dh\\
   &= \int_{\Gamma\backslash G} \left(\frac{1}{T}\int_0^T \frac{1}{\cosh{t}}|(\pi(\mu_t)f)(h)|^2dt\right)dh\\
   %&= \int_H \left(\frac{1}{T}\int_0^T \left|\int_G \frac{1}{\sqrt{\cosh{t}}}f(g^{-1}h)d\mu_t(g)\,dg\right|^2 dt\right)dh\\
   %&= \int_H \left(\frac{1}{T}\int_0^T \frac{1}{\cosh{t}} \left|\int_G f(g^{-1}h)d\mu_t(g)\,dg\right|^2 dt\right)dh\\
  % &\leq \int_H\left(\frac{\|f\|^2_2}{T}\int_0^T\frac{1}{\cosh{t}}\|\lambda_G(\mu_t)\|^{2/n}dt\right)dh\\
   &\leq \frac{\|f\|^2_2}{T}\int_0^T\frac{1}{\cosh{t}}C_p^{2/n}\|\beta_t\|_{p}^{2/n}dt.
\end{align*}
Note that
\begin{align*}
    \|\beta_t\|_{p}^p=\frac{1}{(2\pi)^p}\int_G e^{ps/2}|\chi_{t, \sigma}(g)|^p \left|\sum_{k\neq 0}\frac{e^{ik\theta}}{k^2}\right|^pdg \lesssim \frac{1}{(2\pi)^p}\int_{G_t} e^{ps/2}dg.
\end{align*}
Since, $\cosh{d(0, a_sn_u\cdot 0)}=\frac{u^2e^s+2\cosh{s}}{2}$, \text{ and } $g\in G_t$ iff $\cosh{d(0, a_sn_u\cdot 0)}\leq \cosh t$, we have $|u|\leq \sqrt{\frac{2(\cosh{t}-\cosh{s})}{e^s}}$ and $0\leq s\leq t$. So, for any $p\geq1$,
\begin{align*}
    \|\beta_t\|_{p}^p &\lesssim \frac{1}{(2\pi)^p}\int_{G_t} e^{ps/2}dg=\frac{1}{(2\pi)^{p-1}}\int_{F_t}e^{ps/2}duds\sim \int_0^t\int_{|u|\leq \sqrt{\frac{2(\cosh{t}-\cosh{s})}{e^s}}}e^{ps/2}duds\\
    &\sim \int_0^t e^{ps/2}{e^{-s/2}}\left(\sqrt{\cosh{t}-\cosh{s}}\right)\,ds \lesssim \int_0^t e^{ps/2}{e^{-s/2}}\left(\sqrt{e^t-e^s}\right)\,ds\\
    &=\int_0^t e^{ps/2}\sqrt{e^{t-s}-1}= e^{pt/2}\int_0^t  e^{(s-t)\frac{p-1}{2}}\sqrt{1-e^{s-t}}ds,\,\,\,\,\,(\text{taking } r=t-s)\\
    &=e^{pt/2}\int_0^t  e^{(-r)\frac{p-1}{2}}\sqrt{1-e^{-r}}dr \leq e^{pt/2}\int_0^t  e^{(-r)\frac{p-1}{2}}dr=e^{pt/2} \frac{2}{p-1}\left(1-e^{-\frac{t(p-1)}{2}}\right)\lesssim e^{pt/2},
\end{align*}
where the the penultimate term above, $e^-\frac{t(p-1)}{2}\leq 1$ since $p\geq 1$. For our purpose, fixing any $p$ (say $p=3/2$) suffices, which implies
\begin{align*}
   \iint_{\DDD\times B} \left( \frac{1}{T}\int_0^T |I_\lambda(t, z, b)|^2dt\right)  e^{\langle z, b \rangle}d\mu(z)db    &\leq C_{3/2}^{2/n}\|f\|^2_2 \left(\frac{1}{T}\int_0^T\frac{1}{\cosh{t}}\|\beta_t\|_{3/2}^{2/n}dt\right)\\
   &\lesssim \|f\|^2_2\left(\frac{1}{T}\int_0^T\frac{1}{\cosh{t}}e^{t/n}dt\right)\\
   &\leq \|f\|^2_2\left(\frac{1}{T}\int_0^Te^{-t}e^{t/n}dt\right)\\
   &= \|f\|^2_2\left(\frac{1}{T}\int_0^Te^{t(-1+1/n)}dt\right).
   %&={\color{red}\|f\|^2_2\Vol(SX_\Gamma)\left(\frac{1}{T}\int_0^Te^{t(-1+1/n)}dt\right).}
\end{align*}
Since $n>1$, $\frac{1}{n}-1<0$, and we have
\begin{align*}
   \iint_{\DDD\times B} \left( \frac{1}{T}\int_0^T |I_\lambda(t, z, b)|^2dt\right)  e^{\langle z, b \rangle}d\mu(z)db    &\lesssim \|f\|^2_2\left(\frac{1}{T}\int_0^Te^{t(-1+1/n)}dt\right) \\
   &\leq \Vol(SX_\Gamma)  \tnorm{\frac{\pa^2a_\lambda}{\pa^2\theta}}^2_2\frac{1}{T(1-1/n)}.
\end{align*}
Recall for further use that $n$ depends only on the spectral gap of the Laplacian on $X_\Gamma$.
Now we are ready to estimate our required integral.
\begin{align*}
  &\iint_{\DDD\times\DD}|K_{\overline{A}_T}(z,w)|^2\, d\mu(z)\,d\mu(w) \\
  &\leq\int_{\RR^+} \left(\iint_{\DDD\times B} \left(\frac{1}{T}\int_0^T |h_t(\lambda)|^2dt \right)\left( \frac{1}{T}\int_0^T |I_\lambda(t, z, b)|^2dt\right)  e^{\langle z, b \rangle}d\mu(z)db \right)  \lambda \tanh({2\pi\lambda}) d\lambda\\
  &=\int_{\RR^+} \left(\frac{1}{T}\int_0^T |h_t(\lambda)|^2dt \iint_{\DDD\times B} \left( \frac{1}{T}\int_0^T |I_\lambda(t, z, b)|^2dt\right)  e^{\langle z, b \rangle}d\mu(z)db \right)  \lambda \tanh({2\pi\lambda}) d\lambda\\
  &\lesssim \int_{\RR^+} \tnorm{\frac{\pa^2a_\lambda}{\pa^2\theta}}^2_2\frac{\Vol(SX_\Gamma) }{(1-1/n)} \frac{1}{T^2}\int_0^T |h_t(\lambda)|^2 \lambda \tanh({2\pi\lambda}) \,dt d\lambda\\
  &\leq C\frac{\Vol(SX_\Gamma) }{(1-1/n)} \frac{1}{T^2} \int_{\RR^+} \int_0^T |h_t(\lambda)|^2 \lambda \tanh({2\pi\lambda}) \,dt d\lambda, \,\,\,\,\, \text{where }~~~~~ C=\sup_{\lambda\in \RR^+}\tnorm{\frac{\pa^2a_\lambda}{\pa^2\theta}}^2_2.
\end{align*}

Our assumptions do not imply that $C=\sup_{\lambda\in \RR^+}\|\pa^2_\theta a_\lambda\|^2_{L^2(SX_\Gamma)}$ is finite (in the case of a differential operator, for instance, the symbol $a$ is polynomial in $\lambda$). Using spectral cut-offs, as described in Subsection \ref{subsec: symbol cut-offs}, we want to restrict $\lambda\in I$, which makes $C$ a finite quantity. The following subsection provides details regarding this.

%\subsection{Symbol cut-offs and spectral projections} 

%Let $\rho(\lambda)$ be a compactly supported smooth function and consider the pseudo-differential operator $\mathrm{Op}({a}\rho)$ where $a(z, \lambda, b)$ denotes a symbol on $\DD$. Since $\rho$ depends only on $\lambda$, we have the following

\subsection{Square-integrability of kernel of $\overline{A'}_T$}\label{subsec: square-integ of trunc avg op}

Starting with a $\Gamma$-invariant symbol $a(z, \lambda, b)$, let $A:=\Op(a)$ be defined as in \eqref{eq: Op(a) with inverse Helgason transform}, and assume that the kernel $K_A(z, w)=0$ whenever $d(z, w)\geq S$. If $K_{P_tAP_t}(z,z')\neq 0$, then there exists $w, w'$ such that 
\[
d(z, w)\leq t,\,\,\,\,\, d(w,w')\leq S,\,\,\,\,\, d(w',z')\leq t.
\]
So, $K_{P_tAP_t}(z,z')\neq 0$ whenever
\[
d(z,z')\leq d(z,w)+d(w,w')+d(w',z')=2t+S.
\]
In other words, 
\begin{align}\label{e:Dt}
d(z, z')>2t+S =: S_t~(\text{say}) \implies K_{P_tAP_t}(z,z')= 0.
\end{align}
This implies finite propagation for the operator $\overline{A}_T:=\frac{1}{T}\int_0^TP_tAP_tdt$ corresponding to the symbol $\overline{a}_T=\frac{1}{T}\int_0^Ta_{P_tAP_t}dt$. Now, we define
\[
\overline{A'}_Tu(z):= \mathrm{Op}(\overline{a}_{T}\rho)u(z)= \frac{1}{2\pi}\iint_{\RR^+\times B}\int_{\DD} a(z, \lambda, b) \rho(\lambda) e^{(\frac{1}{2}+i\lambda)\langle z,b\rangle}e^{(\frac{1}{2}-i\lambda)\langle w, b\rangle} u(w)\lambda\tanh({2\pi\lambda}) d\mu(w) \, d\lambda\, db, 
\]
where $\rho(\lambda)$ is a smooth bump function supported on $I'\supset I$ and $\rho(\lambda)=1$ when $\lambda\in I$. Since $K_A$ is $\Gamma$-invariant, we have that $K_{\overline{A}_T}$ is also $\Gamma$-invariant which in turn implies that $K_{\overline{A'}_T}$ is $\Gamma$-invariant. We point out that even though $K_{\overline{A}_T}$ has finite propagation, it does not imply the same for $K_{\overline{A'}_T}$. As a consequence, there might be an issue to pass $K_{\overline{A}_T}$ on to the quotient $X_\Gamma$. To bypass this issue, we use truncated kernels. According to Lemma \ref{lem: truncated kernel HS-norm estimate}, we have 
\begin{multline}\label{ineq: Op_Gamma, r(a_T rho) HS-estimate}
    \|\Op_{\Gamma, r}(\overline{a}_T\rho)\|^2_{\mathrm{HS}(X_\Gamma)}\leq \int_{z\in \DDD}\int_{w\in \DD} |K_{\overline{A'}_T}(z,w)|^2  d\mu(w)d\mu(z) + \\
   \frac{e^{2r}}{l_{X_\Gamma}}\operatorname{Vol}\{z\in X:\operatorname{InjRad}_X(z)<r\}\,\,\sup_{(z,w)\in \DDD\times \mathbb{\DD}}|K_{\overline{A'}_T}(z,w)|^2,
\end{multline}
Now, estimating the integral on the right, we get
\begin{align*}
    \iint_{\DDD\times\DD}|&K_{\overline{A'}_T}(z,w)|^2\, d\mu(z)\,d\mu(w) \\
    %\iiint_{\DDD\times B\times \RR^+} \left|\frac{1}{T} \int_0^Ta_{P_tAP_t}(z, \lambda, b)\rho(\lambda)dt\right|^2 e^{\langle z, b \rangle} \lambda \tanh({2\pi\lambda}) \,d\mu(z)\,db\,d\lambda\\
    &= \int_{\lambda\in I'}\rho^2(\lambda)\iint_{\DDD\times B} \left|\frac{1}{T} \int_0^Ta_{P_tAP_t}(z, \lambda, b)dt\right|^2 e^{\langle z, b \rangle} \lambda \tanh({2\pi\lambda}) \,d\mu(z)\,db\,d\lambda.\\
    &\leq \int_{\lambda\in I'}\iint_{\DDD\times B} \left|\frac{1}{T} \int_0^Ta_{P_tAP_t}(z, \lambda, b)dt\right|^2 e^{\langle z, b \rangle} \lambda \tanh({2\pi\lambda}) \,d\mu(z)\,db\,d\lambda.
\end{align*}
Following a similar computation as the previous subsection with the integral over $\lambda\in \RR_+$ replaced by an integral over $\lambda\in I'$, we get
\begin{align*}
  \iint_{\DDD\times\DD}&|K_{\overline{A'}_T}(z,w)|^2\, d\mu(z)\,d\mu(w) \\
  %\leq \int_{\lambda\in I'}\iint_{\DDD\times B} \left|\frac{1}{T} \int_0^Ta_{P_tAP_t}(z, \lambda, b)dt\right|^2 e^{\langle z, b \rangle} \lambda \tanh({2\pi\lambda}) \,d\mu(z)\,db\,d\lambda.\\
  %&\leq\int_{\lambda\in I'} \left(\iint_{\DDD\times B} \left(\frac{1}{T}\int_0^T |h_t(\lambda)|^2dt \right)\left( \frac{1}{T}\int_0^T |I_\lambda(t)|^2dt\right)  e^{\langle z, b \rangle}d\mu(z)db \right)  \lambda \tanh({2\pi\lambda}) d\lambda\\
  &\leq\int_{\lambda\in I'} \left(\frac{1}{T}\int_0^T |h_t(\lambda)|^2dt \iint_{\DDD\times B} \left( \frac{1}{T}\int_0^T |I_\lambda(t, z, b)|^2dt\right)  e^{\langle z, b \rangle}d\mu(z)db \right)  \lambda \tanh({2\pi\lambda}) d\lambda\\
  &\lesssim \int_{\lambda\in I'} \left\|\frac{\pa^2a_\lambda}{\pa^2\theta}\right\|^2_{L^2(SX_\Gamma)}\frac{\Vol(SX_\Gamma) }{(1-1/n)} \frac{1}{T^2}\int_0^T |h_t(\lambda)|^2 \lambda \tanh({2\pi\lambda}) \,dt d\lambda\\
  &\leq C\frac{\Vol(SX_\Gamma) }{(1-1/n)} \frac{1}{T^2} \int_{\lambda\in I'} \int_0^T |h_t(\lambda)|^2 \lambda \tanh({2\pi\lambda}) \,dt d\lambda, \,\,\,\,\, \text{where }~~~~~ C=\sup_{\lambda\in I'}\tnorm{\pa^2_\theta a_\lambda}^2_2.\\
\end{align*}

%Since $P_t$ is a radial integral operator, we have that the action of the operator $P_t$ on a generalised eigenfunction $\psi_\lambda$ with eigenvalue $\lambda$ is given by 
%\[
%P_t\psi_\lambda(z)=\int_{\DD} K_t(z,w)e^{(\frac{1}{2}+i\lambda)\langle w,b\rangle}d\mu(w)=h_t(\lambda)e^{(\frac{1}{2}+i\lambda)\langle z,b\rangle}=h_t(\lambda)\psi_\lambda(z).
%\]
Recall that $h_t$ is the Selberg transform of the kernel $k_t$. In other words, it is the Fourier transform of the radial function $z\mapsto k_t(d(0, z))=\frac{1}{\sqrt{\cosh{t}}}\chi_{t, \sigma}(d(0, z))$. Using Plancherel's theorem (up to some normalisation factor), and the expression of the volume form in hyperbolic polar coordinates, we have that
\[
\int_{\RR^+}|h_t(\lambda)|^2\lambda\tanh{(2\pi\lambda)}d\lambda= \int_0^\infty |k_t(r)|^2 \sinh{r}dr=\frac{1}{\cosh{t}}\int_0^t |\chi_{t, \sigma}(r)|^2\sinh{r}dr\leq \frac{\cosh{t}-1}{\cosh{t}}.
\]

% Using \eqref{eq: radial_op_Selberg_relation}, we have
%Choosing $z=0$ and using the fact that $\langle 0, b\rangle=0$ for any $b\in B$, we have 
%\begin{align*}
  %  h_t(\lambda)&= h_t(\lambda)e^{(\frac{1}{2}+i\lambda)\langle 0, b \rangle}= \int_\DD k_t(d(0, w))e^{(\frac{1}{2}+i\lambda)\langle w, b \rangle}d\mu(w)=\frac{1}{\sqrt{\cosh{t}}}\int_{B(0,t)}\chi_{t, \sigma}(d(0, w))e^{(\frac{1}{2}+i\lambda)\langle w, b \rangle}d\mu(w)\\
   %  &= \frac{1}{\sqrt{\cosh{t}}}\int_0^{\rho_t}\int_0^{2\pi}\chi_{t, \sigma}(d(0, \rho))e^{(\frac{1}{2}+i\lambda)\langle \rho e^{i\theta}, e^{i\theta_b} \rangle}\frac{4\rho\, d\rho\, d\theta}{(1-\rho^2)^2}\,\,\,\,\,\,\,(\text{where $\rho_t= \tanh{\frac{t}{2}}$})
   % &\leq  \frac{1}{\sqrt{\cosh{t}}}\int_0^{\rho_t}\left(\int_0^{2\pi}e^{\frac{1}{2}\langle \rho e^{i\theta}, e^{i\theta_b} \rangle}\, d\theta\right)\frac{4\rho\, d\rho}{(1-\rho^2)^2}\\
%\end{align*}
%Let $r=\dist_\DD(0, z)= 2\arctanh{(\rho)}$, then $dr= \frac{2}{1-\rho^2}d\rho$ and $\sinh{r}=\frac{2\tanh{(r/2)}}{1-\tanh^2{(r/2)}}=\frac{2\rho}{1-\rho^2}$ gives 
%\[
%\sinh{r}dr=\frac{4\rho\,d\rho}%{(1-\rho^2)^2},
%\]
%which implies
%\[
%h_t(\lambda)=\frac{1}{\sqrt{\cosh{t}}}\int_0^{t}\int_0^{2\pi}\chi_{t, \sigma}(r)e^{(\frac{1}{2}+i\lambda)\langle (r,\theta), b) \rangle}\sinh{r\,dr\,d\theta} = {\color{red}\widehat{k_t}(-\lambda, b)}\,\,\,\, \text{(independent of } b).
%\]

Using this in the above estimate, we have

\begin{align*}
  \iint_{\DDD\times\DD}|K_{\overline{A'}_T}(z,w)|^2\, dz\,dw 
  &\leq C \frac{\Vol(SX_\Gamma) }{(1-1/n)} \frac{1}{T^2} \int_{\lambda\in I'} \int_0^T |h_t(\lambda)|^2 \lambda \tanh({2\pi\lambda}) \,dt d\lambda\\
  &\leq  C \frac{\Vol(SX_\Gamma) }{(1-1/n)} \frac{1}{T^2} \int_0^T \frac{\cosh t-1}{\cosh{t}}dt\leq C \frac{\Vol(SX_\Gamma) }{(1-1/n)} \frac{1}{T^2} \int_0^Tdt \\
  &= C \frac{\Vol(SX_\Gamma) }{(1-1/n)} \frac{1}{T}. 
\end{align*}

\subsection{Estimating the `quantum variance'}\label{subsec: estimating QV}
\iffalse
For $\rho\in C_c^\infty(\RR^+)$, let $\Op(\rho)$ denote the operator corresponding to the symbol $a(z, \lambda, b)=\rho(\lambda)$ (independent of $(z, b)$), and its kernel is given by
\begin{equation}\label{eq:Krho}
K_\rho(u,w)=\frac1{2\pi}\int_{\RR^+\times B}\rho(\lambda) e^{(\frac{1}{2}+i\lambda)\langle u,b\rangle}e^{(\frac{1}{2}-i\lambda)\langle w,b\rangle}\lambda\tanh(2\pi\lambda)\,db\,d\lambda
\end{equation}
\fi

Replacing $a$ with $\overline{a}_T=\frac{1}{T}\int_0^T a_{P_tAP_t}dt$ in \eqref{eq: truncated periodised Op approximation} where $A=\Op_\Gamma(a)$, and considering $\rho(\lambda)$ a smooth bump function supported on $I'\supset I$ and $\rho(\lambda)=1$ when $\lambda\in I$, we use Proposition \ref{thm: truncated periodised Op approximation} 
\begin{align}\nonumber
   &\quad\sum_{j=0}^\infty \left|\langle \psi_j, \Op_{\Gamma, r}(\overline{a}_T\rho)\psi_j \rangle\right|^2 = \sum_{j=0}^\infty \left|\langle \psi_j, \Op_\Gamma(\overline{a}_T)\rho(\lambda_j)\psi_j \rangle+\langle \psi_j, \Op_\Gamma(\overline{a}_T)E_{r, \rho}(\lambda_j)\psi_j \rangle+\langle \psi_j, R_r(\overline{a}_T, \rho)\psi_j \rangle\right|^2\\
\label{e:lowerb}  &\geq \frac{1}{2} \left( \sum_{j=0}^\infty  |\langle \psi_j, \Op_\Gamma(\overline{a}_T)\rho(\lambda_j)\psi_j \rangle|^2 \right) - \sum_{j=0}^\infty|\langle \psi_j, \Op_\Gamma(\overline{a}_T)E_{r,\rho}(\lambda_j)\psi_j \rangle+\langle \psi_j, R_r(\overline{a}_T, \rho)\psi_j \rangle|^2\\
\nonumber    &\geq \frac{1}{2} \left( \sum_{j=0}^\infty  |\langle \psi_j, \Op_\Gamma(\overline{a}_T)\rho(\lambda_j)\psi_j \rangle|^2 \right)- 2\left( \sum_{j=0}^\infty|\langle \psi_j, \Op_\Gamma(\overline{a}_T)E_{r,\rho}(\lambda_j)\psi_j \rangle|^2+\sum_{j=0}^\infty|\langle \psi_j, R_r(\overline{a}_T, \rho)\psi_j \rangle|^2\right),
\end{align}
where the inequality follows from the fact that $\frac{1}{2}|x|^2-|y|^2 \leq |x+y|^2\leq 2(|x|^2+|y|^2)$. Then
\begin{align*}
    &\quad\sum_{j=0}^\infty  |\langle \psi_j, \Op_\Gamma(\overline{a}_T)\rho(\lambda_j)\psi_j \rangle|^2 \\
    &\leq 2 \sum_{j=0}^\infty \left|\langle \psi_j, \Op_{\Gamma, r}(\overline{a}_T\rho)\psi_j \rangle\right|^2 +4\sum_{j=0}^\infty(|\langle \psi_j, \Op_\Gamma(\overline{a}_T)E_{r, \rho}(\lambda_j)\psi_j \rangle|^2+|\langle \psi_j, R_r(\overline{a}_T, \rho)\psi_j \rangle|^2)\\
    &\leq 2\|\Op_{\Gamma, r}(\overline{a}_T\rho)\|^2_{\mathrm{HS}(X_\Gamma)}+ 4(\| \Op_\Gamma(\overline{a}_T)E_{r, \rho}(\lambda_j)\|^2_{\mathrm{HS}(X_\Gamma)}+\|R_r(\overline{a}_T, \rho)\|^2_{\mathrm{HS}(X_\Gamma)})
\end{align*}
Moreover, as mentioned earlier, the action of $P_t$ on any eigenfunction $\psi_j$ with eigenvalue $\nu_j$ of $X_\Gamma$ is given by 
$$P_t\psi_{\nu_j}=h_t(\lambda_j)\psi_{\nu_j},$$
where $\nu_j=\frac{1}{4}+\lambda_j^2$. For some interval ${J'}$, $\nu_j\in {J'}$ whenever $\lambda_j\in I'$. %Slightly abusing the notation, we shall call $\tilde{I'}$ as $I'$.  
Then note that
\begin{align*}
    &\quad\,\,\,\sum_{j=0}^\infty |\langle \psi_j, \Op_\Gamma(\overline{a}_T)\rho(\lambda_j)\psi_j \rangle| = \sum_{j: \nu_j\in J'} |\langle \psi_j, \Op_\Gamma(\overline{a}_T)\rho(\lambda_j)\psi_j \rangle| \geq \sum_{j: \nu_j\in J} |\langle \psi_j, \Op_\Gamma(\overline{a}_T)\rho(\lambda_j)\psi_j \rangle|  \\
    &=\,\sum_{j:\nu_j\in J} \left|\left\langle \psi_j, \left(\frac{1}{T}\int_0^T P_tAP_tdt\right)\psi_j \right\rangle\right|^2
   % =  \sum_{j:\nu_j\in I} \left|\left\langle \psi_j, \left(\frac{1}{T}\int_0^T P_tAP_t \psi_j dt\right)\right\rangle\right|^2\\
  %  &=  \sum_{j:\nu_j\in I} \left|\left\langle \psi_j, \left(\frac{1}{T}\int_0^T P_tAh_t(\lambda_j)\psi_j dt\right)\right\rangle\right|^2
  % = \sum_{j:\nu_j\in I} \left|\frac{1}{T}\int_0^T\left\langle \psi_j,  h_t(\lambda_j)P_tA\psi_j \right\rangle dt\right|^2\\
   % &= \sum_{j:\nu_j\in I} \left|\frac{1}{T}\int_0^T\left\langle P_t\psi_j,  h_t(\lambda_j)A\psi_j \right\rangle dt\right|^2
    =\sum_{j:\nu_j\in J} \left|\frac{1}{T}\int_0^T\left\langle h_t(\lambda_j)\psi_j,  h_t(\lambda_j)A\psi_j \right\rangle dt\right|^2\\
    &= \sum_{j:\nu_j\in J} \left|\frac{1}{T}\int_0^T h_t(\lambda_j)^2 dt\right|^2 |\langle \psi_j,  A\psi_j \rangle|^2.
\end{align*}
From Proposition \ref{prop: MS Prop 2}, we know that 
$$\inf_{j:\nu_j\in J}\left|\frac{1}{T}\int_0^T h_t(\lambda_j)^2 dt\right|^2\geq C_I^2,$$
which implies
%Then, under the assumption that $\int_{SX_\Gamma}a_\lambda(z, b)=0$ for every $\lambda\in \RR^+$, from the above Hilbert-Schmidt norm estimate \eqref{ineq: LMS- PDO HS norm}, we have following estimate of the quantum variance. 
\begin{align*}
 &\sum_{j:\nu_j\in J}  |\langle \psi_j,  A\psi_j \rangle|^2 \leq \frac{1}{C_I^2} \sum_{j=0}^\infty |\langle \psi_j, \Op_\Gamma(\overline{a}_T)\rho(\lambda_j)\psi_j \rangle|\\
 & \lesssim_I \|\Op_{\Gamma, r}(\overline{a}_T\rho)\|^2_{\mathrm{HS}}+ 2\left(\| \Op_\Gamma(\overline{a}_T)E_{r, \rho}(\lambda_j)\|^2_{\mathrm{HS}}+\|R_r(\overline{a}_T, \rho)\|^2_{\mathrm{HS}}\right)\\
 &\leq \|\Op_{\Gamma, r}(\overline{a}_T\rho)\|^2_{\mathrm{HS}}+ 2\left(\| \Op_\Gamma(\overline{a}_T)\|^2_{\mathrm{HS}}|E_{r, \rho}(\lambda_j)|+\|R_r(\overline{a}_T, \rho)\|^2_{\mathrm{HS}}\right)\\
  &\leq \sup_{\lambda\in I'}\|\pa^2_\theta a_\lambda\|_{L^2(SX_\Gamma)} \frac{\Vol(SX_\Gamma) }{(1-1/n)} \frac{1}{T} + 
   \frac{e^{2r}}{l_{X_\Gamma}}\operatorname{Vol}\{z\in X_\Gamma:\operatorname{InjRad}_{X_\Gamma}(z)<r\}\sup_{(z,w)\in \DDD\times \mathbb{\DD}}|K_{\overline{a}_T\rho}(z,w)|^2\\
   &~~~~ +2\|\overline{A}_T\|^2_{\mathrm{HS}}\mathcal{O}\left(\frac{1}{(1+r)^2}\right)+ 2\left(\frac{S_T}{r}\right)^2\left(\sup_{(z,w)\in\DDD\times\DD}|K_{\overline{a}_T}(z,w)|^2\right)\left(\int_0^\infty|\rho(\lambda)|^2\lambda\tanh(2\pi\lambda)d\lambda\right)\\ 
   &\qquad\qquad\qquad\qquad\qquad\qquad\qquad \times e^{2S_T}\left(\Vol(X_\Gamma)+\frac{e^{(r+S_T)}}{l_{X_\Gamma}}\Vol\{z\in X_\Gamma:\InjRad_{X_\Gamma}(z)<r+S_T\}\right)
\end{align*}
where $K_{\overline{a}_T}$ is supported within a distance $S_T=2T+S$ from the diagonal. Using Cauchy-Schwarz, 
\[
|K_{\overline{a}_T\rho}(z,w)|^2\leq \int_\DD|K_{\overline{a}_T}(z,y)|^2d\mu(y)\int_{B(z, S_T)} |K_\rho(y, w)|^2d\mu(y)\leq C_\rho e^{2S_T}\sup_{\DDD\times \DD}|K_{\overline{a}_T}|^2,
\]
where $C_\rho$ is obtained using Plancherel identity and the fact that $K_\rho$ is radial and is the inverse Selberg transform of $\rho\in C_c^\infty(\RR^+)$.

Then, dividing both sides of the above estimate by $\Vol(X_\Gamma)$, we have
\begin{multline*}
    \frac{1}{\Vol(X_\Gamma)} \sum_{\nu_j\in J}\left|\langle \psi_j,  A\psi_j \rangle\right|^2 \\
    \lesssim_I  \frac{\sup_{\lambda\in I'}\tnorm{\pa^2_\theta a_\lambda}^2_2}{(1-1/n)} \frac{1}{T} + \frac{e^{2r+2S_T}}{l_{X_\Gamma}}\frac{\Vol\{z\in X_\Gamma:\InjRad_{X_\Gamma}(z)<r\}}{\Vol(X_\Gamma)}C_\rho\sup_{\DDD\times \DD}|K_{\overline{a}_T}|^2\\
    +2\left(\frac{S_T}{r}\right)^2\sup_{\DDD\times \DD}|K_{\overline{a}_T}|^2 \|K_\rho\|^2_{2}\frac{e^{(r+2S_T)}}{l_{X_\Gamma}}\frac{\Vol\{z\in X_\Gamma:\InjRad_{X_\Gamma}(z)<r+S_T\}}{\Vol(X_\Gamma)}\\
   +2\left(\frac{S_T}{r}\right)^2e^{2S_T}\sup_{\DDD\times \DD}|K_{\overline{a}_T}|^2 \|K_\rho\|^2_{2}  +\frac{2}{\Vol(X_\Gamma)}\|\overline{A}_T\|^2_{\mathrm{HS}}\mathcal{O}\left(\frac{1}{(1+r)^2}\right)
\end{multline*}
Finally, combining the second and third terms on the right-hand side of the above inequality, we have 
\begin{multline}\label{ineq: QV estimate_zero-mean symbol}
    \frac{1}{\Vol(X_\Gamma)} \sum_{\nu_j\in J}\left|\langle \psi_j,  A\psi_j \rangle\right|^2  
    \lesssim_I  \frac{\sup_{\lambda\in I'}\tnorm{\pa^2_\theta a_\lambda}^2_2}{(1-1/n)} \frac{1}{T}\\ + 
    2\left(1+\left(\frac{S_T}{r}\right)^2\right)C'_\rho\sup_{\DDD\times \DD}|K_{\overline{a}_T}|^2 \frac{e^{2(r+S_T)}}{l_{X_\Gamma}}\frac{\Vol\{z\in X_\Gamma:\InjRad_{X_\Gamma}(z)<r+S_T\}}{\Vol(X_\Gamma)}\\
   +2\left(\frac{S_T}{r}\right)^2e^{2S_T}\sup_{\DDD\times \DD}|K_{\overline{a}_T}|^2 \|K_\rho\|^2_{2}  +\frac{2}{\Vol(X_\Gamma)}\|\overline{A}_T\|^2_{\mathrm{HS}}\mathcal{O}\left(\frac{1}{(1+r)^2}\right),
\end{multline}
where $C_\rho'= C_\rho+ \|K_\rho\|^2_2$. The above gives an estimate on the quantum variance under the assumption that $\int_{SX_\Gamma}a_\lambda(z, b)=0$.

\section{Proof of the main theorem}\label{sec: proof_main_thm}

Consider a sequence of closed hyperbolic surfaces $\{X_p\}_{p\in \NN}$ with orthonormal eigendata $(\nu_{j,p}, \psi_{j,p})$ corresponding to the Laplacian $-\Delta_{X_p}$. We assume that
    \begin{enumerate}
    \item $X_p:=\Gamma_p\backslash\DD$ converges in the sense of Benjamini-Schramm to $\DD$.
    \item The injectivity radii are uniformly bounded below by $l_{\min}$, i.e., $\InjRad(X_p)\geq l_{\min}$ for all $p\in \NN$ for some $l_{\min}>0$.
    \item There is a uniform lower bound of the spectral gap of the Laplacian on $X_p$, i.e. there exists some $\eta$ such that $\nu_{1,p}\geq \eta$ for all $p$. 
    \end{enumerate}

 Recall the condition for $\mathrm{BS}$-convergence
\[
    \lim_{p\to \infty} \frac{\Vol(\{z\in X_p: \InjRad_{X_p}(z)<R\})}{\Vol(X_p)}=0 \qquad\text{ for every } R.
\]
\iffalse
is equivalent to saying that there exists $R_m\to \infty$ and $\alpha_m\to 0$ such that 
\[
 \frac{\Vol(\{z\in X_n: \InjRad_{X_m}(z)<R_m\})}{\Vol(X_m)}\leq \alpha_m.
\]
\fi
%Consider a sequence $\{r_m\}\to\infty$ (choice of $r_m$ determined later) as $m\to \infty$. 
Since $1/l_{X_p}\leq 1/l_{\min}$, and the factor $(1-1/n)$ depends only on the spectral gap which has a uniform lower bound, we know from Subsection \ref{subsec: estimating QV} that for each $p$ and fixed $T$, under the assumption $\int_{SX_p} {a_p}_\lambda(z, b)=0$, 
\begin{multline*}
   \frac{1}{\Vol(X_p)} \sum_{\lambda_j\in I}\left|\langle \psi_j,  A_p\psi_j \rangle\right|^2 
    \lesssim_I \frac{\sup_{\lambda\in I'}\tnorm{\pa^2_\theta {a_p}_\lambda}^2_2}{\varrho(\eta)} \frac{1}{T}\\
     +  \left(1+\left(\frac{S_T}{r}\right)^2\right)C'_\rho\|K^2_{\overline{a_p}_T}\|_\infty \frac{e^{2(r+S_t)}}{l_{\min}}\frac{\Vol\{z\in X_p:\InjRad_{X_p}(z)<r+S_T\}}{\Vol(X_p)}\\
   +\left(\frac{S_T}{r}\right)^2e^{2S_T}\|K_{\overline{a_p}_T}^2\|_\infty \|K_\rho\|^2_{2}  +\frac{1}{\Vol(X_p)}\|\overline{A_p}_T\|^2_{\mathrm{HS}}\mathcal{O}\left(\frac{1}{(1+r)^2}\right).
\end{multline*}

\subsection{Quantum variance estimate for general $A_p$ (with truncation)} We now indicate how to obtain the quantum variance estimate for a general $A_p=\Op_\Gamma(a_p)$ (without the  ${a_p}_\lambda$ mean-zero assumption). Define 
\[
M_p(\lambda)= \frac{1}{\Vol(SX_p)} \int_{SX_p}a_p(z,b, \lambda)e^{\langle z, b\rangle}d\mu(z)db.
\]
Let $m_p(\lambda)\in C_c^\infty(\RR^+)$ with $\supp m_p=I'$ (same as the $\supp \rho$ in Subsection \ref{subsec: square-integ of trunc avg op}) such that $m_p(\lambda)=M_p(\lambda)$ for any $\lambda\in I$. In particular, choose $m_p(\lambda)=\beta(\lambda)M_p(\lambda)$ with $\beta\in C_c^\infty(I')$ a non-increasing bump function and $\beta(\lambda)=1$ for $\lambda\in I$.

Since the kernel $K_{m_p}$ of $\Op(m_p):= m_p\left(\sqrt{-\Delta_\DD-\frac{1}{4}}\right)$ is not generally compactly supported, we have the previously discussed summability issues when defining the kernel of $\Op_\Gamma(m_p)$. So, we again work with the truncated operators. Fix $s>0$ and define
\[
K_{m_p, s}(z, w):=K_{m_p}(z, w)\chi\left(\frac{d(z, w)}{s}\right).
\]
Then the operator $\Op_s(m_p)$ defined with the kernel $K_{m_p, s}$ is supported at a distance $s$ from the diagonal. This allows us to push the kernel onto the quotient $X_p$ and define $\Op_{\Gamma_p, s}(m_p)$ whose kernel is $K^{\Gamma_p}_{m_p,s}$ defined as 
\[
K^{\Gamma_p}_{m_p, s}(z,w):=\sum_{\gamma\in\Gamma}K_{m_p}(z,\gamma w)
\chi\left(\frac{d(z,\gamma w)}{s}\right),
\]
with $z, w\in \DDD$. For notational convenience, we temporarily drop the $p$ suffix and use the notations on $X_\Gamma$ instead, but we shall restore it when writing the final estimate.

Note that $\Op_s(m)$ is a radial operator and hence can be written as a spectral multiplier using the spherical transform. Choose the function $m_s$ such that 
\[
\Op_s(m)= m_s\left(\sqrt{-\Delta_\DD-\frac{1}{4}}\right)= \Op(m_s).
\]
Note that $m_s$ is no longer compactly supported, although $m$ was; but $\Op(m_s)$ satisfies the finite propagation property. 
Define on $X_\Gamma$
\[
{A}_s:=A- m_s\left(\sqrt{-\Delta_{X_\Gamma}-\frac{1}{4}}\right)=A-\Op_\Gamma(m_s)
\]
with the corresponding symbol $a_s(z, \lambda, b)= a(z, \lambda, b)-m_s(\lambda)$. For any eigenfunction $\psi_j$ on $X_\Gamma$, since $m_s\left(\sqrt{\Delta_{X_\Gamma}- \frac14}\right)\psi_j= m_s(\lambda_j)\psi_j$, we have
 
\[
\langle \psi_j, {A}_s\psi_j \rangle = \langle \psi_j, A\psi_j \rangle - m_s(\lambda_j).
\]

We keep the former notation $\overline{a}_T$ to denote the symbol corresponding to $\overline{A}_T= \frac{1}{T}\int_0^T P_tAP_tdt$. Consider the operator
\[
B_{T,s}=\frac{1}{T}\int_0^T P_t m_s\left(\sqrt{-\Delta_{X_\Gamma}- \frac{1}{4}}\right)P_tdt
\]
whose symbol is given by 
\[
\overline{m_s}_T(\lambda)= m_s(\lambda)\frac{1}{T}\int_0^T h_t(\lambda)^2dt:=m_s(\lambda)H_T(\lambda)
\]
and 
$\Op_\Gamma(\overline{m_s}_T)=B_{T,s}$. \iffalse{\color{blue}Also define the operator $B_T$ on $L^2(X_\Gamma)$
\[
B_T= \frac{1}{T}\int_0^T P_t m\left(\sqrt{\Delta_{X_\Gamma}- \frac{1}{4}}\right)P_tdt.
\]
Because of the summability issue described above, in general $\Op_\Gamma(\overline{m}_T)\neq B_T$ where $\overline{m}_T(\lambda)=m(\lambda)H_T(\lambda)$.
}
\fi
Replacing $A$ with $A_s$ in \eqref{e:lowerb} and using Proposition \ref{prop: MS Prop 2} as before, we have 
\begin{multline}\label{ineq: A_s HS-estimate}
    \sum_{j:\nu_j\in J}  |\langle \psi_j,  A_s\psi_j \rangle|^2 \leq \frac{1}{C_I^2} \sum_{j=0}^\infty |\langle \psi_j, (\Op_\Gamma(\overline{a}_T)-\Op_\Gamma(\overline{m_s}_T))\rho(\lambda_j)\psi_j \rangle|^2
    \\
    \lesssim_I  \|\Op_{\Gamma, r}((\overline{a}_T-\overline{m_s}_T)\rho)\|^2_{\mathrm{HS}}+ \| \Op_\Gamma(\overline{a}_T-\overline{m_s}_T)\|^2_{\mathrm{HS}}|E_{r, \rho}(\lambda_j)|+\|R_r(\overline{a}_T-\overline{m_s}_T, \rho)\|^2_{\mathrm{HS}}.
\end{multline}

Recall that in the mean-zero case, we applied Nevo's theorem to get an estimate on the first term for which we used the assumption \eqref{condition A1}. Note that for the symbol $a_s$, \eqref{condition A1} does not hold. So, we need to consider $a-m(\lambda)$, and for the first term, decompose 
\[
\overline{a}_T-\overline{m_s}_T= (\overline{a}_T-\overline{m}_T)+ (\overline{m}_T-\overline{m_s}_T).
\]
Therefore, 
\[
\|\Op_{\Gamma, r}((\overline{a}_T-\overline{m_s}_T)\rho)\|^2_{\mathrm{HS}}\leq 2\|\Op_{\Gamma, r}((\overline{a}_T-\overline{m}_T)\rho)\|^2_{\mathrm{HS}} + 2\|\Op_{\Gamma, r}((\overline{m}_T-\overline{m_s}_T)\rho)\|^2_{\mathrm{HS}}.
\]
Moreover, since $\Op(\overline{a_T}-\overline{m_s}_T)$ has finite propagation, 
\[
\| \Op_\Gamma(\overline{a}_T-\overline{m_s}_T)\|^2_{\mathrm{HS}}\leq 2
\| \Op_\Gamma(\overline{a}_T)\|^2_{\mathrm{HS}}|+
2\|\Op_\Gamma( \overline{m_s}_T)\|^2_{\mathrm{HS}}.
\]
\iffalse
{\color{blue}
We make a subtle observation that the decomposition of the symbol above does not work on the second term immediately, since both factors can't automatically be pushed forward to the quotient space. So, for the second term, we use the following decomposition of the operator directly on the quotient space:  
\begin{align*}
     \| \Op_\Gamma(\overline{a}_T-\overline{m_s}_T)\|^2_{\mathrm{HS}}=\| \Op_\Gamma(\overline{a}_T)-B_{T, s}\|^2_{\mathrm{HS}}&=\| \Op_\Gamma(\overline{a}_T)-B_T+B_T-B_{T, s}\|^2_{\mathrm{HS}}\\
    &\leq2\| \Op_\Gamma(\overline{a}_T)-B_T\|^2_{\mathrm{HS}}+2\| B_T-B_{T, s}\|^2_{\mathrm{HS}}
\end{align*}
}
\fi
So, 
\begin{multline}\label{ineq: modified HS-estimate with errors}
    \sum_{j:\nu_j\in J}  |\langle \psi_j,  A_s\psi_j \rangle|^2 \leq \frac{1}{C_I^2} \sum_{j=0}^\infty |\langle \psi_j, (\Op_\Gamma(\overline{a}_T)-\Op_\Gamma(\overline{m_s}_T))\rho(\lambda_j)\psi_j \rangle|\\
    \lesssim_I  \|\Op_{\Gamma, r}((\overline{a}_T-\overline{m}_T)\rho)\|^2_{\mathrm{HS}} + \|\Op_{\Gamma, r}((\overline{m}_T-\overline{m_s}_T)\rho)\|^2_{\mathrm{HS}}\\
    + (\| \Op_\Gamma(\overline{a}_T)\|^2_{\mathrm{HS}}+
\|\Op_\Gamma( \overline{m_s}_T)\|^2_{\mathrm{HS}}) |E_{r,\rho}(\lambda_j)|+\|R_r(\overline{a}_T-\overline{m_s}_T, \rho)\|^2_{\mathrm{HS}}.
\end{multline}
Since $m(\lambda)$ is independent of $\theta$, replacing $a$ by $(a(z, \lambda, b)-m(\lambda))$ does not change the principal term involving $\pa^2_\theta a$ coming from the first term in \eqref{ineq: modified HS-estimate with errors}.

%{\color{red} Why do we need %$B_T$ ? Can't we just write
%$$\| \Op_\Gamma(\overline{a}_T-%\overline{m_s}_T)\|^2_{\mathrm{HS}}\leq 2
%\| \Op_\Gamma(\overline{a}_T)\|^2_{\mathrm{HS}}|+
%2\| \overline{m_s}_T\|^2_{\mathrm{HS}}

%and apply to (2.11) to $\| %\overline{m_s}_T\|^2_{\mathrm{HS}}$ with $R$ replaced by $s$? Here we can use the finite propagation with $R$ replaced by $s$, this term will disappear at the limit because of Benjamini-Schramm convergence.}
%\SSa{ Added a few lines above %and below these comments, removing $B_T$. I also replaced the $E_r$ with $E_{r, \rho}$ in some remaining places and $R_r(a)$ with $R_r(a, \rho)$. I will go through it again later and check if I missed any spots.}

Moreover, using \eqref{ineq: LMS- PDO HS norm} and the fact that $\Op_{\overline{m_s}_T}$ is radial, we have that 
\begin{align*}
    \|\Op_\Gamma(\overline{m_s}_T)\|^2_{\mathrm{HS}} 
    &\leq \int_\DDD\int_\DD |K_{\overline{m_s}_T}(w, z)|^2d\mu(w)d\mu(z)\\ 
    &\qquad\qquad + \frac{e^{2(s+2T)}}{l_{\min}}\operatorname{Vol}\{z\in X:\operatorname{InjRad}_{X_\Gamma}(z)<s+2T\}\,\,\sup_{(z,w)\in \DDD\times \mathbb{\DD}}|K_{\overline{m_s}_T}(z,w)|^2
 \\ &= \Vol(X_\Gamma)
 \int_\DD |K_{\overline{m_s}_T}(0, w)|^2d\mu(w) \\ 
 & \qquad\qquad+ \frac{e^{2(s+2T)}}{l_{\min}}\operatorname{Vol}\{z\in X:\operatorname{InjRad}_{X_\Gamma}(z)<s+2T\}\,\,\sup_{(z,w)\in \DDD\times \mathbb{\DD}}|K_{\overline{m_s}_T}(z,w)|^2.
\end{align*}
Dividing both sides with $\Vol(X_\Gamma)$ we get
\begin{multline}
   \frac{1}{\Vol(X_\Gamma)} \|\Op_\Gamma(\overline{m_s}_T)\|^2_{\mathrm{HS}} 
    \leq \int_\DD |K_{\overline{m_s}_T}(0, w)|^2d\mu(w)\\ + 
\label{e:Opmst}\frac{e^{2(s+2T)}}{l_{\min}}\frac{\operatorname{Vol}\{z\in X_\Gamma:\operatorname{InjRad}_{X_\Gamma}(z)<s+2T\}}{\Vol(X_\Gamma)}\,\,\sup_{(z,w)\in \DDD\times \mathbb{\DD}}|K_{\overline{m_s}_T}(z,w)|^2.
\end{multline}

Using these observations, we now rewrite \eqref{ineq: modified HS-estimate with errors} in terms of $A_p=\Op_\Gamma(a_p)$ defined on $X_p$ with fundamental domain $\DDD_p$.
\begin{multline}\label{ineq: QV general A_p trucnated}
     \frac{1}{\Vol (X_p)} \sum_{j:\nu_{j,p}\in J}  |\langle \psi_{j,p},  A_p\psi_{j,p} \rangle- {m_p}_s(\lambda_{j,p})|^2 \lesssim_I \frac{\sup_{\lambda\in I'}\tnorm{\pa^2_\theta {a_p}_\lambda}^2_2}{\varrho(\eta)} \frac{1}{T}\\
     + C_\rho\frac{ e^{2r+2S_T}}{l_{\min}}\frac{\Vol\{z\in X_p:\operatorname{InjRad}_{X_p}(z)<r\}}{\Vol (X_p)}\sup_{(z,w)\in \DDD_p\times \mathbb{\DD}}|K_{(\overline{a_p}_T-\overline{{m_p}_s}_T)}(z,w)|^2 +  \frac{\|\Op_{\Gamma_p, r}((\overline{m_p}_T-\overline{{m_p}_s}_T)\rho)\|^2_{\mathrm{HS}}}{\Vol(X_p)}\\
     +  \left(\int_\DD |K_{\overline{{m_p}_s}_T}(0, w)|^2d\mu(w) + 
\frac{e^{2(s+2T)}}{l_{\min}}\frac{\operatorname{Vol}\{z\in X_p:\operatorname{InjRad}_{X_p}(z)<s+2T\}}{\Vol(X_p)}\,\,\sup_{(z,w)\in \DDD_p\times \mathbb{\DD}}|K_{\overline{{m_p}_s}_T}(z,w)|^2\right) \\
\times \mathcal{O}\left(\frac{1}{(1+r)^2}\right) + \frac{1}{\Vol(X_p)}\|\overline{A_p}_T\|^2_{\mathrm{HS}}\mathcal{O}\left(\frac{1}{(1+r)^2}\right)+ \frac{1}{\Vol(X_p)}\|R_r(\overline{a_p}_T-\overline{{m_p}_s}_T, \rho)\|^2_{\mathrm{HS}}.
\end{multline}

\subsection{Proof of quantum ergodicity on large surfaces}
From Lemma \ref{lem: sup norm bound of avg op}, we have that 
\begin{equation}
    \sup_{(z, w)\in \DDD_p\times \DD} |K_{P_tA_pP_t}(z, w)|^2\leq C(t, \sigma, k),
\end{equation}
which implies $ \|K^2_{\overline{a_p}_T}\|_\infty$ is uniformly bounded in $p$.  Also, from the definition, we know that $\Op({m_p}_s)=\Op_s(m_p)$ whose kernel is given by $K_{m_p, s}$ and for every $z, w\in \DD$
\[
|K_{m_p,s}(z, w)|\leq |K_{m_p}(z, w)|  \qquad \text{for each $p$}.
\]
Since $m_p\in C_c^\infty(\RR^+)$, using Appendix \ref{app: radial kernel decay} with $N=0$, for every $z, w\in \DD$ we have that 
\[
|K_{m_p}(z, w)|^2\leq C_0(m_p)e^{-d(z, w)}\leq C_0(m_p),
\]
with
\[
C_0(m_p)\leq C_{0,I'}\|m_p\|_{L^\infty(I')},
\]
%{\color{red}{Give precise reference for the previous line.}} \SSa{The expression for $C_0(m_p)$ comes from a step inside the proof of Appendix B. $C_{0, I'}$ comes from considering the integral with $c(\lambda)$, $c(-\lambda)$ and the remainder term}
where $C_{0, I'}$ is obtained by taking the minimum of $C^+_{N, m_p}, C^-_{N, m_p}$, and $C'_{m_p}$ that shows up in \eqref{ineq: uniform C_0,I}, \eqref{ineq: uniform C_0-I}, and \eqref{ineq: uniform C_0I} respectively with $N=0$. Now recall that
\[
a_p(z, \lambda, b)= e^{-(\frac{1}{2}+i\lambda)\langle z, b \rangle}\int_\DD K_{A_p}(z, w)e^{(\frac{1}{2}+i\lambda)\langle w, b \rangle}d\mu(w)= A_p(f_{z, \lambda,b}(w)),
\]
with $f_{z, \lambda, b}(w)=e^{(\frac{1}{2}+i\lambda)(\langle w, b \rangle-\langle z, b \rangle)}$. Since $K_{A_p}$ has finite propagation $S$, using assumption \eqref{ineq: local assumption},
\[
|a_p(z, \lambda, b)|\leq C\|f_{z, \lambda, b}\|_{C^k(B(z, S))},
\]
and for any $w\in B(z, S)$, $|\langle w, b \rangle-\langle z, b \rangle|\leq d(z, w)\leq S$. So $|a_p(z, \lambda, b)|$ is uniformly bounded in $p$. Then
\[
|m_p(\lambda)|=|\beta(\lambda)||M_p(\lambda)|\leq \frac{1}{\Vol(SX_P)}\int_{SX_p}|a_p(z,b, \lambda)|e^{\langle z, b\rangle}d\mu(z)db
\]
is uniformly bounded in $p$ which allows $|K_{m_p}|$ to have a uniform bound. These  bounds imply that 
\begin{equation}
    \sup_{(z,w)\in \DDD_p\times \mathbb{\DD}}|K_{(\overline{a_p}_T-\overline{{m_p}_s}_T)}(z,w)|^2 < C_{0, T} \qquad \text{for every $p$}, 
\end{equation}
 where $C_{0,T}$ is independent of $p$ and $s$. %{\color{red}{but depends on $T$, right ? Because of $K_{(\overline{a_p}_T}$. Don't we want to have the $T$-dependence in the notation, for safety ?}}. \SSa{Yes, I ave added a $T$ to the notations.} 
 In addition,
\[
K_{m_p, s}-K_{m_p}=K_{m_p}(z, w)(\chi(d(z, w)/s)-1),
\]
which is supported on the tail $d(z,w)>s$. Again, using Appendix \ref{app: radial kernel decay} with $N=1$, we have 
\[
|K_{m_p}(z, w)|\leq C_1(m_p)e^{-d(z, w)/2}(1+d(z, w))^{-1}.
\]
Since $C_1(m_p)$ has a uniform upper bound using a similar argument as above,
\[
\sup_{(z,w)\in \DDD_p\times \mathbb{\DD}}|K_{({{m}_p}-{{m_p}_s})}(z,w)|^2 < C_1 \qquad \text{for every $p$},
\]
and
\[
\int_\DD |K_{m_p, s}(0, w)-K_{m_p}(0, w)|^2d\mu(w)\leq C_1 \int_s^\infty e^{-r}(1+r)^{-2}\sinh r dr
\]
and as $s\longrightarrow +\infty$, we have using Plancherel
\begin{equation}
    \|{m_p}_s-m_p\|_{L^2(\RR^+, \lambda\tanh{(2\pi \lambda)}d\lambda)}\longrightarrow 0.
\end{equation}

Finally, we also note,
\begin{equation}
    \|\overline{A_p}_T\|^2_{\mathrm{HS}}\leq \Vol(X_p)e^{S_T} \sup_{\DD\times\DD}|K_{\overline{a_{p}}_T}|^2= C_{2,T}\Vol(X_p) <\infty,
\end{equation}
with $C_{2, T}$ independent of $p,s$ and by assumption \eqref{ineq: local assumption}, $C_{I'}:=\sup_{\lambda\in I'}\tnorm{\pa^2_\theta {a_p}_\lambda}^2_2$ is bounded above uniformly in $p$.

\subsubsection{Taking limit $p\longrightarrow +\infty$} Fixing $r, s$, and $T$, the second and fifth terms on the right-hand side of the above estimate go to zero as $p\longrightarrow \infty$ since
\[
\lim_{p\longrightarrow \infty}\frac{\Vol\{z\in X_p:\InjRad_{X_p}(z)<r+S_T\}}{\Vol(X_p)}= 0,
\]
from the definition of $\mathrm{BS}$-convergence. For the third term in \eqref{ineq: QV general A_p trucnated}, set 
\[
q_{p, s}(\lambda)= (\overline{m_p}_T(\lambda)-\overline{{m_p}_s}_T(\lambda))\rho(\lambda),
\]
which has compact support, the same as $\supp \rho$. Using Lemma \ref{lem: truncated kernel HS-norm estimate},
\begin{multline*}
    \frac{1}{\Vol(X_p)}\|\Op_{\Gamma_p, r}((\overline{m_p}_T-\overline{{m_p}_s}_T)\rho)\|^2_{\mathrm{HS}}\lesssim \int_\DD |K_{q_{p,s}}(0, w)|^2d\mu(w) \\+ \frac{e^{2r}}{l_{\min}} \frac{\mathrm{Vol}\{x\in X_p:\mathrm{InjRad}_{X_p}(x)<r\}}{\Vol(X_p)} \sup_{(z,w)\in \DDD_p\times\DD}|K_{q_{p,s}}(z,w)|^2,
\end{multline*}
where the second term above goes to zero as $p\longrightarrow \infty$. For the final term, 
\begin{multline*}
     \frac{1}{\Vol(X_p)}\|R_r(\overline{a_p}_T-\overline{{m_p}_s}_T, \rho)\|^2_{\mathrm{HS}}=  \left(\frac{S_T}{r}\right)^2e^{2S_T}\sup_{\DDD_p\times \DD}|K_{\overline{a_p}_T-\overline{{m_p}_s}_T}|^2\|K_\rho\|^2_2  \\
    + \left(\frac{S_T}{r}\right)^2e^{2S_T}\sup_{\DDD_p\times \DD}|K_{\overline{a_p}_T-\overline{{m_p}_s}_T}|^2\|K_\rho\|^2_2 \frac{e^{(r+S_T)}}{l_{\min}} \frac{ \Vol\{z\in X_p:\InjRad_{X_p}(z)<r+S_T\}}{\Vol(X_p)}.
\end{multline*}
where the second term on the right goes to zero as $p\longrightarrow \infty$. These imply
\begin{multline}\label{ineq: QE with limpt p}
 \limsup_{p\longrightarrow \infty}  \frac{1}{\Vol (X_p)} \sum_{j:\nu_j\in J}  |\langle \psi_{j, p},  A_p\psi_{j,p} \rangle- {m_p}_s(\lambda_j)|^2 
    \lesssim_I \frac{C_{I'}}{\varrho(\eta)} \frac{1}{T}+\limsup_{p\longrightarrow \infty} \int_\DD |K_{q_{p,s}}(0, w)|^2d\mu(w)\\
  +\mathcal{O}\left(\frac{1}{(1+r)^2}\right) \limsup_{p\longrightarrow \infty}\int_\DD |K_{\overline{{m_p}_s}_T}(0, w)|^2d\mu(w) 
  +\left(\frac{S_T}{r}\right)^2e^{2S_T}C_{0,T}\|K_\rho\|^2_{2}  +C_{2,T}\mathcal{O}\left(\frac{1}{(1+r)^2}\right).
\end{multline}

\subsubsection{Taking limit $s\longrightarrow +\infty$} Now, we fix $r$ and $T$  in \eqref{ineq: QE with limpt p}. The last two terms are independent of $s$.  
 For the second term in \eqref{ineq: QE with limpt p}, applying Plancherel to the symbol $q_{p, s}$ depending only on $\lambda$ (up to a normalisation factor), we have that 
\[
\int_\DD |K_{q_{p, s}}(0, w)|^2d\mu(w)= \int_0^\infty |H_T(\lambda)|^2 |m_p(\lambda)-{m_p}_s(\lambda)|^2 \rho(\lambda)^2 \lambda\tanh{(2\pi \lambda)}d\lambda,
\]
where the integrand is supported over $I'$. For fixed $T$, $H_T(\lambda)$ is bounded on $I'$. Since, for every $p$, ${m_p}_s$ converges to $m_p$ in $L^2(\RR^+, \lambda\tanh{(2\pi \lambda)}d\lambda)$ as $s\longrightarrow+\infty$ from above, we have %because the same holds for the corresponding kernels (thanks to the result of Appendix B and to Plancherel). Then, we have 
\[
\int_\DD |K_{q_{p,s}}(0, w)|^2d\mu(w)= \lim_{s\longrightarrow \infty}\|q_{p,s}\|^2_{L^2(\RR^+, \lambda\tanh{(2\pi \lambda)}d\lambda)}\longrightarrow 0.
\]

%Applying Plancherel to the symbol $q_s$ depending only on $\lambda$ as before, the corresponding radial kernels converge to 0 in $L^2(\DDD\times \DD)$ %{\color{red}{Problem here. Should be $L^2(\cD\times \DD)$ but maybe with some normalization ? I already corrected the same mistake somewhere else, check it does not appear anywhere else. A radial operator CANNOT be Hilbert-Schmitt on the whole $L^2(\DD\times \DD)$}}. \SSa{Yes, I should have been more careful.} 
%Since $\Op_{\Gamma, r}(q_s)$ is defined by the truncated kernel, its Hilbert-Schmidt norm is controlled by the $L^2$-norm of the radial kernel, and therefore
%\[
%\frac{1}{\Vol (X_\Gamma)} \|\Op_{\Gamma, r}((\overline{m}_T-\overline{m_s}_T)\rho)\|^2_{\mathrm{HS}}\longrightarrow 0.
%\]

For the third term, up to a normalisation factor, we have that 
\[
\int_\DD |K_{\overline{{m_p}_s}_T}(0, w)|^2d\mu(w)= \int_0^\infty |H_T(\lambda)|^2 |{m_p}_s(\lambda)|^2 \lambda\tanh{(2\pi \lambda)}d\lambda.
\]
Recall that $h_t$ is the Fourier transform of the radial function $w\mapsto k_t(d(0, w))=\frac{1}{\sqrt{\cosh{t}}}\chi_{t, \sigma}(d(0, w))$,
\[
h_t(\lambda)=\int_\DD k_t(d(0, w))\varphi_\lambda(d(0, w)) d\mu(w)=\frac{2\pi}{\sqrt{\cosh t}}\int_0^t \chi_{t, \sigma}(r)\varphi_\lambda(r)\sinh r dr,
\]
It is known (see Equation (3.1.5) from Chapter 5 of \cite{bray1994fourier}) that for $\lambda\in \RR^+$ and some constant $C$ such that
\[
|\varphi_\lambda(r)|\leq C\varphi_0(r)\leq Ce^{-r/2}(1+r).
\]
Then 
\[
|h_t(\lambda)|\leq \frac{2\pi}{\sqrt{\cosh t}}\int_0^t |\chi_{t, \sigma}(r)\varphi_\lambda(r)|\sinh r dr\leq  \frac{2\pi}{\sqrt{\cosh t}}\int_0^t |\varphi_\lambda(r)|\sinh r dr
\]
Using the above bound of $\varphi_\lambda$, 
\[
|h_t(\lambda)|\lesssim \frac{1}{\sqrt{\cosh t}}\int_0^t e^{-r/2}(1+r) \sinh r dr  \lesssim  \frac{1}{\sqrt{\cosh t}}\int_0^t e^{r/2}(1+r) dr\lesssim \frac{1}{\sqrt{\cosh t}}te^{t/2}.
\]
But $1/\sqrt{\cosh{t}}\leq e^{-t/2}$. So, 
\[
|H_T(\lambda)|^2 \leq \left(\frac{1}{T}\int_0^T |h_t(\lambda)|^2dt \right)^2\lesssim \left(\frac{1}{T}\int_0^T t^2dt\right)^2 = C_T ~(\text{say}).
\]
Then, for any fixed $T$,
\[
 \int_0^\infty |H_T(\lambda)|^2 |m_s(\lambda)|^2 \lambda\tanh{(2\pi \lambda)}d\lambda\leq C_T \int_0^\infty |{m_p}_s(\lambda)|^2 \lambda\tanh{(2\pi \lambda)}d\lambda \longrightarrow C_T \|{m_p}(\lambda)\|^2_{L^2(\RR^+, \lambda \tanh{2\pi \lambda}d\lambda)}.
\]

\iffalse
{\color{purple}
Try 1: Then, using Cauchy-Schwarz
\begin{align*}
    |h_t(\lambda)|^2&\leq \frac{(2\pi)^2}{\cosh{t}}\left(\int_0^t |\varphi_\lambda(r)|^2\sinh r dr\right)\left(\int_0^t |\chi_{\sigma, t}(r)|^2\sinh r dr\right)\\
    &\leq \frac{(2\pi)^2(\cosh t -1)}{\cosh t}\int_0^t |\varphi_\lambda(r)|^2\sinh r dr.
\end{align*}
Using Lemma 3.1.5 from Chapter 5 of \cite{bray1994fourier}), up to some constant, we have that
\[
\int_0^t |\varphi_\lambda(r)|^2\sinh r dr \cong |c(\lambda)|^2t + o(t),
\]
where $c(\lambda)$ is the Harish-Chandra $c$-function as defined in Appendix \ref{app: radial kernel decay}. Then for large $t$,
\[
H_T(\lambda)= \frac{1}{T} \int_0^T h_t(\lambda)^2dt \lesssim \frac{1}{T} \int_0^T \frac{(\cosh t -1)}{\cosh t} |c(\lambda)|^2 t \, dt\leq |c(\lambda)|^2\frac{T}{2}.
\]
Then 
\[
\int_\DD |K_{\overline{m_s}_T}(0, w)|^2d\mu(w)\leq C_T \int_0^\infty  |c(\lambda)|^4|m_s(\lambda)|^2 \lambda\tanh{(2\pi \lambda)}d\lambda = C_T\int_0^\infty |c(\lambda)|^2|m_s(\lambda)|^2d\lambda,
\]
since $|c(\lambda)|^{-2}= \lambda \tanh{2\pi \lambda}$. 
}
\fi

Finally, taking $s\longrightarrow +\infty$ and interchanging the limits on the left-hand side, we have $m_s(\lambda_j)\to m(\lambda_j)=M(\lambda_j)$ and 
\begin{multline}\label{ineq: QE with limpt s}
\limsup_{p\longrightarrow \infty}  \frac{1}{\Vol(X_p)}\sum_{j:\nu_j\in J} \left |\langle \psi_{j,p},  {A_p}\psi_{j,p}\rangle -\frac{1}{\Vol(SX_p)}\int_{SX_p}a_p(z,\lambda_{j,p}, b)e^{\langle z,b \rangle}d\mu(z)db\right|^2 
    \lesssim_I \frac{C_{I'}}{\varrho(\eta)} \frac{1}{T}
\\  +C_T \limsup_{p\longrightarrow \infty}\|m_p(\lambda)\|^2_{L^2(\RR^+, \lambda \tanh{2\pi \lambda}d\lambda)} \mathcal{O}\left(\frac{1}{(1+r)^2}\right) +\left(\frac{S_T}{r}\right)^2e^{2S_T}C_{0, T} \|K_\rho\|^2_{2}  +C_{2,T}\mathcal{O}\left(\frac{1}{(1+r)^2}\right).
\end{multline}

\iffalse
On $X_\Gamma$, for any eigenfunction $\psi_j$, we have 
\[
(B_T-B_{T, s})\psi_j= q_s(\lambda_j)= (m(\lambda_j)-m_s(\lambda_j))H_T(\lambda_j), 
\]
and 
\[
\|B_T-B_{T, s}\|^2_{\mathrm{HS}}= \sum_{j} |m(\lambda_j)-m_s(\lambda_j)|^2|H_T(\lambda_j)|^2\leq \max_j |H_T(\lambda_j)|^2\sum_{j} |m(\lambda_j)-m_s(\lambda_j)|^2.
\]
Because $m, m_s$ are compactly supported {\color{red}{No, not $m_s$, I'm not sure we can proceed like that.}} \SSa{Yes, I see the mistake.} and $m_s\to m$ uniformly on compact intervals, we have $m_s(\lambda_j)\to m(\lambda_j)$ for each fixed $j$, and only finitely many $j$ contribute to the above sum. Therefore
\[
\|B_T-B_{T, s}\|^2_{\mathrm{HS}}\longrightarrow 0.
\]
\fi
%Finally, letting $s\longrightarrow +\infty$, the left-hand side of \eqref{ineq: modified HS-estimate with errors} converges to 
%\[
%\frac{1}{\Vol(X_\Gamma)}\sum_{j:\nu_j\in J} \left |\langle \psi_j,  {A}\psi_j\rangle -\frac{1}{\Vol(SX_\Gamma)}\int_{SX_\Gamma}a(z,\lambda_j, b)e^{\langle z,b \rangle}d\mu(z)db\right|^2.
%\]

\subsubsection{Taking limit $r\longrightarrow \infty$} Fix $T$ in \eqref{ineq: QE with limpt s}. The left-hand side is independent of $r$ and on the right, the penultimate term is $\mathcal{O}(S_T^2/r^2)$, which goes to zero while $\mathcal{O}(1/(1+r)^2)$ makes the second and final term go to zero. So, 

\begin{equation}\label{ineq: QE with limpt T}
   \limsup_{p\longrightarrow \infty}  \frac{1}{\Vol(X_p)}\sum_{j:\nu_j\in J} \left |\langle \psi_{j,p},  {A_p}\psi_{j,p}\rangle -\frac{1}{\Vol(SX_p)}\int_{SX_p}a_p(z,\lambda_{j,p}, b)e^{\langle z,b \rangle}d\mu(z)db\right|^2  
    \lesssim_I \frac{C_{I'}}{\varrho(\eta)} \frac{1}{T}.
\end{equation}

Now, we use a result analogous to Weyl's law for eigenvalues in a fixed bounded interval for the sequences of surfaces $X_p$.
\begin{Lemma}[Lemma 9.1, \cite{MassonSahlsten17}]
For any compact interval $J\subset (1/4, \infty)$,
    \[
    \lim_{p\to \infty} \frac{N(X_p, J)}{\Vol(X_p)}=\frac{1}{4\pi}\int_{\RR} \chi_I(1/4+s^2)s\tanh{(\pi s)}ds,
    \]
    where $N(X_p, J)$ is the number of eigenvalues in $J$ with multiplicity.
\end{Lemma}
In other words, we have that $N(X_p, J)\sim_I \Vol(X_p)$, which gives
\[
\limsup_{p\longrightarrow \infty}  \frac{1}{N(X_p, J)}\sum_{j:\nu_j\in J} \left |\langle \psi_{j,p},  {A_p}\psi_{j,p}\rangle -\frac{1}{\Vol(SX_p)}\int_{SX_p}a_p(z,\lambda_{j,p}, b)e^{\langle z,b \rangle}d\mu(z)db\right|^2  \lesssim_I \frac{C_{I'}}{\varrho(\eta)T}.  
\]
Since the above inequality is true for any $T$, we must have 
\begin{equation*}
     \limsup_{p\to \infty}  \frac{1}{N(X_p, J)} \sum_{j:\nu_j\in J} \left |\langle \psi_{j,p},  {A_p}\psi_{j,p}\rangle -\frac{1}{\Vol(SX_p)}\int_{SX_p}a_p(z,\lambda_{j,p}, b)e^{\langle z,b \rangle}d\mu(z)db\right|^2 =0.
\end{equation*}
Moreover, the quantum variance is non-negative and $\limsup_p$ goes to 0. So, 
\begin{equation*}
     \lim_{p\to \infty}  \frac{1}{N(X_p, J)} \sum_{j:\nu_j\in J} \left |\langle \psi_{j,p},  {A_p}\psi_{j,p}\rangle -\frac{1}{\Vol(SX_p)}\int_{SX_p}a_p(z,\lambda_{j,p}, b)e^{\langle z,b \rangle}d\mu(z)db\right|^2 =0,
\end{equation*}
which completes the proof.

\appendix

\section{Proof of Proposition \ref{prop: MS Prop 2}}\label{app: proof of prop 3.3}
Let $g_{t,\sigma}$ and $g_t^\sharp$ denote the Abel transforms of $k_{t,\sigma}$ and $k_t^\sharp$, respectively. By the relation between the Selberg transform and the Abel transform of a radial kernel, we have
\[
h_{t,\sigma}(\lambda)=\int_{\mathbb R}e^{i\lambda u}g_{t,\sigma}(u)du,
\qquad
h_t^\sharp(\lambda)=\int_{\mathbb R}e^{i\lambda u}g_t^\sharp(u)du.
\]
%where 
%\[
%g(u)=\sqrt{2}\int_{|u|}^\infty \frac{k(\rho)\sinh{\rho}}{\sqrt{\cosh{\rho}-\cosh{u}}}d\rho, \qquad \text{with } k:[0, \infty]\to \CC. 
%\]
Now,
\begin{align*}
    g_t^\sharp(u)&=\sqrt{\frac{2}{\cosh t}}\int_{|u|}^{t}\frac{\sinh r}{\sqrt{\cosh r-\cosh u}}dr\\
    &= \sqrt{\frac{2}{\cosh t}}\left( \int_{|u|}^{t-\sigma}\frac{\sinh r}{\sqrt{\cosh r-\cosh u}}dr + \int_{t-\sigma}^{t}\frac{\sinh r}{\sqrt{\cosh r-\cosh u}}dr \right)
\end{align*}
and
\begin{align*}
    g_{t,\sigma}(u)&=\sqrt{\frac{2}{\cosh t}}\int_{|u|}^{t}\chi_{t, \sigma}(r)\frac{\sinh r}{\sqrt{\cosh r-\cosh u}}dr\\
    &=\sqrt{\frac{2}{\cosh t}}\left( \int_{|u|}^{t-\sigma}\frac{\sinh r}{\sqrt{\cosh r-\cosh u}}dr + \int_{t-\sigma}^{t}\chi_{t, \sigma}(r)\frac{\sinh r}{\sqrt{\cosh r-\cosh u}}dr \right),
    %&\leq \sqrt{\frac{2}{\cosh t}}\left( \int_{|u|}^{t}\frac{\sinh r}{\sqrt{\cosh r-\cosh u}}dr + \int_{t}^{t+\sigma}\frac{\sinh r}{\sqrt{\cosh r-\cosh u}}dr \right)
\end{align*}
while noting that $g_{t, \sigma}(u)=g_t^\sharp(u)=0$ whenever $|u|>t$. If $|u|\leq t-\sigma$, we have
\[
g_{t, \sigma}(u)=g_t^\sharp(u)+\sqrt{\frac{2}{\cosh t}}\int_{t-\sigma}^{t}(\chi_{t, \sigma}(r)-1)\frac{\sinh r}{\sqrt{\cosh r-\cosh u}}dr,
\]
and if $t-\sigma<|u|<t$, we have
\[
g_{t, \sigma}(u)=g_t^\sharp(u)+\sqrt{\frac{2}{\cosh t}}\int_{|u|}^{t}(\chi_{t, \sigma}(r)-1)\frac{\sinh r}{\sqrt{\cosh r-\cosh u}}dr.
\]
Combining the two cases gives us
\[
g_{t, \sigma}(u)=g_t^\sharp(u)+\sqrt{\frac{2}{\cosh t}}\int_{\max\{|u|, t-\sigma\}}^{t}\frac{(\chi_{t, \sigma}(r)-1)\sinh r}{\sqrt{\cosh r-\cosh u}}dr,
\]
and using the fact that $g_{t,\sigma}$ is even,
\begin{align*}
    h_{t, \sigma}(\lambda)&= 2\int_0^{t}\cos{(\lambda u)}g_{t,\sigma}(u)du\\
    &= h_t^\sharp(\lambda)+2\sqrt{\frac{2}{\cosh t}}\int_0^{t}\cos{(\lambda u)}\left(\int_{\max\{u, t-\sigma\}}^{t} (\chi_{t, \sigma}(r)-1)\frac{\sinh r}{\sqrt{\cosh r-\cosh u}}dr\right)du.
\end{align*}
If we can show that $\delta h_{t,\sigma}:= h_{t, \sigma}-h_t^\sharp$ is uniformly small, then a similar estimate would also hold for $h_{t,\sigma}$.
\begin{align*}
    \delta h_{t,\sigma}(\lambda)&= 2\sqrt{\frac{2}{\cosh t}}\int_0^{t}\cos{(\lambda u)}\left(\int_{\max\{u, t-\sigma\}}^{t} \frac{(\chi_{t, \sigma}(r)-1)\sinh r}{\sqrt{\cosh r-\cosh u}}dr\right)du\\
&= 2\sqrt{\frac{2}{\cosh t}} \int_0^{t-\sigma}\int_{t-\sigma}^{t}\frac{ (\chi_{t, \sigma}(r)-1)\cos{(\lambda u)}\sinh r}{\sqrt{\cosh r-\cosh u}}dr du \\
&\qquad\qquad\qquad\qquad\qquad\qquad\qquad
+2\sqrt{\frac{2}{\cosh t}} \int_{t-\sigma}^{t}\int_u^{t}\frac{(\chi_{t, \sigma}(r)-1)\cos{(\lambda u)}\sinh r}{\sqrt{\cosh r-\cosh u}}dr du\\
&= 2\sqrt{\frac{2}{\cosh t}} \int_{t-\sigma}^{t}\int_0^{t-\sigma}\frac{ (\chi_{t, \sigma}(r)-1)\cos{(\lambda u)}\sinh r}{\sqrt{\cosh r-\cosh u}}du dr \\
&\qquad\qquad\qquad\qquad\qquad\qquad\qquad
+2\sqrt{\frac{2}{\cosh t}}  \int_{t-\sigma}^{t}\int_{t-\sigma}^{r}\frac{ (\chi_{t, \sigma}(r)-1)\cos{(\lambda u)}\sinh r}{\sqrt{\cosh r-\cosh u}}du dr\\
&= 2\sqrt{\frac{2}{\cosh t}}\int_{t-\sigma}^{t} (\chi_{t, \sigma}(r)-1)\sinh r \left(\int_0^r \frac{\cos{(\lambda u)}du}{\sqrt{\cosh r-\cosh u}}\right)dr.
\end{align*}
Since we are interested in large $t$, we can assume $t-\sigma>1$. Note that the inner integral is improper at $u=r$. Then for $r>1$, the inner integral needs to be treated separately for $u\in [0, r-1]$ and $u\in[r-1, r]$.
\begin{Lemma}\label{lem: uniform bound error term h_{t,sigma}}
    For some compact interval $I\subset (0, \infty)$, there exists constant $C_I$ (depending only on $I$) such that for every $r>1$ and $\lambda\in I$,
    \[
    \left|\int_0^r \frac{\cos{(\lambda u)}}{\sqrt{\cosh r-\cosh u}}du\right|\leq C_I e^{-r/2}.
    \]
\end{Lemma}
\begin{proof}
   Note that
\begin{align*}
    \int_0^r \frac{\cos{(\lambda u)}du}{\sqrt{\cosh r-\cosh u}}&= \int_0^r \frac{\cos{(\lambda u)}du}{\sqrt{2\sinh{\frac{r+u}{2}}\sinh{\frac{r-u}{2}}}}=\int_0^r \frac{\sqrt{2}\cos{(\lambda u)}du}{\sqrt{e^r(1-e^{-(r+u)})(1-e^{-(r-u)})}}\\
    &=\sqrt{2}e^{-r/2}\int_0^r\cos{(\lambda u)}(1-e^{-(r+u)})^{-1/2}(1-e^{-(r-u)})^{-1/2}du\\
    &=\sqrt{2}e^{-r/2}\left(\int_0^{r-1}\cos{(\lambda u)}\rho(u)du+ \int_{r-1}^r\cos{(\lambda u)}\rho(u)du\right),
\end{align*}
where $\rho(u):= (1-e^{-(r+u)})^{-1/2}(1-e^{-(r-u)})^{-1/2}$.

For the second integral, since $u\in [r-1, r]$, we have $r-u\leq 1$ and $r+u > r> 1$. Therefore,
\begin{align*}
    \left|\int_{r-1}^r \cos{(\lambda u)}\rho(u)du\right|&\leq (1-e^{-1})^{-1/2}\int_{r-1}^r \left(1-e^{-(r-u)}\right)^{-1/2}du\\
    &= C\int_0^1 \left(1-e^{-x}\right)^{-1/2} dx= C\lim_{\epsilon\to 0^-}\int_\epsilon^1 (1-e^{-x})^{-1/2}dx \\
    &\leq C\lim_{\epsilon\to 0^-}\int_\epsilon^1 \frac{dx}{\sqrt{x}}=2C.
\end{align*}

For the first integral, 
\begin{align*}
  \left| \int_0^{r-1}\cos{(\lambda u)}\rho(u)du \right| &=\left| \frac{1}{\lambda}\int_0^{r-1}\frac{d}{du}(\sin{(\lambda u)})\rho(u)du\right|\\
   &= \left|\frac{1}{\lambda}\int_0^{r-1}\sin{(\lambda u)}\rho'(u) du +\frac{1}{\lambda} \sin(\lambda(r-1))\rho(r-1)\right|\\
   &\leq \frac{1}{\lambda}\int_0^{r-1}|\rho'(u)| du + \frac{1}{\lambda}|\rho(r-1)|.
\end{align*}
Since $0\leq u\leq r-1\implies r-u\geq 1$. Also, $r+u\geq r>1$. Hence, both the factors in $\rho(u)$ are uniformly bounded, i.e.,
\[
(1-e^{-(r+u)})\geq 1-e^{-1};\qquad (1-e^{-(r-u)})\geq 1-e^{-1}.
\]
So $\rho(r-1)<C$, and $\frac{1}{\lambda}$ is bounded since $I\Subset (1/4, \infty)$. Moreover
\begin{align*}
    \rho'(u)&=(1-e^{-(r+u)})^{-1/2}\frac{d}{du}(1-e^{-(r-u)})^{-1/2}+ (1-e^{-(r-u)})^{-1/2}\frac{d}{du}(1-e^{-(r+u)})^{-1/2}\\
    &=\frac{1}{2}\,e^{-(r-u)}(1-e^{-(r-u)})^{-3/2} (1-e^{-(r+u)})^{-1/2}-\frac{1}{2}\,e^{-(r+u)} (1-e^{-(r-u)})^{-1/2} (1-e^{-(r+u)})^{-3/2}.
\end{align*}
Hence $|\rho'(u)|\leq C(e^{-(r-u)}+ e^{-(r+u)})$, and 
\begin{align*}
    \int_0^{r-1} |\rho'(u)|du&\leq C\int_0^{r-1}\left(e^{-(r-u)}+ e^{-(r+u)}\right)du=C\int^r_1 e^{-x}dx + Ce^{-r}\int_0^{r-1} e^{-u} du\\
    &= C(e^{-1}-e^{-r})+ Ce^{-r}(1-e^{-(r-1)}) = C\left(e^{-1}-e^{-(2r-1)}\right)\leq Ce^{-1}.
\end{align*}
Combining these bounds, we have 
\[
\left| \int_0^{r-1}\cos{(\lambda u)}\rho(u)du \right| \leq C_I,
\]
where $C_I$ is a constant independent of $t, r, \sigma$. Combining the estimates on those two integrals completes the proof.
\end{proof}

\begin{proof}[Proof of Proposition \ref{prop: MS Prop 2}]
   Recall that
\[
\delta h_{t,\sigma}(\lambda)=  2\sqrt{\frac{2}{\cosh t}}\int_{t-\sigma}^{t} (\chi_{t, \sigma}(r)-1)\sinh r \left(\int_0^r \frac{\cos{(\lambda u)}du}{\sqrt{\cosh r-\cosh u}}\right)dr.
\]
For $t> 1+\sigma$, $r\in[t-\sigma, t]$ is $>1$ and
\[
|\delta h_{t,\sigma}(\lambda)|\leq 2C_I\sqrt{\frac{2}{\cosh t}}\int_{t-\sigma}^{t}e^{-r/2}\sinh{r} dr\leq C_I\sqrt{\frac{2}{\cosh t}}\int_{t-\sigma}^{t}e^{r/2} dr,
\]
where the final inequality uses the fact $\sinh{r}\leq e^r/2$. Moreover, since $\cosh{t}\geq e^t/2$,
\[
|\delta h_{t,\sigma}(\lambda)|\leq 2C_Ie^{-t/2}\int_{t-\sigma}^{t}e^{r/2}dr=4C_Ie^{-t/2}\left(e^{t/2}-e^{(t-\sigma)/2}\right)=4C_I(1-e^{-\sigma/2})
\]
for all $t>1+\sigma$ and any $\lambda\in I$.Since $h_{t, \sigma}(\lambda)=h_t^\sharp(\lambda)+ \delta h_{t, \sigma}(\lambda),$ for any $T\geq 1+\sigma$,
\[
\left(\frac{1}{T}\int_1^T |h_{t, \sigma}(\lambda)|^2dt\right)^{1/2} \geq \left(\frac{1}{T}\int_1^T |h_{t}^\sharp(\lambda)|^2dt\right)^{1/2}-\left(\frac{1}{T}\int_1^T |\delta h_{t, \sigma}(\lambda)|^2dt\right)^{1/2}.
\]
The uniform bound in Lemma \ref{lem: uniform bound error term h_{t,sigma}} implies 
\[
\left(\frac{1}{T}\int_1^T |\delta h_{t, \sigma}(\lambda)|^2dt\right)^{1/2} \leq 4C_I(1-e^{-\sigma/2}).
\]
Moreover, we have the identity
\[
\frac{1}{T}\int_1^T |h_t^\sharp(\lambda)|^2 dt=\frac{1}{T}\int_0^T |h_t^\sharp(\lambda)|^2dt - \frac{1}{T}\int_0^1 |h_t^\sharp(\lambda)|^2\,dt.
\]
By the sharp-kernel estimate \eqref{ineq: LeMasson-Sahlsten Prop 4.2}, for sufficiently large $T$,
\[
\frac{1}{T}\int_0^T |h_t^\sharp(\lambda)|^2\,dt \ge C^*_I, \qquad \lambda\in I.
\]
On the other hand, since $h_t^\sharp(\lambda)$ is continuous on the compact set $[0,1]\times I$, there exists a constant $M_I>0$ such that $|h_t^\sharp(\lambda)|^2\le M_I$ for all $t\in[0,1], \lambda\in I$. It follows that
\[
\frac1T\int_1^T |h_t(\lambda)|^2\,dt
\ge
C^*_I-\frac{M_I}{T}.
\]
If $T$ is large enough so that $T\geq 2M_I/C^*_I:=T_I$, then
\[
\frac{1}{T}\int_1^T |h_t(\lambda)|^2dt \geq \frac{C_I^*}{2}.
\]
Hence
\[
\left(\frac{1}{T}\int_0^T |h_{t, \sigma}(\lambda)|^2dt\right)^{1/2} \geq\left(\frac{1}{T}\int_1^T |h_{t, \sigma}(\lambda)|^2dt\right)^{1/2} \geq \sqrt{\frac{C_I^*}{2}} - 4C_I(1-e^{-\sigma/2}).
\]
Choose $\sigma_0$ small enough such that for every $\sigma<\sigma_0$, $\sqrt{C_I^*/2}> 4C_I(1-e^{-\sigma/2})$. Then 
\[
\frac{1}{T}\int_0^T |h_{t, \sigma}(\lambda)|^2dt\geq \left(\sqrt{C_I^*/2}- 4C_I(1-e^{-\sigma/2}) \right)^2:= C(I, \sigma)
\]
for all $T\geq \max\{T_I, 1+\sigma\},~ \lambda\in I,~ \sigma<\sigma_0$. 
\end{proof}

\section{Decay of the radial kernel of $\Op(\rho)$}\label{app: radial kernel decay}
\begin{proposition}\label{prop: k_rho-decay}
For $\rho\in C_c^\infty(\mathbb R^+)$, define
\[
k_\rho(t):=\int_0^\infty\rho(\lambda)\varphi_\lambda(t)\lambda\tanh(2\pi\lambda)d\lambda,\qquad t\geq 0,
\]
where $\varphi_\lambda$ denotes the spherical function on $\mathbb D$. Then for every integer $N\geq 0$, there exists a constant $C_{N, \rho}>0$ such that
\[
|k_\rho(t)|\le C_{N, \rho}\,e^{-t/2}(1+t)^{-N} \quad \text{for every } t\geq 0.
\]
\end{proposition}
\begin{proof}
Recall that given $\rho\in C_c^\infty(\RR^+)$, $\Op(\rho)$ has the kernel
\[
K_\rho(z, w)=k_\rho(d(z, w))=\int_0^\infty \rho(\lambda)\varphi_\lambda(d(z,w))\lambda \tanh{(2\pi\lambda)}d\lambda,
\]
where $\varphi_\lambda(d(z, w))=\frac{1}{2\pi}\int_B e^{(\frac{1}{2}+i\lambda)\langle z,b\rangle}e^{(\frac{1}{2}-i\lambda)\langle w,b\rangle}$. Calling $t=d(z,w)$, the spherical function $\varphi_\lambda(t)$ can be expressed as 
\[
\varphi_\lambda(t)=c(\lambda)e^{(-\frac{1}{2}+i\lambda)t}\sum_{l=0}^\infty\Gamma_l(\lambda)e^{-2lt}+ c(-\lambda)e^{(-\frac{1}{2}-i\lambda)t}\sum_{l=0}^\infty\Gamma_l(-\lambda)e^{-2lt},
\]
with $\Gamma_0(\lambda)=1$, $c(\lambda)=\frac{\Gamma(i\lambda)}{\pi^{1/2}\Gamma(\frac{1}{2}+i\lambda)}$, and there exist constants $d_1, d_2$ such that $|\Gamma_l(\lambda)|\leq d_1(1+l^{d_2})$ for $l>0$. See Chapter 5, Theorem 3.1.4 of \cite{bray1994fourier} for a more detailed description of the coefficients.
Then 
\[
\varphi_\lambda(t)=\left(c(\lambda)e^{(-\frac{1}{2}+i\lambda)t}+ c(-\lambda)e^{(-\frac{1}{2}-i\lambda)t}\right) + T_\lambda(t), 
\]
where $T_\lambda(t)$ is the remainder of the sum on $l\geq 1$. Replacing $\varphi_\lambda$ in the above formulation of $K_\rho$, we have
\begin{multline*}
k_\rho(t)= e^{-\frac{t}{2}}\int_0^\infty \rho(\lambda) c(\lambda)e^{i\lambda t}\lambda \tanh{(2\pi\lambda)}d\lambda+ e^{-\frac{t}{2}}\int_0^\infty \rho(\lambda) c(-\lambda)e^{-i\lambda t}\lambda \tanh{(2\pi\lambda)}d\lambda \\
+\int_0^\infty \rho(\lambda) T_\lambda(t)\lambda \tanh{(2\pi\lambda)}d\lambda.
\end{multline*}
We now use integration by parts $N$-times on the first integral.
\begin{align*}
e^{-\frac{t}{2}}\int_0^\infty (\rho(\lambda) c(\lambda) \lambda \tanh{(2\pi\lambda)})e^{i\lambda t}d\lambda&=\frac{e^{-t/2}}{(it)^N} \int_0^L (\rho(\lambda) c(\lambda) \lambda \tanh{(2\pi\lambda)})(it)^Ne^{i\lambda t}d\lambda\\
&=\frac{e^{-t/2}}{(it)^N} \int_0^L \pa_\lambda^N(\rho(\lambda) c(\lambda) \lambda \tanh{(2\pi\lambda)})e^{i\lambda t}d\lambda,
\end{align*}
where $\supp(\rho)\subset [0, L]$, and the boundary term after integration by parts vanishes since $\rho$ is compactly supported and $c(\lambda)$ is smooth when $\lambda\in \RR^+$. Then 
\begin{equation}\label{ineq: uniform C_0,I}
    \left|e^{-\frac{t}{2}}\int_0^\infty (\rho(\lambda) c(\lambda) \lambda \tanh{(2\pi\lambda)})e^{i\lambda t}d\lambda\right|\leq \frac{e^{-t/2}}{t^N}C^{+}_{N, \rho}L
\end{equation}
A similar estimate can be acquired for the second term, 
\begin{equation}\label{ineq: uniform C_0-I}
  \left|e^{-\frac{t}{2}}\int_0^\infty (\rho(\lambda) c(-\lambda) \lambda \tanh{(2\pi\lambda)})e^{-i\lambda t}d\lambda\right|\leq \frac{e^{-t/2}}{t^N}C^{-}_{N, \rho}L
\end{equation}
and taking $C'_{N, \rho}=\min\{C_{N, \rho}^+, C_{N, \rho}^-\}$ we have for every $N$
\begin{equation*}
    |k_\rho(t)|\leq \frac{e^{-t/2}}{t^N}C'_{N, \rho}L+ \left|\int_0^\infty  \rho(\lambda) T_\lambda(t)\lambda \tanh{(2\pi\lambda)}d\lambda\right|.
\end{equation*}
For the final remainder term containing $T_\lambda(t)$, note that $|\Gamma_l(\lambda)|\leq d_1(1+l^{d_2})$ for any $l>0$. Hence 
\[
\sum_{l=1}^\infty |\Gamma_l(\lambda)|e^{-2lt}\leq\sum_{l=1}^\infty d_1(1+l^{d_2})e^{-2lt}= e^{-2t}\sum_{l=1}^\infty d_1(1+l^{d_2})e^{-2(l-1)t}\leq C_{d_1,d_2}e^{-2t},
\]
where the penultimate series converges because exponential decay beats any polynomial growth. Moreover $c(\lambda)$ is compact on $\supp(\rho)$ and let $|c(\pm \lambda)|\leq C_\rho$. Then $|T_\lambda(t)|\leq C_\rho C_{d_1, d_2} e^{-t/2} e^{-2t}$ on $\supp(\rho)$ and
\begin{equation}\label{ineq: uniform C_0I}
    \left|\int_0^\infty  \rho(\lambda) T_\lambda(t)\lambda \tanh{(2\pi\lambda)}d\lambda\right|\leq \int_0^L \|\rho\|_\infty C'_\rho e^{-5t/2}\lambda\tanh{(2\pi\lambda)d\lambda} \leq C''_{N, \rho} e^{-5t/2}.
\end{equation}
For large $t$, in particular $t\geq 1$, we have that $e^{-2t}t^N\leq C_N$ is bounded since exponential decay beats polynomial decay which implies $e^{-5t/2}\leq C_Ne^{-t/2}t^{-N}$, and considering $C_{N, \rho}=\min\{C'_{N, \rho}, C''_{N, \rho}\}$, we have 
\[
|k_\rho(t)|\leq C_{N, \rho}e^{-t/2}t^{-N}\quad \text{for every } t\geq 1.
\]
Moreover, for $t\geq 1$, $(2t)^{-N}\leq (1+t)^{-N}$ and 
\[
|k_\rho(t)|\leq C_{N, \rho}e^{-t/2}t^{-N}\leq C_{N, \rho}2^N e^{-t/2}(1+t)^{-N}:= C_{N, \rho} e^{-t/2}(1+t)^{-N}\quad \text{for every } t\geq 1.
\]

For $0\leq t\leq 1$, $\rho\in C_c^\infty(\RR^+)$ and $\varphi_\lambda(t)$ is uniformly bounded in $t$ from the fact that $\varphi_\lambda(t)\leq C_0 e^{-t/2}(1+t)$ (see Equation (3.1.5) from Chapter 5 of \cite{bray1994fourier}). So, 
\[
\sup_{t\in[0,1]} |k_\rho(t)|=:C'_\rho<\infty.
\]
Also, for $0\leq t\leq 1$, $e^{-t/2}(1+t)^{-N}\geq e^{-t/2}2^{-N}\geq e^{-1/2}2^{-N}$ and 
\[
|k_\rho(t)|\leq C_\rho'\leq C_\rho' 2^N e^{1/2}e^{-t/2}(1+t)^{-N}:=C_{N, \rho}e^{-t/2}(1+t)^{-N} \quad \text{for every } 0\leq t\leq 1,
\]
%So, only the first term in estimate of $|k_\rho(t)|$ survives for large $t$ %and assuming $t\geq 1$, we have $1+t\le 2t$, hence $t^{-M}\leq 2^M(1+t)^{-M}$ 
which completes the proof.

\end{proof}

\subsection*{Acknowledgement} The first and second named authors would like to thank Institut de Recherche Math\'{e}matique Avanc\'{e}e (IRMA), France, for providing ideal working conditions. Both authors would like to thank the European Union and the ERC (grant Acronym = InSpeGMos, Number = 101096550) for funding the work.

\bibliographystyle{amsalpha}
\bibliography{references}

\end{document}